\documentclass[11pt]{amsart}

\usepackage[a4paper]{geometry}

\usepackage{amsmath}
\usepackage{amssymb}
\usepackage{enumerate}
\usepackage{tikz}
\usepackage{float}
\usepackage{verbatim}
\usepackage{youngtab}
\usepackage{blkarray}

\allowdisplaybreaks

\title{Representations of the rational Cherednik algebra $H_{t,c}(S_3,\h)$ in positive characteristic}
\author{Martina Balagovi\'c and Jordan Barnes}

\usepackage{amsthm}
\newtheorem{theorem}{Theorem}[section]
\newtheorem{prop}[theorem]{Proposition}
\newtheorem{lemma}[theorem]{Lemma}
\newtheorem{cor}[theorem]{Corollary}
\newtheorem{conjecture}[theorem]{Conjecture}
\theoremstyle{definition}

\newtheorem{example}[theorem]{Example}
\newtheorem{remark}[theorem]{Remark}
\newtheorem{assumption}[theorem]{Assumption}

\newcommand{\x}[1]{\overline{x_{#1}}}
\newcommand{\s}[1]{\overline{\sigma_{#1}}}
\newcommand{\Triv}{\mathrm{Triv}}
\newcommand{\Sign}{\mathrm{Sign}}
\newcommand{\Stand}{\mathrm{Stand}}
\newcommand{\h}{\mathfrak{h}}
\newcommand{\id}{\mathrm{id}}

\def\cato/{Category $\mathcal{O}_{t,c}$}

\begin{document}
	\maketitle

\begin{abstract}
We study the rational Cherednik algebra $H_{t,c}(S_3,\h)$ of type $A_2$ in positive characteristic $p$, and its irreducible category $\mathcal{O}$ representations $L_{t,c}(\tau)$. 
For every possible value of $p,t,c$, and $\tau$ we calculate the Hilbert polynomial and the character of $L_{t,c}(\tau)$, and give explicit generators of the maximal proper graded submodule of the Verma module.  
\end{abstract}

\section{Introduction}	

\subsection{The algebra}
For a general overview of rational Cherednik algebras and their representation see \cite{EtMa11}; for specifics in positive characteristic see \cite{BaCh13a}. 

Let $\Bbbk$ be an algebraically closed field of positive characteristic $p$, and $\mathbb{F}_p=\left\{0,\ldots, p-1\right\}$ its finite subfield with $p$ elements. Consider the symmetric group $S_3$, let $V$ be its permutation representation with a basis $y_1,y_2,y_3$ and the action $g.y_i=y_{g(i)}$, $\h=\left\{\sum_{i=1}^3a_iy_i\mid \sum_{i=1}^3a_i=0 \right\}$ its two dimensional subrepresentation (the standard representation), $V^*$ the dual of $V$ with the dual basis $x_1,x_2,x_3$, and $\h^*$ the dual of $\h$, realised as a quotient of $V^*$ and spanned by $\overline{x_1}$, $\overline{x_2}$, $\overline{x_3}$ satisfying $\overline{x_1}+\overline{x_2}+\overline{x_3}=0$. Let $\left<\cdot,\cdot \right>$ denote the evaluation pairings $V^*\otimes V\to \Bbbk$ (satisfying $\left<x_i,y_j \right>=\delta_{ij}$) and $\h^*\otimes \h\to \Bbbk$. The group $S_3$ is a reflection group with $\h$ its reflection representation, so we can associate a rational Cherednik algebra to it. 

Let $t,c\in \Bbbk$ be parameters. Let $\Bbbk[S_3]$ be the group algebra of $S_3$ over $\Bbbk$ and $T(\h\oplus \h^*)$ the tensor algebra of $\h\oplus \h^*$. The rational Cherednik algebra $H_{t,c}(S_3,\h)$ is the quotient of $\Bbbk[S_3]\ltimes T(\h\oplus \h^*)$ by the relations (for all $y,y'\in \h$, $x,x'\in \h^*$)
\begin{equation}
[x,x']=0, \quad  [y,y']=0, \quad [y,x]=t\left<x,y \right>-\sum_{1\le i<j\le 3} c\left<x-(ij).x,y \right>(ij).
\end{equation}

The algebra $H_{t,c}(S_3,\h)$ is graded by $\deg x=1$, $\deg y=-1$ and $\deg g=0$. It has a triangular decomposition in the sense that the group algebra $\Bbbk[S_3]$ and the symmetric algebras $S\h$ and $S\h^*$ are subalgebras, and the multiplication map 
$$S\h \otimes \Bbbk[S_3] \otimes S\h^*\to H_{t,c}(S_3,\h)$$
is a vector space isomorphism. Let $S=\left\{(12),(23),(13)\right\}$
be the set of reflections in $S_3$. 

\subsection{The modules}

Let $\tau$ be an irreducible representation of $S_3$, and augment it to a $\Bbbk[S_3]\ltimes S\h$ module by setting the action of $\h$ on $\tau$ to be $0$. The Verma module $M_{t,c}(\tau)$ is the induced module
$$M_{t,c}(\tau)=H_{t,c}(S_3,\h)\otimes_{\Bbbk[S_3]\ltimes S\h}\tau.$$
As a vector space $M_{t,c}(\tau)\cong S\h^* \otimes \tau$, with the action of $g\in S_3$ on $S\h^* \otimes \tau$ being diagonal, the action of $x\in \h^*$ being by multiplication on the first tensor factor, and the action of $y\in \h$ given by the Dunkl operators
$$D_{y}=t\partial_y\otimes \mathrm{id} - \sum_{1\le i<j\le 3}c\left<x_i-x_j,y\right>\frac{1-(ij)}{x_i-x_j}\otimes (ij).$$ Setting $\deg (\tau)=0$ makes $M_{t,c}(\tau)$ into a graded representation of $H_{t,c}(S_3,\h)$, and this grading coincides with the natural grading on $S\h^*$.

In positive characteristic $M_{t,c}(\tau)$ has a maximal proper graded submodule $J_{t,c}(\tau)$, and the quotient $L_{t,c}(\tau)=M_{t,c}(\tau)/J_{t,c}(\tau)$ is irreducible and finite dimensional. Category $\mathcal{O}$ (denoted $\mathcal{O}_{t,c}$ or $\mathcal{O}_{t,c}(S_3,\h)$ when necessary) is the category of all finite dimensional graded $H_{t,c}(S_3,\h)$ modules. The morphisms are all $H_{t,c}(S_3,\h)$-maps which preserve the grading. For a graded module $M=\oplus_k M^k$ and an integer $l$, the shifted module $M[l]$ has the same underlying vector space and the algebra action and the grading $M[l]^k=M^{l+k}$. Irreducible modules in category $\mathcal{O}_{t,c}$ up to isomorphism are $L_{t,c}(\tau)[k]$ for all irreducible representations $\tau$ of $S_3$ and $k\in \mathbb{Z}$. 

\subsection{The main questions}

We wish to describe all $L_{t,c}(\tau)$ in three ways: 
\begin{enumerate}
\item By calculating their Hilbert polynomials. For a graded vector space $M=\oplus_{k}M^k$ with finite dimensional graded pieces $M^k$, its Hilbert series is $$h_M(z)=\sum_{k} \dim M^k z^k\in \mathbb{Z}[z,z^{-1}].$$
\item By calculating their characters. Let $K_{0}(S_3)$ be the Grothendieck group of $S_3$ representations over $\Bbbk$, and for an $S_3$ representation $M$ let $[M]$ be its class in $K_{0}(S_3)$. For a graded $S_3$ representation $M=\oplus_{k}M^k$ with finite dimensional graded pieces $M^k$, its character is $$\chi_M(z)=\sum_{k} [M^k] z^k \in K_{0}(S_3)[z,z^{-1}].$$ 
\item By specifying a finite set of generators of the maximal proper graded submodule $J_{t,c}(\tau)$. One can always find such a set by considering singular vectors. A vector $v$ in a graded $H_{t,c}(S_3,\h)$ module is called \emph{singular} if it is homogeneous, not of lowest degree, and satisfies $y.v=0$ for all $y\in \h$. Every such vector generates a proper graded subrepresentation. To find a set of generators of $J_{t,c}(\tau)$, we first look for singular vectors in $M_{t,c}(\tau)$, take a quotient of $M_{t,c}(\tau)$ by the submodule they generate, then look for singular vectors in this quotient, and iterate this process until we get an irreducible module. (It always terminates.)
\end{enumerate}

Notice that (1), (2) and (3) above are not independent, as each item carries strictly more information than the previous one. Knowing the generators of $J_{t,c}(\tau)$ usually makes it possible to find the characters $\chi_{L_{t,c}(\tau)}(z)$, and from there it is straightforward to compute the Hilbert polynomials $h_{L_{t,c}(\tau)}(z)$. 
 
Once the above are known for $L_{t,c}(\tau)$, they are easily computed for $L_{t,c}(\tau)[k]$ by shifting degrees; for example, the Hilbert series satisfy
$h_{M[l]}(z)=z^{-l}h_M(z).$

\subsection{The parameters}
In this paper we aim to answer (1), (2) and (3) for $L_{t,c}(\tau)$ for all $t,c,\tau$ and $p$ for which they are previously unknown. As the group $S_3$ has order $3!=6$, the representation theory of $S_3$ is semisimple with three irreducible representations $\tau=\Triv,\Sign,\Stand$ when $p>3$, not semisimple with two irreducible representations $\tau=\Triv,\Sign$ when $p=3$, and not semisimple with two irreducible representations $\tau=\Triv,\Stand$ when $p=2$ (see Section \ref{sect-IrrepsofS3}). The parameter $t\in \Bbbk$ can be assumed to be $0$ or $1$, leading to different behaviours (see Lemma \ref{lemma-rescaleparameter}). The parameter $c$ takes all values in $\Bbbk$. When $t=0$, the irreducible modules $L_{0,c}(\tau)$ behave in one way when $c\ne 0$ (generic values) and in a different way when $c=0$ (special value). When $t=1$, the irreducible modules $L_{1,c}(\tau)$ behave in one way when $c\notin \mathbb{F}_p$ (generic values) and in a different way when $c\in \mathbb{F}_p$ (special values) - see Lemma \ref{WhatIsGeneric} and the paragraph above it.

\subsection{History and related work}

Rational Cherednik algebras were introduced in 2002 by Pavel Etingof and Victor Ginzburg \cite{EtGi02}, building on earlier work of Ivan Cherednik \cite{Ch92,Ch95}. Since then, their representation theory has been studied extensively. A lot of work takes place over fields of characteristic 0, see \cite{DJO94, BEG03a, BEG03b, ChEt03, GGOR03, Go03, Du04, Gr08}. Less is known in positive characteristic, see \cite{La05, BFG06, BeMa13, BaCh13a, BaCh13b, DeSa14, DeSu16, CaKa21}.

Rational Cherednik algebras associated to the reflection group $S_n$ in characteristic $0$ have been studied in \cite{BFG06} and \cite{Go03}. The setup for the representation theory of rational Cherednik algebra in positive characteristics has been described in \cite{BaCh13a}. 

The preprint \cite{Li14} proves several results about rational Cherednik algebras associated to the symmetric and dihedral groups. Most importantly, it clarifies that for $S_n$, the parameter $c$ being generic means $c\notin \mathbb{F}_p$. The setup uses $V$ instead of $\h$ for the reflection representation, but the Hilbert polynomials, characters, and singular vectors can be translated, see Proposition \ref{prop-dictionaryVvsh}. This paper calculates the Hilbert polynomial of $L_{1,0}(\Triv)$, $L_{1,3^{-1}}(\Triv)$ when $p\ne 2$, describes the singular vectors in the case $p>3$, $t=1$, $\tau=\Triv$, $c\in \left\{0,1,2,\ldots, p-1\right\}$ with $0<c<p/2$ or $2p/3<c<p$, and conjectures the degrees of the singular vectors when $p/2<c<2p/3$. We prove their conjecture. The previously known cases are marked \cite{Li14} in the tables below. 

The paper \cite{DeSa14} considers rational Cherednik algebras associated to complex reflection groups $G(m,r,n)$ in positive characteristic and their representation theory for generic values of the parameter $c$. Its Section 4 provides characters for $G=G(m,1,n)$ in the non-modular case, which overlaps with our work when $m=1,n=3$, $p>3$. The characters calculated are for all $L_{t,c}(\tau)$ with $p>n$, $t=0,1$, $c\notin \mathbb{F}_p$, and all $\tau$. Those cases are marked \cite{DeSa14} in the tables below, and our contribution in those cases is the calculation of singular vectors. (These singular vectors are later used in the description of $L_{t,c}(\tau)$ with $p>3$, $t=0,1$, $c\in \mathbb{F}_p$, and all $\tau$, which is new). 

The paper \cite{DeSu16} examines the characteristic $p$ representation theory of rational Cherednik algebras associated to $S_n$ in the case that $p\mid n$, $\tau=\Triv$, $t=1$, providing singular vectors and the Hilbert polynomials. This is relevant for us in the case $p=3$, $\tau=\Triv$, $t=1$, marked \cite{DeSu16} in the tables below.

The paper \cite{CaKa21} studies the rational Cherednik algebra $H_{t,c}(S_n,\h)$ for generic $c$ and its representation theory in characteristic $p\mid {n-1}$. This only coincides with our work for $n=3$, $p=2$, $c\notin \mathbb{F}_2$, where they calculate the Hilbert polynomials of $L_{t,c}(\Triv)$. We mark this case by \cite{CaKa21} in the table below.

\subsection{Methods} We extensively use the standard computational methods: the grading Casimir element $\Omega$ (Lemma \ref{OmegaAction}), equivalences of categories for algebras obtained by rescaling the parameters (Lemmas \ref{lemma-rescaleparameter} and \ref{Sign}), the fact that for generic $c$ interesting things only happen in degrees divisible by $p$ (Proposition 3.4. in \cite{BaCh13a}), information about which $c$ are generic (Proposition \ref{WhatIsGeneric}),  and general information about some special values of $c$ (Lemmas \ref{t=c=0}, \ref{t=1,c=0}, \ref{Lian3.2.}). In early stages of the project we used the computational software Magma \cite{Magma} to generate conjectures, guess singular vectors, and check the Hilbert polynomials of $L_{t,c}(\tau)$ for specific $\tau,t,c,p$.

A new computational method we use is to write $\h^*$ in the Young basis (see Section \ref{sec-basisofh}), and use this and our understanding of symmetric polynomials to choose a good basis for $S\h^*$ and $M_{t,c}(\tau)\cong S\h^*\otimes \tau$ (see Theorems \ref{decmposeShp>3}, \ref{decmposeVermap>3}, \ref{decmposeShp=2},  and \ref{decmposeVermaStandp=2}). Such a basis is useful when looking for singular vectors in $M_{t,c}(\tau)$ and its quotients because: 
\begin{itemize}
\item[i)] It is compatible with the decomposition into $S_3$ subrepresentations. When looking for singular vectors, the Casimir element $\Omega$ (Lemma \ref{OmegaAction}) can tell us which isotypic components we need to examine in which degrees, and this basis makes it easy to restrict to a particular isotypic component in some degree. This reduces the size of the space we are examining for singular vectors. 
\item[ii)] It is compatible with the restriction to $S_2$, in the sense that every vector is either $S_2$-symmetric or skew symmetric. A consequence is that we can reduce the search for singular vectors (which are in the kernel of all Dunkl operators $D_y$) to the search for vectors in the kernel of $D_{y_1}$, see Lemma \ref{KerD1}. 
\item[iii)] Every vector in it is a product of an invariant from $(S\h^*)^{S_3}$ and one of the small number of fixed vectors of small degree. The difference parts of Dunkl operators are $(S\h^*)^{S_3}$-linear and the differential parts of Dunkl operators are compatible with products (Leibniz rule). This makes the Dunkl operator calculations manageable. 
\end{itemize}

Using these properties to reduce the number of computations we need to do, we calculate the values of Dunkl operators explicitly and look for singular vectors. This lets us describe the generators of $J_{t,c}(\tau)$, from where we compute the characters and the Hilbert polynomials of $L_{t,c}(\tau)$.

\subsection{Main results}	

For characteristic $p=2$, Theorem \ref{mainthmp=2} describes $L_{t,c}(\tau)$ for all $t,c,\tau$ by giving the generators of $J_{t,c}(\tau)$, characters, and Hilbert polynomials. We now list the Hilbert polynomials  along with the references in cases where these polynomials have been previously found:

\begin{center}
{\renewcommand*{\arraystretch}{1.5}\begin{tabular}{ c |c |c }
$p=2$ & $\Triv$ & $\Stand$ \\ \hline \hline 
 $t=0$, $c\ne 0$ & $1+2z+2z^2+z^3$  & $2+2z+2z^2$ \\
 & \scriptsize{\cite[Thm 2.11]{CaKa21}} & \\  \hline 
 $t=0$, $c=0$ & 1 & 2 \\ \hline 
 $t=1$, $c\notin \mathbb{F}_2$  
  & $(1+z)^2 (1+2z^2+2z^4+z^6)$ & $2(1+z)^2(1+z^2+z^4)$ \\ 
  & \scriptsize{\cite[Thm 3.17]{CaKa21}} & \\  \hline  
 $t=1$, $c=0$  & $\left( 1+z \right)^2$  & $2\left( 1+z \right)^2$   \\ 
 $t=1$, $c=1$  & $1$ & $\frac{2-z-z^3-z^5-z^7+2z^8}{(1-z)^2}$ \\
 & \scriptsize{\cite[Thm 3.2]{Li14}}
\end{tabular}
}
\end{center}

For characteristic $p=3$, Theorem \ref{mainthmp=3} describes $L_{t,c}(\tau)$ for all $t,c,\tau$ by giving the generators of $J_{t,c}(\tau)$, characters, and Hilbert polynomials. We now list the Hilbert polynomials, along with the references to previously known polynomials:

\begin{center}
{\renewcommand*{\arraystretch}{1.5}\begin{tabular}{ c |c |c }
$p=3$ & $\Triv$ & $\Sign$ \\ \hline \hline 
 $t=0$, all $c$ & $1$ & $1$ \\ \hline 
% $t=0$ $c$ special & - & - \\ \hline 
 $t=1$, all $c$  & $(1+z+z^2)^2$  & $(1+z+z^2)^2$   \\ 
  & \scriptsize{\cite{DeSu16}}   &  \\ %\hline 
% $t=1$, $c$ special  &  - &  -  
\end{tabular}}
\end{center}

For characteristic $p>3$ and generic $c$, \cite{DeSa14} calculate the characters and the Hilbert polynomials of all $L_{t,c}(\tau)$. We describe the generators of $J_{t,c}(\tau)$ (for general interest, and to use in describing $L_{t,c}(\tau)$ for special $c$), and along the way we reprove their results in this case. This can be found in Theorem \ref{thm-p>3generic}. 

For characteristic $p>3$ and special $c$, see Theorem \ref{thm-p>3special}. The case $c=0$ is standard and well known. For $c \in \mathbb{F}_p\setminus \left\{0\right\}$, \cite{Li14} describes $L_{1,c}(\Triv)$ except when $p/2<c<2p/3$, where they give conjectured degrees of the generators. We deal with this case, finding generators of $J_{1,c}(\Triv)$, characters, and Hilbert polynomials. This finishes the description of $L_{1,c}(\Triv)$. The description of $L_{1,c}(\Sign)$ follows by Lemma \ref{Sign}. 

The description of $L_{1,c}(\Stand)$ for $c\in \mathbb{F}_p$, $0<c<p/3$ is done by finding the generators of $J_{1,c}(\Stand)$, characters, and Hilbert polynomials. We are only able to prove this using a technical Assumption \ref{assum} on the field $\mathbb{F}_p$. This assumption holds for all primes $p < 2023$, and we conjecture it holds for all primes. The description of $L_{1,c}(\Stand)$ for $c\in \mathbb{F}_p$, $2p/3<c<p$ follows by Lemma \ref{Sign}.

The description of $L_{1,c}(\Stand)$ for $c\in \mathbb{F}_p$, $p/3<c<p/2$, is conjectural, see Conjecture \ref{finalconj}, and the description of $L_{1,c}(\Stand)$ for $c\in \mathbb{F}_p$, $p/2<c<2p/3$ would follow from there. We checked these conjectures in Magma \cite{Magma} for primes up to $20$. 

We now list the Hilbert polynomials.
\begin{center}
{\renewcommand*{\arraystretch}{1.5}\begin{tabular}{ c |c |c |c }
$p>3$ & Triv & Sign & Stand \\ \hline \hline 
 $t=0$, $c$ generic & $1+2z+2z^2+z^3$ & $1+2z+2z^2+z^3$ & $2+2z+2z^2$ \\ 
   & \scriptsize{\cite[Prop 4.1]{DeSa14}} &  \scriptsize{\cite[Prop 4.1]{DeSa14}}  & \scriptsize{\cite[Prop 4.1]{DeSa14}} \\ \hline 
 $t=0$, $c=0$ & 1 & 1 & 2 \\ \hline 
 $t=1$, & $\frac{(1-z^{2p})(1-z^{3p})}{(1-z)^2}$ & $\frac{(1-z^{2p})(1-z^{3p})}{(1-z)^2}$ & $\frac{2(1-z^{p})(1-z^{3p})}{(1-z)^2}$ \\ 
 $c\notin \mathbb{F}_p$ &    \scriptsize{\cite[Prop 4.2]{DeSa14}}  & \scriptsize{\cite[Prop 4.2]{DeSa14}} & \scriptsize{\cite[Prop 4.2]{DeSa14}}  \\ \hline 
 $t=1$, $c\in \mathbb{F}_p$  & &  &   \\
 $c=0$ & $\left(\frac{1-z^{p}}{1-z}\right)^2$ & $\left(\frac{1-z^{p}}{1-z}\right)^2$ & $2\left(\frac{1-z^{p}}{1-z}\right)^2$ \\
 $0<c<p/3$ & $\left(\frac{1-z^{3c+p}}{1-z}\right)^2$ & $\left(\frac{1-z^{p-3c}}{1-z}\right)^2$  & $\frac{1-z^{p-3c}-2z^p -z^{p+3c}+2z^{2p}}{(1-z)^2}$ \\
& \scriptsize{\cite[Thm 3.3]{Li14}} & & \\ 
 $p/3<c<p/2$ & $\left(\frac{1-z^{3c-p}}{1-z}\right)^2$  & $\frac{(1-z^{3p-6c})(1-z^p)}{(1-z)^2}$  & $\frac{1-z^{-p+3c}-z^{3p-3c}-z^{p+3c}-z^{5p-3c}+2z^{4p}} 
 {(1-z)^2}$\\
 & \scriptsize{\cite[Thm 3.3]{Li14}} & & \scriptsize{(Conjecture)} \\  
 $p/2<c<2p/3$ & $\frac{(1-z^{6c-3p})(1-z^p)}{(1-z)^2}$  & $\left(\frac{1-z^{2p-3c}}{1-z}\right)^2$ & $\frac{1-z^{2p-3c}-z^{3c}-z^{4p-3c}-z^{2p+3c}+2z^{4p}}{(1-z)^2}$\\
&  & & \scriptsize{(Conjecture)} \\
 $2p/3<c<p$ & $\left(\frac{1-z^{3c-2p}}{1-z}\right)^2$  & $\left(\frac{1-z^{4p-3c}}{1-z}\right)^2$  & $\frac{1-z^{3c-2p}-2z^p -z^{4p-3c}+2z^{2p}}{(1-z)^2}$ \\
& \scriptsize{\cite[Thm 3.3]{Li14}} & & \\ 
 \end{tabular}}
\end{center}

\subsection{Further questions} Obvious follow up questions are to prove Conjecture \ref{finalconj}, and show that Assumption \ref{assum} holds for all primes $p$. It would be interesting to see if similar methods can be used beyond $S_3$, most interestingly for special $c$ and representations other than the trivial. This could help gather evidence for a conjecture by Etingof and Reins that the representation theory of $H_{t,c}(S_n)$ in characteristic $p$ in a certain sense only depends on $n \bmod p$.

\subsection{Roadmap of the paper}	
In Section \ref{sect-prelim}	we list all the preliminaries, including both the standard tools in Category $\mathcal{O}$ and the new results about the bases of $\h^*$, $S\h^*$ and 	$M_{t,c}(\tau)$ which we will use. In Section \ref{sectionp=2} we completely describe all $L_{t,c}(\tau)$ when $p=2$, and in Section \ref{sectionp=3} all $L_{t,c}(\tau)$ when $p=3$. Section \ref{sect-auxiliary} contains many explicit auxiliary computations we will need to fully use the basis of $M_{t,c}(\tau)$ in case $p>3$. We tackle the cases $p>3$ in Section \ref{sect-p>3-generic} (generic $c$) and Section \ref{sect-p>3-special} (special $c$).

\subsection*{Acknowledgements}	
We thank Pavel Etingof, Gwyn Ballamy and Stefan Kolb for useful comments, and Alina Vdovina for access to MAGMA. This work was done as a part of the PhD thesis of J.B. We thank Alina Vdovina, Peter Jorgensen, and David Seifert for support during this time. The thesis was funded by EPSRC DTP EP/N509528/1.

\section{Preliminaries}	\label{sect-prelim}

\subsection{Irreducible representations of $S_3$ in positive characteristic}\label{sect-IrrepsofS3}

We will use $s_1=(12)$ and $s_2=(23)$ to denote the generators of the symmetric group $S_3$. Over algebraically closed fields of characteristic $0$ the group $S_3$ has three finite dimensional irreducible representations up to isomorphism: the trivial, the sign and the standard representation. The notation we use for them and their character table is given below. 
\begin{center}
{\renewcommand*{\arraystretch}{1.5}\begin{tabular}{c|ccc}
$S_3$ & e & (12) & (123) \\
& 1 & 3 & 2 \\ \hline
$\Triv$ & 1 & 1 & 1 \\
$\Sign$ & 1 & -1 & 1 \\
$\Stand$ & 2 & 0 & -1
\end{tabular}}
\end{center}

Considering its Brauer characters over an algebraically closed field $\Bbbk$ of characteristic $p>0$ we immediately see that:
\begin{itemize}
\item if $p=2$, $\Triv=\Sign$, and irreducible representations are $\Triv$ and $\Stand$;
\item if $p=3$, $\Stand$ is not irreducible, and irreducible representations are $\Triv$ and $\Sign$;
\item if $p>3$, irreducible representations are $\Triv$, $\Sign$ and $\Stand$, and the category of representations of $S_3$ over $\Bbbk$ is semisimple. 
\end{itemize}

The following calculation works for arbitrary characteristic, though in some characteristics (namely, $2$ and $3$) the resulting scalars coincide. 

\begin{lemma}\label{OmegaAction}
Consider the Casimir element $\Omega$ in $H_{1,c}(S_3,\h)$, given by 
$$\Omega=\sum_{i=1}^2\x{i}(y_i-y_3)+c\sum_{1\le i<j\le 3}(1-(ij)).$$
Assume $M$ is a graded $H_{1,c}(S_3,\h)$-module and $\tau\subseteq M$ an irreducible $S_3$-subrepresentation contained in the kernel of all Dunkl operators. Then $\Omega$  acts on $\tau$ by a constant as follows:
$$\Omega|_\Triv=0\cdot \mathrm{id}, \quad 
\Omega|_\Sign=6c \cdot \mathrm{id}, \quad 
\Omega|_\Stand=3c \cdot \mathrm{id}.$$
\end{lemma}
\begin{proof}
Standard direct computation; the term $\sum_{i=1}^2\x{i}(y_i-y_3)$ acts by $0$ on any vector in the kernel of Dunkl operators, and the term $\sum_{1\le i<j\le 3}(1-(ij))$ is central in $\Bbbk[S_3]$ so acts on any irreducible $\tau$ by a constant which can be calculated from the characters. 
\end{proof}

\subsection{Reflection representations of $S_n$ in positive characteristic}\label{sect-reflectionrepn}
The group $S_n$ has a natural permutation representation $V$ with a basis $y_1,\ldots ,y_n$ and the action given on this basis as $g.y_i=y_{g(i)}$. It is self dual, with $x_1,\ldots ,x_n$ the dual basis of $V^*$ and the isomorphism $V\to V^*$ given by $y_i\mapsto x_i$. 

Let $Y=y_1+\ldots+y_n$ and $X=x_1+\ldots+x_n$. They span $\Bbbk Y\subseteq V$ and $\Bbbk X \subseteq V^*$, which are one dimensional subrepresentations isomorphic to the trivial representation $\Triv$. Another subrepresentation of $V$ is $$\h=\left\{\sum_i a_iy_i \mid \sum a_i =0\right\}.$$ It is called the standard representation and we denote it $\Stand$. 

When the characteristic $p$ of the field $\Bbbk$ does not divide $n$, $\Bbbk Y$ and $\h$ are in direct sum, $V=\Bbbk Y\oplus \h$, and $\h$ is irreducible. The dual is then also a direct sum, $V^*=\Bbbk X\oplus \h^*$. Here $$\Bbbk X\cong (\Bbbk Y)^*=\left\{x\in V^*\mid \left<x,y\right>=0 \quad \forall y\in \h \right\}$$ with $\left<X,Y\right>=n\ne 0$, and $\h^*$ can be realised as 
$$V^*/\Bbbk X= \mathrm{span}\left\{ \x{i}=x_i+\Bbbk X \right\}$$
or as 
$$\left\{x \in V^* \mid \left<x,Y \right>=0 \right\}=\mathrm{span}\left\{x_i-x_j\right\}\subseteq V^*.$$
The isomorphism between these two realisations is given by the splitting
$$\pi: V^*/\Bbbk X \hookrightarrow V^* \qquad \pi(\x{i})=x_i-\frac{1}{n}X.$$

When the characteristic $p$ of the field $\Bbbk$ divides $n$, $Y$ lies in $\h$ so $\Bbbk$ and $\h$ are not in direct sum and $\h$ is not irreducible. In that case there is a non-split exact sequence
$$0\to \h\to V \to V/\h \to 0,$$
with $V/\h\cong \Triv$. Dualising we get another non-split exact sequence
$$0\to (V/\h)^* \to  V \to \h^* \to 0,$$
with $$(V/\h)^*=\left\{x \in V^* \mid \left<x,\h \right>=0 \right\}=\Bbbk X,$$
so in this case $$\h^*=V^*/\Bbbk X=\mathrm{span}\left\{ \x{i}=x_i+\Bbbk X \right\}.$$

In any characteristic both the representation $V$ and its subrepresentation $\h$ are reflection representations, so one can consider the rational Cherednik algebras $H_{t,c}(S_n,V)$ and $H_{t,c}(S_n,\h)$. The two choices lead to slightly different conventions and computations in the literature, though the two algebras and their representations are very closely related. For example, we have the following lemma. 

\begin{prop}\label{prop-dictionaryVvsh}
Let $\Bbbk$ be an algebraically closed field whose finite characteristic $p$ does not divide $n$, $\tau$ an irreducible representation of $S_n$,  
$L_{t,c}(S_n,V, \tau)$ the irreducible representation with lowest weight $\tau$ for the rational Cherednik algebra $H_{t,c}(S_n,V)$ over $\Bbbk$ generated by $V,V^*,S_n$, and $L_{t,c}(S_n,\h, \tau)$ the irreducible representation with lowest weight $\tau$ for the rational Cherednik algebra $H_{t,c}(S_n,\h)$ over $\Bbbk$ generated by $\h,\h^*,S_n$. Their characters are related as follows: 
\begin{itemize}
\item if $t=0$, 
$$\chi_{L_{0,c}(S_n,V, \tau)}(z)=\chi_{L_{0,c}(S_n,\h, \tau)}(z)$$
\item if $t= 1$, 
$$\chi_{L_{1,c}(S_n,V, \tau)}(z)=\chi_{L_{1,c}(S_n,\h, \tau)}(z)\cdot \left(\frac{1-z^p}{1-z}\right).$$
\end{itemize}
\end{prop}
\begin{proof}
The representation $L_{t,c}(S_n,\h, \tau)$ can be written as a quotient of the Verma module $M_{t,c}(S_n,\h, \tau)$ by a submodule $\left<v_1,v_2,\ldots, v_k\right>$ generated by some homogeneous vectors $v_i$ of degrees $d_i\in \mathbb{N}$. We can assume without loss of generality that $0<d_1\le d_2\le \ldots \le d_k$, and for every $i$ and all $y\in \h$ we have $D_{y}(v_i)\in \left<v_1,\ldots, v_{i-1} \right>$. 

As $p$ does not divide $n$, we have $V^*\cong \h^*\oplus \Bbbk X$. Using the inclusion $S\h^*\subseteq SV^*$, we can consider $v_i \in M_{t,c}(S_n,\h, \tau)\cong S\h^* \otimes \tau$ as elements of $M_{t,c}(S_n,V,\tau)\subseteq S V^*\otimes \tau$. We claim that
\begin{align*}
L_{0,c}(S_n,V, \tau)&= M_{0,c}(S_n,V, \tau)/\left<X\otimes \tau,v_1,v_2,\ldots, v_k\right>,\\
L_{1,c}(S_n,V, \tau)&= M_{1,c}(S_n,V, \tau)/\left<X^p \otimes \tau,v_1,v_2,\ldots, v_k\right>.
\end{align*}

First, let us show that for every $i$ and all $y\in V$ we have $D_{y}(v_i)\in \left<v_1,\ldots, v_{i-1} \right>$. If $y\in \h\subseteq V$ this is true by assumption, so it remains to check it for $y=Y$. This follows from the fact that $[Y,x]=0$ for all $x\in \h^*$ and that $v_i\in S\h^*\otimes \tau$. 

Next, note that when $t=0$ we have $[y,X]=0$ for all $y\in V$, and when $t=1$ we have $[y,X^p]=0$ for all $y\in V$. As a consequence, for $t=0$ the set $X\otimes \tau$ consists of singular vectors in degree $1$, and when $t=1$ the set $X^p\otimes \tau$ consists of singular vectors in degree $p$. 

Let $J_0=\left<X\otimes \tau,v_1,v_2,\ldots, v_k\right>$ for $t=0$ and $J_1=\left<X^p \otimes \tau,v_1,v_2,\ldots, v_k\right>$ for $t=1$. The above shows that $J_t$ is a proper graded submodule of $M_{t,c}(S_n,V,\tau)$, and so $L_{t,c}(S_n,V, \tau)$ is a quotient of $M_{t,c}(S_n,V, \tau)/J_t$. 

Finally, let us show that the module $M_{t,c}(S_n,V, \tau)/J_t$ is irreducible. If it is not, then there is a homogeneous vector $v\in M^m_{t,c}(S_n,V, \tau)$, which is not in $J_t$, and is such that for all $y\in V$ we have $D_y(v)\in J_t$. Write $v\in SV^*\otimes \tau\cong \Bbbk[X]\otimes S\h^* \otimes \tau$ as 
$v=\sum_{i=0}^m X^{m-i} f_{i}$
for $f_i \in S^i\h^*\otimes \tau$. The condition $D_{y_j}(v)\in J_t$ for all $j=1,\ldots, n$ then becomes
\begin{align*}
D_{y_j}(v)%&=\sum_{i=0}^k y_aX^i f_{k-i}\\
&=\sum_{i=0}^m [y_j,X^{m-i}] f_{i}+ X^{m-i}D_{y_j}(f_{i})\\
&=\sum_{i=0}^m (m-i) X^{m-i-1} f_{i}+ X^{m-i}D_{y_j}(f_{i})\\
&=\sum_{i=1}^{m}\left(  (m-i+1) X^{m-i} f_{i-1}+ X^{m-i}D_{y_j}(f_{i})\right) \in J_t.
\end{align*}
This can be rewritten as 
$$X^{m-i} D_{y_j}(f_{i})=-(m-i+1) X^{m-i} f_{i-1} \pmod {J_t} \quad \textrm{ for all } j=1,\ldots, n,$$
so for any $j,k\in \left\{1,\ldots n\right\}$ we have 
$$X^{m-i} D_{y_j-y_k}(f_{i})=X^{m-i} D_{y_j}(f_{i})-X^{m-i} D_{y_k}(f_{i}) \in J_t$$
and consequently
$X^{m-i} D_{y}(f_{i})\in J_t$ for all $y \in \h.$

We now consider two cases: 
\begin{itemize}
\item $t=0$. As $D_y(f_m)=0$ for all $y \in \h$
and $M_{t,c}(S_n,\h, \tau) / \left<v_1,v_2,\ldots, v_k\right> = L_{0,c}(S_n,\h,\tau)$ is irreducible, we conclude that 
$$f_m\in \left< v_1,\ldots, v_n\right>.$$
But then 
$$v=f_m+X\sum_{i=0}^{m-1}X^{m-i-1}f_i\in \left< X, v_1,\ldots, v_n\right>=J_0,$$
which contradicts the assumption that $v\notin J_0.$

\item $t=1$. 
As for all $ y \in \h$, $i=0, \ldots, m$ we have 
$$X^{m-i} D_{y}(f_{i})\in \left< X^p\otimes \tau, v_1,\ldots, v_n\right>$$
we can in particular conclude that for all $i$ such that $m-i<p$ and all $ y \in \h$ we have 
$$D_{y}(f_{i})\in \left< v_1,\ldots, v_n\right>.$$
The module $M_{1,c}(S_n,\h, \tau) / \left<v_1,v_2,\ldots, v_k\right> = L_{1,c}(S_n,\h,\tau)$ is irreducible, so 
$$f_i\in \left< v_1,\ldots, v_n\right>, \quad  \textrm{for all } m-p<i\le m.$$
But then 
$$v=X^p\sum_{i=0}^{m-p}X^{m-p-i}f_i+\sum_{i=m-p+1}^{m}X^{m-i}f_i\in \left< X^p\otimes \tau, v_1,\ldots, v_n\right>=J_1,$$
which contradicts the assumption that $v\notin J_1.$

\end{itemize}

\end{proof}

We will be working with $H_{t,c}(S_n,\h)$ but will need to cite several papers which work with $H_{t,c}(S_n,V)$ and we will use Proposition \ref{prop-dictionaryVvsh} to translate between the two.

\subsection{Choices of bases in $\h^*$}\label{sec-basisofh}
In order to do explicit computations with Verma modules, we will need to fix bases for $S\h^*$, all irreducible representations $\tau$ of $S_3$, and the Verma modules $M_{t,c}(\tau)$. Let $n=3$ for this section. 

By Section \ref{sect-reflectionrepn}, in any characteristic $\h^*$ has a basis $\left\{\x{1},\x{2}\right\}$, and this is the basis we will use for $p=2,3$. 

In characteristic $p>3$, it is convenient to use a basis which is well behaved when restricting from $S_3$ to $S_2=\left\{e,s_1\right\}\subseteq S_3$, in order to use the additional symmetries to reduce the number of computations (see Lemma \ref{KerD1}). We start by identifying 
$$\h^*\cong \left\{\sum_i a_ix_i\mid \sum_i a_i=0\right\}\subseteq V^*$$
via $\pi$, and then define the \emph{rescaled Young basis} of $\h^*$ as
\begin{align*}
b_+&=x_1+x_2-2x_3\\
b_-&=x_1-x_2.
\end{align*}
This basis satisfies: 
\begin{align*}
s_1.b_+=b_+ \quad  
s_1.b_-=-b_- \quad  
b_-=\frac{2}{3}\left(s_2+\frac{1}{2}\right)b_+.
\end{align*}

Conversely, expressing $\pi(\x{i})$ in this basis we get
\begin{align*}
\pi(\x{1})&=x_1-\frac{X}{3}=\frac{b_++3b_-}{6}\\
\pi(\x{2})&=x_2-\frac{X}{3}=\frac{b_+-3b_-}{6}\\
\pi(\x{3})&=x_3-\frac{X}{3}=\frac{-b_+}{3}.
\end{align*}

\subsection{Symmetric polynomials}\label{sect-sympol}

Singular vectors are in general difficult to describe, but some can always be constructed from symmetric polynomials in $\h^*$. Specifically, let $(S\h^*)^{S_3}_+$ be the set of symmetric polynomials of strictly positive degree. For all $f\in (S\h^*)^{S_3}_+$ and $v\in \tau$, the vector $f\otimes v$ is singular in $M_{0,t}(\tau)$ and $f^p\otimes v$ is singular in $M_{1,t}(\tau)$. Consequently they generate proper submodules, and we can define quotient modules
\begin{align*}
N_{0,c}(\tau)&=M_{0,c}(\tau)/S\h^*(S\h^*)^{S_3}_+\otimes \tau,\\
N_{1,c}(\tau)&=M_{1,c}(\tau)/S\h^*((S\h^*)^{S_3}_+)^p\otimes \tau.
\end{align*} 
These are called \emph{baby Verma modules}, are finite dimensional, and $L_{t,c}(\tau)$ are quotients of $N_{t,c}(\tau)$. We specify their characters for the cases we will need in Lemmas \ref{characterShp>3} and \ref{characterShp=2}. 

In order to describe baby Verma modules, but also in order to decompose $M^k_{t,c}(\tau)\cong S^k\h^*\otimes \tau$ as an $S_3$ representation (which can help us limit the space in which we are looking for singular vectors), we will need to understand the $S_3$ invariants in $SV^*$ and both incarnations of $S\h^*$. We do this for general $n$ because it is just as easy as for $n=3$. 

First, by the fundamental theorem of symmetric polynomials, for a field $\Bbbk$ of any characteristic, $(SV^*)^{S_n}$ is a polynomial algebra in $n$ variables of degrees $1,2, \ldots n$, and those variables can be chosen to be the elementary symmetric polynomials
$$\sigma^V_k=\sum_{1\le i_1<i_2<\ldots <i_k\le n}x_{i_1}x_{i_2}\ldots x_{i_k}.$$
For example, for $n=3$: 
\begin{align*}
\sigma^V_1&=X=x_1+x_2+x_3\\
\sigma^V_2&=x_1x_2+x_1x_3+x_2x_3\\
\sigma^V_3&=x_1x_2x_3.
\end{align*}
Over some fields there are other good choices of symmetric polynomials as generators of the invariant ring, but this choice works over any field.

The representations $V^*$ and $\h^*$ are faithful representations of $S_n$ over a field $\Bbbk$ of any characteristic $p$ unless $(n,p)=(2,2)$, and there is always a surjection $V^*\to \h^*$. So, by \cite{Na79} Prop 4.1, for all $(n,p)\ne (2,2)$, $(S\h^*)^{S_n}$ is also a polynomial algebra. (Note this is not true for $(S\h)^{S_n}$ in general.) Let $\s{i}$ be the image of $\sigma^V_i$ under the quotient map $SV^*\to S\h^*$. For $n=3$ we have:
\begin{align*}
\s{1}&=0\\
\s{2}&=\x{1}\x{2}+\x{1}\x{3}+\x{2}\x{3}=-(\x{1}^2+\x{1}\x{2}+\x{2}^2)\\
\s{3}&=\x{1}\x{2}\x{3}=-(\x{1}^2\x{2}+\x{1}\x{2}^2).
\end{align*}
The polynomials $\s{2},\s{3}, \ldots, \s{n}$ are nonzero, algebraically independent, there are $n-1=\dim \h^*$ of them, and the product of their degrees is $\prod_{i=2}^ni=n!=|S_n|$, so by \cite{Ke96} Prop 16 we can conclude that for all $(n,p)\ne (2,2)$
$$(S\h^*)^{S_n}=\Bbbk[\s{2},\s{3}, \ldots , \s{n}].$$

As explained in section \ref{sect-reflectionrepn}, in good characteristic (when $p$ does not divide $n$), $\h^*$ can be alternatively realised as a subrepresentation rather than a quotient of $V^*$ via the map $\pi$. Call the images of $\s{i}$ in this realisation 
$\sigma_i=\pi(\s{i})\in (S\h^*)^{S_n}.$

Finally, for $n=3$ and $p>3$, let us calculate $\sigma_i$ explicitly in terms of the basis $b_-,b_+$ from Section \ref{sec-basisofh}: 
\begin{align}
\sigma_2&=\pi(\s{2})=-\pi(\x{1}^2+\x{1}\x{2}+\x{2}^2)=\frac{-1}{12}(b_+^2+3b_-^2) \notag\\
\sigma_3&=\pi(\s{3})=-\pi(\x{1}^2\x{2}+\x{1}\x{2}^2)=\frac{-1}{2^23^3}(b_+^3-9b_+b_-^2). \label{sigmaexplicit}
\end{align}

Closely related to the symmetric polynomial are the antisymmetric polynomials, $f\in S\h^*$ which transform as the $\Sign$ representation. By considering each transposition in turn and writing out $(ij).f=-f$, one sees that $f$ is divisible by $x_i-x_j$ for all $i\ne j$, and consequently that $f$ is divisible by the Vandermonde polynomial $q=\prod_{i<j}(x_i-x_j)$. It then quickly follows that any antisymmetric polynomial $f$ is of the form $f=q\cdot f'$, where $f'$ is a symmetric polynomial. 

For $n=3$ we can calculate the Vandermonde polynomial $q$ in terms of the basis $b_-,b_+$ from Section \ref{sec-basisofh} as 
$$q=(x_1-x_2)(x_2-x_3)(x_1-x_3)=b_+^2b_--b_-^3.$$

\subsection{Bases of Verma modules}

We will need the following combinatorial lemma. 

\begin{lemma}\label{numberinbasis}
For any $k\in \mathbb{N}_0$ the number of non-negative integral solutions $(a,b)\in \mathbb{N}_0^2$ of the equation $2a+3b=k$
equals
$$\begin{cases}\lfloor \frac{k}{6} \rfloor + 1, & k \textrm{ even}\\
\lfloor \frac{k-3}{6} \rfloor + 1, & k \textrm{ odd.} \end{cases}$$
\end{lemma}
\begin{proof}
We will parametrise the solutions for later use. 

First assume $k=2k'$ is even. The equation $2a+3b=k$ then becomes 
$$2a+3b=2k'$$
so we see $b$ is even. Write $b=2j$ for some $j\ge 0$. The equation now becomes
$$a=k'-3j$$
which gives another condition $j\le \frac{k'}{3}$. 
So, the set of solutions can be parametrised as
$$\left\{\left(\frac{k}{2}-3j,2j\right)\mid 0\le j\le \frac{k}{6}\right\}$$ and their total number is $\lfloor \frac{k}{6} \rfloor + 1$.

Now assume $k=2k'+1$ is odd. The equation $2a+3b=k$ becomes 
$$2a+3b=2k'+1$$
so $b$ is odd. Writing $b=2j+1$ for some $j\ge 0$ we get
$$a=k'-1-3j$$
which gives another condition $j\le \frac{k'-1}{3}$. 
So, the set of solutions can be parametrised as
$$\left\{\left(\frac{k}{2}-1-3j,2j+1\right)\mid 0\le j\le \frac{k-3}{6}\right\}$$ and their total number is $\lfloor \frac{k-3}{6} \rfloor + 1$.

\end{proof}

\begin{theorem}\label{decmposeShp>3}
In characteristic  $p>3$, $S^k\h^*$ is a direct sum of the following irreducible $S_3$ representations: 
\begin{itemize}
\item for every $a,b\in \mathbb{N}_0$ satisfying $2a+3b=k$, a subrepresentation isomorphic to $\Triv$ with a basis $$\left\{ \sigma_2^a\sigma_3^b \right\};$$
\item for every $a,b \in \mathbb{N}_0$ satisfying $2a+3b+3=k$,  a subrepresentation isomorphic to $\Sign$ with a basis
$$\left\{\sigma_2^a\sigma_3^b q \right\};$$
\item for every $a,b \in \mathbb{N}_0$ satisfying $2a+3b=k-1$, 
a subrepresentation isomorphic to $\Stand$ with a basis $$\left\{ \sigma_2^a\sigma_3^b b_+, \sigma_2^a\sigma_3^b b_- \right\};$$
\item for every $a,b \in \mathbb{N}_0$ satisfying $2a+3b=k-2$, 
a subrepresentation isomorphic to $\Stand$ with a basis $$\left\{ \sigma_2^a\sigma_3^b (-b_+^2+3b_-^2), \sigma_2^a\sigma_3^b 2b_+b_- \right\}.$$
\end{itemize}
\end{theorem}

\begin{proof}
The $S_3$ representation $S^k\h^*$ is a direct sum of its isotypic components, so it is enough to show that the above vectors are in the correct isotypic components and form a basis.

It is explained at the start of Section \ref{sect-sympol} that the $\Triv$ isotypic component is the polynomial algebra in $\sigma_2$ and $\sigma_3$ and that the $\Triv$ isotypic component consists of products of invariants and the Vandermonde polynomial $q$. 

One can check directly that for every $a,b$ the span of $\sigma_2^a\sigma_3^b b_+, \sigma_2^a\sigma_3^b b_-$ is an $S_3$ subrepresentation isomorphic to $\Stand$ with the isomorphism given by $b_\pm\mapsto \sigma_2^a\sigma_3^b b_\pm$, and that for every choice of $a,b$ the span of $\sigma_2^a\sigma_3^b (-b_+^2+3b_-^2), \sigma_2^a\sigma_3^b 2b_+b_-$ is an $S_3$ subrepresentation isomorphic to $\Stand$ with the isomorphism given by $b_+\mapsto \sigma_2^a\sigma_3^b (-b_+^2+3b_-^2)$, $b_-\mapsto \sigma_2^a\sigma_3^b 2b_+b_-.$

Next, let us show that the set $$\left\{ \sigma_2^a\sigma_3^b b_+, \sigma_2^a\sigma_3^b b_- \mid 2a+3b+1=k \right\} \cup \left\{ \sigma_2^a\sigma_3^b (-b_+^2+3b_-^2), \sigma_2^a\sigma_3^b 2b_+b_- \mid 2a+3b+2=k \right\}$$
is linearly independent. Assume a nontrivial linear combination of these vectors is zero. Gathering the terms, this can be written as 
$$f_1\cdot b_++f_2\cdot b_-+f_3\cdot (-b_+^2+3b_-^2)+f_4\cdot 2b_+b_-=0,$$
where $f_1,f_2$ are symmetric polynomials of degree $k-1$, $f_3,f_4$ are symmetric polynomials of degree $k-2$, and $f_1,f_2,f_3,f_4$ are not all zero.
Recall that $b_+$ and $(-b_+^2+3b_-^2)$ are fixed by $s_1$, while $b_-$ and $2b_+b_-$ are multiplied by $-1$ by $s_1$. If $f_2,f_4$ are not both zero, let us apply the projection $\frac{1}{2}(\mathrm{id}-s_1)$ to get 
$$f_2\cdot b_-+f_4\cdot 2b_+b_-=0.$$
If $f_2,f_4$ are both zero, the expression is 
$$f_1\cdot b_++f_3\cdot (-b_+^2+3b_-^2)=0,$$
which after applying $\frac{2}{3}\left(s_2+\frac{1}{2}\mathrm{id} \right)$ becomes
$$f_1\cdot b_-+f_3\cdot 2b_+b_-=0.$$
So, in either case we can assume we have a linear combination of the form $$f_2\cdot b_-+f_4\cdot 2b_+b_-=0$$ with $f_2,f_4$ symmetric and not both zero. Dividing by $b_-$ we get
$$f_2=-f_4\cdot 2b_+,$$
which is an equality between a symmetric polynomial and a polynomial which is not symmetric, so it only holds if $f_2,f_4$ are both zero. This is a contradiction, and we conclude that the set $$\left\{ \sigma_2^a\sigma_3^b b_+, \sigma_2^a\sigma_3^b b_- \mid 2a+3b+1=k \right\} \cup \left\{ \sigma_2^a\sigma_3^b (-b_+^2+3b_-^2), \sigma_2^a\sigma_3^b 2b_+b_- \mid 2a+3b+2=k \right\}$$
is linearly independent. 

We have now seen that the vectors from the statement of the theorem lie in the correct isotypic components and are linearly independent. Their number is equal to  
\begin{align*}
\textrm{number of vectors }&=|\left\{(a,b)\in \mathbb{N}_0^2\mid 2a+3b=k\right\}|+|\left\{(a,b)\in \mathbb{N}_0^2 \mid 2a+3b+3=k\right\}|+\\
&+2 |\left\{(a,b)\in \mathbb{N}_0^2\mid 2a+3b+1=k\right\}|+2|\left\{(a,b)\in \mathbb{N}_0^2 \mid 2a+3b+2=k\right\}|.
\end{align*}
By Lemma \ref{numberinbasis}, if $k$ is even, this is equal to 
\begin{align*}\textrm{number of vectors }&=\lfloor \frac{k}{6} \rfloor + 1+\lfloor \frac{k-6}{6} \rfloor + 1+2(\lfloor \frac{k-4}{6} \rfloor + 1)+2(\lfloor \frac{k-2}{6} \rfloor + 1)\\
&=5+2\left( \lfloor\frac{k}{6}\rfloor+\lfloor\frac{k-2}{6}\rfloor+\lfloor\frac{k-4}{6}\rfloor \right)\\
&=5+2\left( \frac{k}{2}-2\right)=k+1=\dim S^k\h^*.
\end{align*}
If $k$ is odd, then similarly
\begin{align*}\textrm{number of vectors }&=\lfloor \frac{k-3}{6} \rfloor + 1+\lfloor \frac{k-3}{6} \rfloor + 1+2(\lfloor \frac{k-1}{6} \rfloor + 1)+2(\lfloor \frac{k-5}{6} \rfloor + 1)\\
&=6+2\left( \lfloor\frac{k-1}{6}\rfloor+\lfloor\frac{k-3}{6}\rfloor+\lfloor\frac{k-5}{6}\rfloor \right)\\
&=6+2\left(\frac{k-1}{2}-2 \right)=k+1=\dim S^k\h^*.
\end{align*}

So, the set of vectors in the statement is a linearly independent set of size equal to the dimension of the vector space, so it is a basis. 

\end{proof}

We wish to find a nice basis of every Verma module - compatible with the decomposition as an $S_3$ representation and one in which we can reasonably compute Dunkl operators. For $\tau=\Triv$, we have $M_{t,c}(\Triv)\cong S\h^*\otimes \Triv \cong S\h^*$ as an $S_3$ representation, so Theorem \ref{decmposeShp>3} gives us such a basis. A similar computation works when $\tau=\Sign$. For $\tau=\Stand$, we have $M_{t,c}(\Stand)\cong S\h^*\otimes \Stand$ as an $S_3$ representation, so we need to understand how the above basis of $S\h^*$ behaves after taking a tensor product with $\Stand$. The following lemma is a standard exercise in representations of finite groups. 

\begin{lemma}\label{tensorproductsp>3}
Let $p>3$. As $S_3$ representations, 
\begin{enumerate}
\item $\Triv \otimes \Stand\cong \Stand$ tautologically; 
\item $\Sign \otimes \Stand\cong \Stand$ with the isomorphism given by 
$$v_{\Sign}\otimes b_+\mapsto -3b_-, \quad v_{\Sign}\otimes b_-\mapsto b_+;$$
\item $\Stand \otimes \Stand\cong \Triv \oplus \Sign \oplus \Stand$; with a compatible basis given by 
\begin{itemize}
\item $b_+\otimes b_++3b_-\otimes b_-$, spanning a subrepresentation isomorphic to $\Triv$;
\item $b_+\otimes b_--b_-\otimes b_+$ spanning a subrepresentation isomorphic to $\Sign$;
\item $-b_+\otimes b_++3b_-\otimes b_-, \, b_+\otimes b_-+b_-\otimes b_+$ spanning a subrepresentation isomorphic to $\Stand$, with the isomorphism given by $$b_+\mapsto -b_+\otimes b_++3b_-\otimes b_-, \quad b_-\mapsto b_+\otimes b_-+b_-\otimes b_+.$$
\end{itemize}
\end{enumerate}
\end{lemma}

Putting together Theorem \ref{decmposeShp>3} and Lemma \ref{tensorproductsp>3} we immediately get the following theorem. 

\begin{theorem}\label{decmposeVermap>3}
In characteristic $p>3$, $M^k_{t,c}(\Stand)\cong S^k\h^*\otimes \Stand$ is a direct sum of the following irreducible $S_3$ representations: 
\begin{itemize}
\item for every $a,b\in \mathbb{N}_0$ satisfying $2a+3b=k$, a subrepresentation isomorphic to $\Stand$ with a basis $$\left\{ \sigma_2^a\sigma_3^b\otimes b_+,\, \sigma_2^a\sigma_3^b\otimes b_-  \right\};$$
\item for every $a,b \in \mathbb{N}_0$ satisfying $2a+3b+3=k$,  a subrepresentation isomorphic to $\Stand$ with a basis
$$\left\{\sigma_2^a\sigma_3^b q \otimes b_-,  \,\frac{-1}{3}\cdot \sigma_2^a\sigma_3^b q \otimes b_+ \right\};$$
\item for every $a,b \in \mathbb{N}_0$ satisfying $2a+3b=k-1$, 
a subrepresentation isomorphic to $\Triv$ with a basis $$\left\{ \sigma_2^a\sigma_3^b \cdot (b_+\otimes b_++ 3 b_-\otimes b_-) \right\};$$
\item for every $a,b \in \mathbb{N}_0$ satisfying $2a+3b=k-1$, 
a subrepresentation isomorphic to $\Sign$ with a basis $$\left\{ \sigma_2^a\sigma_3^b \cdot (b_+\otimes b_--b_-\otimes b_+) \right\};$$
\item for every $a,b \in \mathbb{N}_0$ satisfying $2a+3b=k-1$, 
a subrepresentation isomorphic to $\Stand$ with a basis $$\left\{ \sigma_2^a\sigma_3^b \cdot (-b_+\otimes b_++3b_-\otimes b_-),  \, \sigma_2^a\sigma_3^b \cdot (b_+\otimes b_-+b_-\otimes b_+) \right\};$$
\item for every $a,b \in \mathbb{N}_0$ satisfying $2a+3b=k-2$, 
a subrepresentation isomorphic to $\Triv$ with a basis $$\left\{ \sigma_2^a\sigma_3^b \cdot \left( (-b_+^2+3b_-^2)\otimes b_++3\cdot (2b_+b_-)\otimes b_-\right) \right\}.$$
\item for every $a,b \in \mathbb{N}_0$ satisfying $2a+3b=k-2$, 
a subrepresentation isomorphic to $\Sign$ with a basis $$\left\{ \sigma_2^a\sigma_3^b \cdot \left( (-b_+^2+3b_-^2)\otimes b_--(2b_+b_-)\otimes b_+\right) \right\}.$$
\item for every $a,b \in \mathbb{N}_0$ satisfying $2a+3b=k-2$, 
a subrepresentation isomorphic to $\Stand$ with a basis $$\left\{ \sigma_2^a\sigma_3^b 
\left(-(-b_+^2+3b_-^2)\otimes b_++3(2b_+b_-)\otimes b_- \right), \sigma_2^a\sigma_3^b \cdot \left((-b_+^2+3b_-^2)\otimes b_-+(2b_+b_-)\otimes b_+ \right) \right\}.$$
\end{itemize}
Wherever in this statement a representation isomorphic to $\Stand$ is given by a basis $\left\{u,v\right\}$, the map $b_+\mapsto u, b_-\mapsto v$ is an $S_3$-isomorphism.
\end{theorem}

\begin{cor}\label{characterShp>3}
Assume $p>3$.
\begin{enumerate}
\item The character of the graded $S_3$ representation $S\h^*$ is 
$$\chi_{S\h^*}(z)=\frac{1}{(1-z^2)(1-z^3)}\left([\Triv]+(z+z^2)[\Stand]+z^3[\Sign]\right).$$
\item The characters of Verma modules for $H_{t,c}(S_3,\h)$ are given by 
\begin{align*}
\chi_{M_{t,c}(\Triv)}(z)&=\chi_{S\h^*}(z)\\
\chi_{M_{t,c}(\Sign)}(z)&=\frac{1}{(1-z^2)(1-z^3)}\left([\Sign]+(z+z^2)[\Stand]+z^3[\Triv]\right)\\
\chi_{M_{t,c}(\Stand)}(z)&=\frac{1}{(1-z^2)(1-z^3)}\left((1+z+z^2+z^3)[\Stand]+(z+z^2)([\Triv]+[\Sign])\right).
\end{align*}
\item The characters of baby Verma modules for $H_{t,c}(S_3,\h)$ are given by 
\begin{align*}
\chi_{N_{0,c}(\tau)}(z)&=\chi_{M_{0,c}(\tau)}(z)(1-z^{2})(1-z^{3}).\\
\chi_{N_{1,c}(\tau)}(z)&=\chi_{M_{1,c}(\tau)}(z)(1-z^{2p})(1-z^{3p}).%\\
%\chi_{N_{t,c}(\Triv)}(z)&=\chi_{S\h^*}(z)(1-z^{2p})(1-z^{3p})\\
%\chi_{N_{t,c}(\Sign)}(z)&=\frac{(1-z^{2p})(1-z^{3p})}{(1-z^2)(1-z^3)}\left([\Sign]+z^3[\Triv]+(z+z^2)[\Stand]\right)\\
%\chi_{N_{t,c}(\Stand)}(z)&=\frac{1}{(1-z^2)(1-z^3)}\left((1+z+z^2+z^3)[\Stand]+(z+z^2)([\Triv]+[\Sign])\right). 
\end{align*}
\end{enumerate}
\end{cor}
\begin{proof}
By Theorem \ref{decmposeShp>3}  
\begin{align*}
\chi_{S\h^*}(z)&=\sum_{a,b} z^{2a+3b}[\Triv]+\sum_{a,b} z^{2a+3b+3}[\Sign]+\sum_{a,b} (z^{2a+3b+1}+z^{2a+3b+2})[\Stand] \\
&=\frac{1}{(1-z^2)(1-z^3)}\left([\Triv]+z^3[\Sign]+(z+z^2)[\Stand]\right).
\end{align*}
The characters of Verma and baby Verma modules follow directly. 
\end{proof}

\begin{cor}\label{characterSphp>3}
Let $S^{(p)}\h^*$ be the quotient of $S\h^*$ by the ideal generated by $x^p$ for all $x\in \h^*$. For $p>3$, its character is 
\begin{align*}&\chi_{S^{(p)}\h^*}(z)=\chi_{S\h^*}(z)\cdot (1-z^p[\Stand]+z^{2p}[\Sign])\\
&\quad =\frac{1}{(1-z^2)(1-z^3)}\left([\Triv]+z^3[\Sign]+(z+z^2)[\Stand]\right) (1-z^p[\Stand]+z^{2p}[\Sign]).
\end{align*}
\end{cor}
%\begin{proof}
%The ideal in $S\h^*$ generated by $b_+^p,b_-^p$ is isomorphic to the image of the map of $S\h^*$ modules $\phi:S\h^*\otimes \Stand\to S\h^*$ induced by $\phi(1\otimes b_{\pm})=b_{\pm}^p$. The kernel of this map is generated by $b_+^p\otimes b_--b_-^p\otimes b_+$ and generates a copy of $S\h^*\otimes \Sign$. This explains the first line. The second line then follows from Corollary \ref{characterShp>3}. 
%\end{proof}

\begin{remark}\label{remark--restrictedpoly}
This is a special case of the formula for an arbitrary representation $U$,  stating $\chi_{S^{(p)}U}(z)=\chi_{SU}(z)\cdot \sum_{i=0}^{\dim U} z^p[\Lambda^i U].$
\end{remark}

In characteristic $2$ the rescaled Young basis $b_+,b_-$ does not make sense. (As written, $b_+$ and $b_-$ become equal. More generally, this basis relies on restricting $\h^*$ to $S_2=\left\{e,s_1\right\}$ and decomposing it into the trivial and sign representation of $S_2$, but in characteristic $2$ these representations are isomorphic.) Instead, in characteristic $2$ we use the basis $\x{1},\x{2}$ of the standard representation with $\x{3}=-\x{1}-\x{2}=\x{1}+\x{2}$. The elementary symmetric polynomials in this case are $\s{2}$ and $\s{3}$. The analogue of Theorem \ref{decmposeShp>3} is the following theorem. 

\begin{theorem}\label{decmposeShp=2}
In characteristic $p=2$, $S^k\h^*$ is a direct sum of the following indecomposable $S_3$ representations: 
\begin{itemize}
\item for every $a\in \mathbb{N}_0$ satisfying $2a=k$, a subrepresentation isomorphic to $\Triv$ with a basis $$\left\{ \s{2}^a \right\};$$
\item for every $a,b \in \mathbb{N}_0$ satisfying $2a+3b=k$, $b>0$, an indecomposable extension of two copies of $\Triv$, with a basis 
$$\left\{\s{2}^a\s{3}^b,\,  \s{2}^a\s{2}^{b-1}(\x{1}^3+\x{1}^2\x{2}+\x{2}^3)\right\};$$
\item for every $a,b \in \mathbb{N}_0$ satisfying $2a+3b=k-1$, 
a subrepresentation isomorphic to $\Stand$ with a basis $$\left\{ \s{2}^a\s{3}^b \x{1}, \s{2}^a\s{3}^b \x{2} \right\};$$
\item for every $a,b \in \mathbb{N}_0$ satisfying $2a+3b=k-2$, 
a subrepresentation isomorphic to $\Stand$ with a basis $$\left\{ \s{2}^a\s{3}^b \x{1}^2, \s{2}^a\s{3}^b \x{2}^2 \right\}.$$
\end{itemize}
\end{theorem}
\begin{proof}
It is clear that for every $a,b \in \mathbb{N}_0$ satisfying $2a+3b=k$ the space spanned by $\s{2}^a\s{3}^b$ is a subrepresentation isomorphic to $\Triv$. When $b>0$, it is extended by the one dimensional space with a basis $\s{2}^a\s{3}^{b-1}(\x{1}^3+\x{1}^2\x{2}+\x{2}^3)$, as
\begin{align*}
s_1.(\s{2}^a\s{3}^{b-1}(\x{1}^3+\x{1}^2\x{2}+\x{2}^3))&=\s{2}^a\s{3}^{b-1}(\x{1}^3+\x{1}^2\x{2}+\x{2}^3)+\s{2}^a\s{3}^b\\
s_2.(\s{2}^a\s{3}^{b-1}(\x{1}^3+\x{1}^2\x{2}+\x{2}^3))&=\s{2}^a\s{3}^{b-1}(\x{1}^3+\x{1}^2\x{3}+\x{3}^3)\\
&=\s{2}^a\s{3}^{b-1}(\x{1}^3+\x{1}^2\x{2}+\x{2}^3)+\s{2}^a\s{3}^b.
\end{align*}
It is also clear that the space spanned by $\left\{ \s{2}^a\s{3}^b \x{1}, \s{2}^a\s{3}^b \x{2} \right\}$ is a subrepresentation isomorphic to $\Stand$ via the isomorphism $\x{i}\mapsto \s{2}^a\s{3}^b \x{i}$ and that the space spanned by $\left\{ \s{2}^a\s{3}^b \x{1}^2, \s{2}^a\s{3}^b \x{2}^2 \right\}$ is a subrepresentation isomorphic to $\Stand$ via the isomorphism $\x{i}\mapsto \s{2}^a\s{3}^b \x{i}^2$. 

The set $$\left\{\s{2}^a\s{3}^b,\,  \s{2}^a\s{2}^{b-1}(\x{1}^3+\x{1}^2\x{2}+\x{2}^3) \mid 2a+3b=k\right\}$$ is linearly independent, and so is the set 
$$\left\{ \s{2}^a\s{3}^b \x{1},\, \s{2}^a\s{3}^b \x{2} \mid 2a+3b=k-1\right\}  \cup
\left\{ \s{2}^a\s{3}^b \x{1}^2,\, \s{2}^a\s{3}^b \x{2}^2 \mid 2a+3b=k-2\right\}.$$
The central element $(123)+(132)\in \Bbbk[S_3]$ acts on the span of the first of these sets by $0$ and on the second of them by $1$, $(123)+(132)$ and $e-(123)-(132)$ act as projections on the span of their union, and this allows us to conclude that the union of these sets is linearly independent as well. To show that it is a basis of $S^k\h^*$, and thus conclude that $S\h^*$ is indeed a direct sum as stated in the theorem, it now suffices to show that the number of elements in this set and compare it to $\dim S^k\h^*=k+1$. This is exactly the same calculation as in the proof of Theorem \ref{decmposeShp>3}. 

\end{proof}

Next, we will need an analogue of Lemma \ref{tensorproductsp>3}. 

\begin{lemma}\label{tensorproductsp=2}
Let $p=2$. As an $S_3$ representation, $\Stand \otimes \Stand$ is a direct sum of the following indecomposable subrepresentations: 
\begin{itemize}
\item an indecomposable extension of two copies of $\Triv$, with a basis 
$$\left\{\x{1}\otimes \x{2}+\x{2}\otimes \x{1}, \x{1}\otimes \x{1}+\x{2}\otimes \x{1}+\x{2}\otimes \x{2} \right\},$$
with $\x{1}\otimes \x{2}+\x{2}\otimes \x{1}$ spanning a subrepresentation of this extension; 
\item a subrepresentation isomorphic to $\Stand$ with a basis 
$$\left\{\x{1}\otimes \x{1}+\x{2}\otimes \x{1}+\x{1}\otimes \x{2}, \, \x{2}\otimes \x{2}+\x{2}\otimes \x{1}+\x{1}\otimes \x{2}\right\},$$
with the isomorphism given by $\x{i}\mapsto \x{i}\otimes \x{i}+\x{2}\otimes \x{1}+\x{1}\otimes \x{2}$. 
\end{itemize}
\end{lemma}

Putting together Theorem \ref{decmposeShp=2} and Lemma \ref{tensorproductsp=2} we get the following theorem.

\begin{theorem}\label{decmposeVermaStandp=2}
In characteristic $p=2$, $M_{t,c}(\Stand)\cong S^k\h^*\otimes \Stand$ is a direct sum of the following indecomposable $S_3$ representations: 
\begin{itemize}
\item for every $a,b \in \mathbb{N}_0$ satisfying $2a+3b=k$, a subrepresentation isomorphic to $\Stand$ with a basis 
$$\left\{\s{2}^a\s{3}^b\otimes \x{1},\, \s{2}^a\s{3}^b\otimes \x{2}\right\};$$
\item for every $a,b \in \mathbb{N}_0$ satisfying $2a+3b=k-3$, a subrepresentation isomorphic to $\Stand$ with a basis 
$$ \left\{\s{2}^a\s{2}^{b}\left( (\x{1}^3+\x{1}^2\x{2}+\x{2}^3)\otimes \x{1}+\s{2}\otimes \x{2}\right)\right.,$$
$$\left.\s{2}^a\s{2}^{b}\left( (\x{1}^3+\x{1}^2\x{2}+\x{2}^3)\otimes \x{2}+\s{3}\otimes (\x{1}+\x{2})\right)\right\};$$
\item for every $a,b \in \mathbb{N}_0$ satisfying $2a+3b=k-1$, 
an indecomposable extension of two copies of $\Triv$, with a basis 
$$\left\{ \s{2}^a\s{3}^b (\x{1}\otimes \x{2}+\x{2}\otimes \x{1}) \right\}$$
of the subrepresentation isomorphic to  $\Triv$ and a basis
$$\left\{ \s{2}^a\s{3}^b ( \x{1}\otimes \x{1}+\x{2}\otimes \x{1}+\x{2}\otimes \x{2})\right\}$$
of the quotient isomorphic to $\Triv$;
\item for every $a,b \in \mathbb{N}_0$ satisfying $2a+3b=k-1$, 
a subrepresentation isomorphic to $\Stand$ with a basis 
$$\left\{ \s{2}^a\s{3}^b (\x{1}\otimes \x{1}+\x{2}\otimes \x{1}+\x{1}\otimes \x{2}),\, \s{2}^a\s{3}^b (\x{2}\otimes \x{2}+\x{2}\otimes \x{1}+\x{1}\otimes \x{2})  \right\};$$
\item for every $a,b \in \mathbb{N}_0$ satisfying $2a+3b=k-2$, 
an indecomposable extension of two copies of $\Triv$, with a basis 
$$\left\{ \s{2}^a\s{3}^b (\x{1}^2\otimes \x{2}+\x{2}^2\otimes \x{1}) \right\}$$
of the subrepresentation isomorphic to  $\Triv$ and a basis
$$\left\{ \s{2}^a\s{3}^b ( \x{1}^2\otimes \x{1}+\x{2}^2\otimes \x{1}+\x{2}^2\otimes \x{2})\right\}$$
of the quotient isomorphic to $\Triv$;
\item for every $a,b \in \mathbb{N}_0$ satisfying $2a+3b=k-2$, 
a subrepresentation isomorphic to $\Stand$ with a basis 
$$\left\{ \s{2}^a\s{3}^b (\x{1}^2\otimes \x{1}+\x{2}^2\otimes \x{1}+\x{1}^2\otimes \x{2}),\, \s{2}^a\s{3}^b (\x{2}^2\otimes \x{2}+\x{2}^2\otimes \x{1}+\x{1}^2\otimes \x{2})  \right\}.$$
\end{itemize}
\end{theorem}

\begin{cor}\label{characterShp=2}
Assume $p=2$.
\begin{enumerate}
\item The character of the graded $S_3$ representation $S\h^*$ is 
$$\chi_{S\h^*}(z)=\frac{1}{(1-z^2)(1-z^3)}\left( (1+z^3)[\Triv]+(z+z^2)[\Stand]\right).$$
\item The character of the graded $S_3$ representation $S^{(2)}\h^*$ is 
$$\chi_{S^{(2)}\h^*}(z)= [\Triv](1+z^2)+[\Stand]z.$$
\item The characters of Verma modules for $H_{t,c}(S_3,\h)$ are given by 
\begin{align*}
\chi_{M_{t,c}(\Triv)}(z)&=\chi_{S\h^*}(z)\\
\chi_{M_{t,c}(\Stand)}(z)&=\frac{1}{(1-z^2)(1-z^3)}\left((1+z+z^2+z^3)[\Stand]+2 (z+z^2)[\Triv]\right).
\end{align*}
\item The characters of baby Verma modules for $H_{t,c}(S_3,\h)$ are given by 
\begin{align*}
\chi_{N_{0,c}(\tau)}(z)&=\chi_{M_{0,c}(\tau)}(z)(1-z^{2})(1-z^{3}).\\
\chi_{N_{1,c}(\tau)}(z)&=\chi_{M_{1,c}(\tau)}(z)(1-z^{2p})(1-z^{3p}).%\\
%\chi_{N_{t,c}(\Triv)}(z)&=\chi_{S\h^*}(z)(1-z^{2p})(1-z^{3p})\\
%\chi_{N_{t,c}(\Sign)}(z)&=\frac{(1-z^{2p})(1-z^{3p})}{(1-z^2)(1-z^3)}\left([\Sign]+z^3[\Triv]+(z+z^2)[\Stand]\right)\\
%\chi_{N_{t,c}(\Stand)}(z)&=\frac{1}{(1-z^2)(1-z^3)}\left((1+z+z^2+z^3)[\Stand]+(z+z^2)([\Triv]+[\Sign])\right). 
\end{align*}
\end{enumerate}
\end{cor}
\begin{proof}
Similar to the proof of Corollary \ref{characterShp>3}. 
\end{proof}

\subsection{A few well known lemmas}
For later use let us state several well known lemmas and prove some of them. Most of them hold in much greater generality but we only state them for $H_{t,c}(S_3,\h)$ here: Lemma \ref{lemma-rescaleparameter}, \ref{t=c=0} and \ref{t=1,c=0} extend to all rational Cherednik algebras, Lemma \ref{Sign} to all rational Cherednik algebras and all their characters, and Proposition \ref{WhatIsGeneric} to all rational Cherednik algebras of type $A$.

\begin{lemma}\label{lemma-rescaleparameter}
For every $a\in \Bbbk^\times$, the map $\Psi: H_{t,c}(S_3,\h) \to H_{at,ac}(S_3,\h)$ given on the generators $x\in \h^*,y\in\h, g\in S_3$
$$\Psi(x)=ax, \quad \Psi(y)=y, \quad \Psi(g)=g$$
extends uniquely to an isomorphism of algebras.
\end{lemma}
As a consequence, it is enough to consider just $t=0$ and $t=1$.

\begin{lemma}\label{Sign}
\begin{enumerate}
\item The map $\Phi:H_{t,c}(S_3,\h)\to H_{t, -c}(S_3,\h)$ given on the generators $x\in \h^*,y\in\h, s_1,s_2\in S_3$ by 
$$\Phi(x)=x, \quad \Phi(y)=y, \quad \Phi(s_i)=-s_i$$
is an isomorphism of algebras. 
\item The pullback map $\Phi^*:\mathcal{O}_{t,-c}\to \mathcal{O}_{t,c}$ is an equivalence of categories and
$$\Phi^*(L_{t,-c}(\tau))\cong L_{t, c}(\Sign \otimes \tau).$$
\end{enumerate}
\end{lemma}

Recall that the behaviour of Cherednik algebras depends on the parameter $c$ in a specific way: for generic $c$ the Calogero-Moser space is smooth and irreducible representations are ``big'' (in characteristic $0$ this specifically means equal to Verma modules), while for special $c$ the Calogero-Moser space is singular and irreducible representations are smaller than for generic $c$ (in terms of dimension or Hilbert series). 

\begin{prop}[\cite{Li14} Prop 2.8] \label{WhatIsGeneric}
For the rational Cherednik algebra $H_{t,c}(S_3,\h)$: 
\begin{enumerate}
\item If $t=0$, all $c\ne 0$ are generic. 
\item If $t=1$, all $c\notin \mathbb{F}_p$ are generic. 
\end{enumerate}
\end{prop}
\begin{proof}
For $t=0$ see Lemma \ref{lemma-rescaleparameter}, for $t=1$ see \cite{Li14} Proposition 2.8. 
\end{proof}

\begin{lemma}\label{t=c=0}
The irreducible representation $L_{0,0}(\tau)$ is the quotient of the Verma module $M_{0,0}(\tau)$ by all the positively graded vectors, with the character
$$\chi_{L_{0,0}(\tau)}(z)=[\tau]$$
and the
Hilbert polynomial $$h_{L_{0,0}(\tau)}(z)=\dim \tau.$$
\end{lemma}
\begin{proof}
When $t=c=0$, the Dunkl operators are all identically equal to $0$, so all vectors of strictly positive degree are singular. Thus, $L_{0,0}(\tau)$ is concentrated in degree $0$, where it equals $\tau$. 
\end{proof}

\begin{lemma}\label{t=1,c=0}
The irreducible representation $L_{1,0}(\tau)$ is the quotient of the Verma module $M_{1,0}(\tau)$ by all the vectors of the form $x^p\otimes v$ for any $x\in \h^*$ and $v\in \tau$.
It has the character
$$\chi_{L_{1,0}(\tau))}(z)=\chi_{S^{(p)}\h^*}(z)\cdot [\tau] $$
and the
Hilbert polynomial $$h_{L_{1,0}(\tau)}(z)=\left(\frac{1-z^p}{1-z}\right)^{\dim \h} \dim \tau.$$
\end{lemma}
\begin{proof}This is well known, see also \cite{Li14}. When $t=1$ and $c=0$, the Dunkl operators have a very simple form $D_y=\partial_y\otimes \id,$
and their joint kernel consists of vectors of the form $x^p\otimes v$ for all $x\in \h^*$ and $v\in \tau$. 
\end{proof}

We now rephrase Theorem 3.2 from \cite{Li14} in our conventions (using $\h$ instead of $V$) and state a slight strengthening of it with the same proof. 
\begin{lemma}[\cite{Li14} Theorem 3.2]\label{Lian3.2.}
Let $\Bbbk$ be an algebraically closed field of characteristic $p$. The character of the irreducible representation $L_{t,c}(\Triv)$ of $H_{t,c}(S_n)$ is equal to $[\Triv]$ if and only if $t=cn$. In particular, this holds when $t=c=0$, and when $p$ does not divide $n$ and the values of the parameters are $t=1,c=1/n$. 
\end{lemma}
\begin{proof}
It suffices to note that
\begin{align*}
D_{y_i-y_j}(\x{k})&=t\cdot(\delta_{ij}-\delta_{jk})-c\left(\sum_{a<b}\left<y_i-y_j,\x{a}-\x{b}\right> \frac{\x{k}-(ab)\x{k}}{\x{a}-\x{b}} \right)\\
%&=t\cdot(\delta_{ij}-\delta_{jk})-c\left(\sum_{k<b}\left<y_i-y_j,\x{k}-\x{b}\right> \frac{\x{k}-\x{b}}{\x{k}-\x{b}}+\sum_{a<k}\left<y_i-y_j,\x{a}-\x{k}\right> \frac{\x{k}-\x{a}}{\x{a}-\x{k}} \right)\\
%&=t\cdot(\delta_{ij}-\delta_{jk})-c\left(\sum_{k<b}\left<y_i-y_j,\x{k}-\x{b}\right> -\sum_{a<k}\left<y_i-y_j,\x{a}-\x{k}\right> \right)\\
%&=t\cdot(\delta_{ij}-\delta_{jk})-c\left(\sum_{a\ne k }\left<y_i-y_j,\x{k}-\x{a}\right> \right)\\
&=t\cdot(\delta_{ij}-\delta_{jk})-c\left(\sum_{a}\left<y_i-y_j,\x{k}-\x{a}\right> \right)\\
%&=t\cdot(\delta_{ij}-\delta_{jk})-c\left(\sum_{a}\left<y_i-y_j,\x{k}\right> \right)\\
&=(t-cn) \cdot(\delta_{ij}-\delta_{jk}).
\end{align*}
This shows that all $\x{k}\in \h^*\cong M_{t,c}^1(\Triv)$ are singular if and only if $t=cn$. 

\end{proof}

\begin{lemma}\label{KerD1}
Assume that $p>3$, $M$ is a Verma module $M_{t,c}(\tau)$ or its quotient, and that a homogeneous $v \in M$ is either $S_3$ invariant or anti-invariant. Then $v$ is singular if and only if $D_{y_1}v=0.$
\end{lemma}
\begin{proof}
First, extend the action of $y\in \h$ to the action of all $y\in V$ by the same formula for Dunkl operators. This amounts to letting $Y=y_1+y_2+y_3$ act by $0$.

$\Rightarrow$
Assume $v$ satisfies $D_{y_i-y_j}(v)=0$ for all $i,j$. Then
$$D_{y_1}(v)=\frac{1}{3}\left( D_{y_1-y_2}(v)+D_{y_1-y_3}(v)+D_{Y}(v) \right)=0$$
as claimed. (Notice that this argument fails in characteristic $3$.)

$\Leftarrow$
Assume $D_{y_1}(v)=0$ and $s.v=\epsilon v$ for all $s\in S$ and $\epsilon=\pm 1$. Then 
\begin{align*}
D_{y_2}(v)&=(12)D_{y_1}((12).v)=\epsilon (12) D_{y_1}(v)=0,\\
D_{y_3}(v)&=D_{Y}(v)-D_{y_1}(v)-D_{y_2}(v)=0,
\end{align*}
so $D_y(v)=0$ for all $y\in V$ and in particular all $y\in \h.$
\end{proof}

\section{Irreducible representations of $H_{t,c}(S_3,\h)$ in characteristic $2$}\label{sectionp=2}

As explained in Section \ref{sect-IrrepsofS3}, over an algebraically closed field $\Bbbk$ of characteristic $2$ the irreducible representations of $H_{t,c}(S_3,\h)$ are $L_{t,c}(\Triv)$ and $L_{t,c}(\Stand)$. We work with the spanning set $\x{1},\x{2},\x{3}$ of $\h^*$ satisfying $\x{1}+\x{2}+\x{3}=0$ and the basis $\x{1},\x{2}$. The aim of the section is to prove the following theorem. Recall that the character of $S^{(2)}\h^*$ can be found in Corollary \ref{characterShp=2}.

\begin{theorem}\label{mainthmp=2}
The characters of the irreducible representations $L_{t,c}(\tau)$ of the rational Cherednik algebra $H_{t,c}(S_3,\h)$ over an algebraically closed field of characteristic $2$ for any $c,t$ and $\tau$ are given in the following tables. 
\begin{center}
{\renewcommand*{\arraystretch}{1.5}\begin{tabular}{ c |c }
$p=2$ & $\Triv$  \\ \hline \hline 
 $t=0$, $c\ne 0$ & $[\Triv](1+z^3)+[\Stand](z+z^2)$  \\
 $t=0$, $c=0$ & $[\Triv]$ \\ \hline 
 $t=1$, $c\notin \mathbb{F}_2$
  & $\chi_{S^{(2)}\h^*}(z) \left( [\Triv](1+z^6)+[\Stand](z^2+z^4)\right)$ \\ 
 $t=1$, $c=0$  & $\chi_{S^{(2)}\h^*}(z)$     \\ 
 $t=1$, $c=1$  & $[\Triv]$ 
\end{tabular}}
\end{center}

\begin{center}
{\renewcommand*{\arraystretch}{1.5}\begin{tabular}{ c |c }
$p=2$  & $\Stand$ \\ \hline \hline 
 $t=0$, $c\ne 0$ & $[\Stand](1+z^2)+2[\Triv]z$ \\
 $t=0$, $c=0$  & $[\Stand]$ \\ \hline 
 $t=1$, $c\notin \mathbb{F}_2$ %& $L_{t,c}(\tau)=N_{t,c}(\tau)$ &  \\  
  & $\chi_{S^{(2)}\h^*}(z)  \left([\Stand](1+z^4)+2[\Triv]z^2   \right)$ \\ 
 $t=1$, $c=0$  & $[\Stand](1+z+z^2)+2[\Triv]z$   \\ 
 $t=1$, $c=1$   & $[\Stand](1+z+z^2+2z^3+z^4+z^5+z^6)+[\Triv](z+2z^2+2z^4+z^5)$
\end{tabular}}
\end{center}

The corresponding Hilbert polynomials are: 
\begin{center}
{\renewcommand*{\arraystretch}{1.5}\begin{tabular}{ c |c |c }
$p=2$ & $\Triv$ & $\Stand$ \\ \hline \hline 
 $t=0$, $c\ne 0$ & $1+2z+2z^2+z^3$  & $2+2z+2z^2$ \\
 & \scriptsize{\cite[Thm 2.11]{CaKa21}} & \\  \hline 
 $t=0$, $c=0$ & 1 & 2 \\ \hline 
 $t=1$, $c\notin \mathbb{F}_2$  
  & $(1+z)^2 (1+2z^2+2z^4+z^6)$ & $2(1+z)^2(1+z^2+z^4)$ \\ 
  & \scriptsize{\cite[Thm 3.17]{CaKa21}} & \\  \hline  
 $t=1$, $c=0$  & $\left( 1+z \right)^2$  & $2\left( 1+z \right)^2$   \\ 
 $t=1$, $c=1$  & $1$ & $\frac{2-z-z^3-z^5-z^7+2z^8}{(1-z)^2}$ \\
 & \scriptsize{\cite[Thm 3.2]{Li14}}
\end{tabular}}
\end{center}
The singular vectors are given  explicitly in the lemmas below. 
\end{theorem}
\begin{proof}
The irreducible representation $L_{t,c}(\Triv)$ is described in the following lemmas:
\begin{itemize}
\item for $t=0,c\ne 0$ in Lemma \ref{p=2,t=0,generic,Triv};
\item for $t=c=0$ in Lemma \ref{t=c=0} or Lemma \ref{Lian3.2.}; 
\item for $t=1$, $c\ne 0,1$ in Lemma \ref{p=2,t=1,generic,Triv};
\item for $t=1$, $c=0$ in Lemma \ref{t=1,c=0};
\item for $t=1$, $c=1$ in Lemma \ref{Lian3.2.}.
\end{itemize}
The irreducible representation $L_{t,c}(\Stand)$ is described in the following lemmas:
\begin{itemize}
\item for $t=0,c\ne 0$ in Lemma \ref{pne3t=0cne0char};
\item for $t=c=0$ in Lemma \ref{t=c=0};
\item for $t=1$, $c\ne 0,1$ in Lemma \ref{p=2,t=1,Stand,cgen-endproof} and \ref{p=2,t=1,Stand,char};
\item for $t=1$, $c=0$ in Lemma \ref{t=1,c=0};
\item for $t=1$, $c=1$ in Lemma \ref{p=2,t=1,Stand,c=1}.
\end{itemize}

\end{proof}

\subsection{The irreducible representation $L_{0,c}(\Triv)$ in characteristic $2$ for $c\ne 0$}

We start by citing a theorem from \cite{CaKa21}. 

\begin{theorem}[\cite{CaKa21} Theorem 2.11.]\label{Ca-Ka-2-11}
Let $\Bbbk$ be an algebraically closed field of characteristic $2$, let $n$ be an arbitrary odd integer, $t=0$ and $c\in \Bbbk$ be generic. The singular vectors in $M_{0,c}(\Triv)$ which generate the maximal proper graded submodule are $\overline{x_i}^2+\overline{x_i}\overline{x_j}+\overline{x_j}^2, \overline{x_i}\overline{x_j}\overline{x_k}, \quad i<j<k<n.$ The Hilbert polynomial of the irreducible representation $L_{0,c}(\Triv)$ of $H_{0,c}(S_n)$  equals $$h_{L_{0,c}(\Triv)}(z)=(1+z)(1+(n-2)z+z^2).$$
\end{theorem}

We reprove a special case of this theorem, while also calculating the character of the representation $L_{0,c}(\Triv)$ and clarifying that ``generic $c$'' in this case means $c\ne 0$. The special case we will prove is the following lemma. The proof is following the ideas of \cite{CaKa21} and simplifying where possible.

\begin{lemma}\label{p=2,t=0,generic,Triv}
Let $\Bbbk$ be an algebraically closed field of characteristic $2$, and let the values of the parameters be $t=0$ and $c\ne 0$. The irreducible representation $L_{0,c}(\Triv)$ of $H_{0,c}(S_3)$ is equal to the baby Verma module $N_{0,c}(\Triv)$, with the character
$$\chi_{L_{0,c}(\Triv)}(z)=[\Triv]+[\Stand]z+[\Stand]z^2+[\Triv]z^3$$
and the
Hilbert polynomial $$h_{L_{0,c}(\Triv)}(z)=(1+z)(1+z+z^2)=1+2z+2z^2+z^3.$$
\end{lemma}
\begin{proof}
When $n=3$, the elementary symmetric polynomials $\overline{\sigma_2}$ and $\overline{\sigma_3}$ in $S\h^*$ are exactly the vectors from Theorem \ref{Ca-Ka-2-11}. The quotient of $M_{0,c}(\Triv)$ by the submodule generated by these vectors is the baby Verma module $N_{0,c}(\Triv)$, which has the above stated character and Hilbert series by Corollary \ref{characterShp=2}. 

It remains to show that when $c\ne 0$, the baby Verma module $N_{0,c}(\Triv)$ is irreducible. We do this by checking that there are no singular vectors in it. 

For any $a_1,a_2\in \Bbbk$ and $i=1,2$ we directly calculate that 
\begin{align*}
D_{y_i-y_3}(a_1\x{1}+a_2\x{2})&=ca_i.
\end{align*}
As $c\ne 0$ this shows there are no singular vectors in $N_{0,c}^1(\Triv)$. 

A basis of $M_{0,c}^2(\Triv)\cong S^2\h^*$ is $\left\{\x{1}^2,\x{1}\x{2},\x{2}^2\right\}$, and $\overline{\sigma_2}$ is their sum. After taking the quotient by $\overline{\sigma_2}$, we choose $\left\{\x{1}^2,\x{2}^2\right\}$ as a basis of $N_{0,c}^2(\Triv)$. To look for singular vectors in $N_{0,c}^2(\Triv)$ we calculate, for $a_1,a_2\in \Bbbk$
\begin{align*}
D_{y_1-y_3}(a_1\x{1}^2+a_2\x{2}^2)&=-c\frac{\id-(12)}{\x{1}-\x{2}}(a_1\x{1}^2+a_2\x{2}^2)-c\frac{\id-(23)}{\x{2}-\x{3}}(a_1\x{1}^2+a_2\x{2}^2)\\
%&=c\frac{a_1\x{1}^2+a_2\x{2}^2-a_1\x{2}^2+a_2\x{1}^2}{\x{1}-\x{2}}+c\frac{a_1\x{1}^2+a_2\x{2}^2-a_1\x{1}^2+a_2\x{3}^2}{\x{2}-\x{3}}\\
&=c\frac{(a_1+a_2)(\x{1}^2-\x{2}^2)}{\x{1}-\x{2}}+c\frac{a_2\x{2}^2-a_2\x{3}^2}{\x{2}-\x{3}}\\
%&=c\left((a_1+a_2)(\x{1}+\x{2}) +a_2 (\x{2}+\x{3})\right)\\
&=c\left( a_1\x{1}+ (a_1+a_2)\x{2} \right)\ne 0.
\end{align*}
This shows that there are no singular vectors in $N_{0,c}^2(\Triv)$.

Finally, a basis of $N_{0,c}^3(\Triv)$ is $\left\{\x{1}^2\x{2}\right\}$, and we have 
\begin{align*}
D_{y_1-y_3}(\x{1}^2\x{2})%&=-c(\frac{\id-(12)}{\x{1}-\x{2}}(\x{1}^2\x{2})-c\frac{\id-(23)}{\x{2}-\x{3}}(\x{1}^2\x{2}))\\
&=-c\frac{\x{1}^2\x{2}-\x{1}\x{2}^2}{\x{1}-\x{2}}-c\frac{\x{1}^2\x{2}-\x{1}^2\x{3}}{\x{2}-\x{3}}=-c\left(\x{1}\x{2}+\x{1}^2\right),
\end{align*}
which is not zero in $N_{0,c}^2(\Triv)$, showing there are no singular vectors in $N_{0,c}^3(\Triv)$.
\end{proof}

\subsection{The irreducible representation $L_{1,c}(\Triv)$ in characteristic $2$ for generic $c$} 
Again, we start by citing a theorem from \cite{CaKa21}.

\begin{theorem}[\cite{CaKa21} Theorem 3.17.]\label{thm-CaiKalinov317}
Let $\Bbbk$ be an algebraically closed field , let $n$ be an arbitrary odd integer, $t=1$ and $c\in \Bbbk$ be transcendental over $\mathbb{F}_2$. The Hilbert polynomial of the irreducible representation $L_{1,c}(\Triv)$ of $H_{1,c}(S_n,\h)$ equals $$h_{L_{1,c}(\Triv)}(z)=(1+z)^{n-1}(1+(n-1)z^2+(n-1)z^4+z^6).$$
\end{theorem}

Note that the reduced Hilbert polynomial (see \cite{BaCh13a} Proposition 3.4]) of $L_{1,c}(\Triv)$ equals the Hilbert polynomial of $L_{0,c}(\Triv)$. 

Let us reprove a special case $n=3$ of this theorem, while also calculating the character of $L_{1,c}(\Triv)$ and clarifying that generic $c$ means $c\notin \mathbb{F}_2$. 

\begin{lemma}\label{p=2,t=1,generic,Triv}
Let $\Bbbk$ be an algebraically closed field of characteristic $2$, $t=1$ and $c\ne 0,1$. The irreducible representation $L_{1,c}(\Triv)$ equals the baby Verma module $N_{1,c}(\Triv)$ with
\begin{align*}
\chi_{L_{1,c}(\Triv)}(z)&=([\Triv]+[\Stand]z+[\Triv]z^2)([\Triv]+[\Stand]z^2+[\Stand]z^4+[\Triv]z^6)\\
h_{L_{1,c}(\Triv)}(z)&=(1+z)^2(1+2z^2+2z^4+z^6).
\end{align*}
\end{lemma}
\begin{proof}
The only thing to check is that for $c\ne 0,1$ the baby Verma module $N_{1,c}(\Triv)$ is irreducible, as its character has been computed in Corollary \ref{characterShp=2}. 

Assume that $U\ne 0$ is a submodule of $N_{1,c}(\Triv)$. The module $N_{1,c}(\Triv)$ is 
a Frobenius algebra, in the sense that for every $k$ the multiplication $N_{1,c}^k(\Triv)\times N_{1,c}^{8-k}(\Triv)\to N_{1,c}^8(\Triv)$ is a nondegenerate pairing. Using the action of $S\h^*\subseteq H_{t,c}(S_3,\h)$, this implies $\x{1}^5\x{2}^3\ne 0 \in U$. Using the action of $S\h\subseteq H_{t,c}(S_3,\h)$, we get that 
$$D_{y_1-y_2}^3D_{y_1-y_3}^5(\x{1}^5\x{2}^3) \in U.$$

Let us calculate this vector. We work with the basis $\left\{\x{1}^i\x{2}^j \mid i<6, j<4\right\}$ of $N_{1,c}(\Triv)$, and use that $\x{1}+\x{2}+\x{3}=0$, $\overline{\sigma_2^2}=0$, and $\overline{\sigma_3^2}=0$ in $N_{1,c}(\Triv)$. 
\begin{align*}
D_{y_1-y_3}(\x{1}^5\x{2}^3)
&=5\cdot \x{1}^4\x{2}^3   -c \cdot  \left( \frac{\x{1}^5\x{2}^3-\x{1}^3\x{2}^5}{\x{1}-\x{2}}+\frac{\x{1}^5\x{2}^3-\x{1}^5\x{3}^3}{\x{2}-\x{3}}\right)\\
%&=\x{1}^4\x{2}^3 +c\cdot \left(\x{1}^4\x{2}^3+\x{1}^3\x{2}^4+\x{1}^5\x{2}^2+\x{1}^5\x{2}\x{3}+\x{1}^5\x{3}^2\right) \\
%&=\x{1}^4\x{2}^3 +c\cdot \left(\x{1}^4\x{2}^3+\x{1}^3\x{2}^4+\x{1}^5\x{2}^2+\x{1}^6\x{2}+\x{1}^5\x{2}^2+\x{1}^7+\x{1}^5\x{2}^2\right) \\
%&=\x{1}^4\x{2}^3 + c\cdot \left(\x{1}^4\x{2}^3+\x{1}^3\x{2}^4+\x{1}^5\x{2}^2\right) \\
&=(1+c) \x{1}^4\x{2}^3,\\
D_{y_1-y_3}^2(\x{1}^5\x{2}^3)
%&=D_{y_1-y_3}((1+c) \x{1}^4\x{2}^3)\\
&=(1+c)\left( 4\cdot \x{1}^3\x{2}^3 -c\left(\frac{\x{1}^4\x{2}^3-\x{1}^3\x{2}^4}{\x{1}-\x{2}}+\frac{\x{1}^4\x{2}^3-\x{1}^4\x{3}^3}{\x{2}-\x{3}} \right) \right)\\
%&=(1+c)c \left( (\x{1}^3\x{2}^3+\x{1}^4 (\x{2}^2+\x{2}\x{3}+\x{3}^2)\right) \\
%&=(1+c)c \left( \left(\x{1}^3\x{2}^3+\x{1}^4 (\x{2}^2+\x{2}\x{1}+\x{1}^2)\right) \right)\\
%&=(1+c)c \left(\x{1}^3\x{2}^3+\x{1}^4\x{2}^2+\x{1}^5\x{2}+\x{1}^6 \right)\\
&=(1+c)c \left(\x{1}^5\x{2}+\x{1}^4\x{2}^2+ \x{1}^3\x{2}^3\right)\\
D_{y_1-y_3}^3(\x{1}^5\x{2}^3)
%&=(1+c)c D_{y_1-y_3}\left(\x{1}^5\x{2}+\x{1}^4\x{2}^2+ \x{1}^3\x{2}^3\right)\\
%&=(1+c)c \left( \x{1}^4\x{2}+\x{1}^2\x{2}^3+c\left( 
%\x{1}^4\x{2}+\x{1}^3\x{2}^2+\x{1}^2\x{2}^3+\x{1}\x{2}^4+ 
%\x{1}^3\x{2}^2+\x{1}^2\x{2}^3+
%\x{1}^5+
%\x{1}^4(\x{2}+\x{3})+
%\x{1}^3(\x{2}^2+\x{2}\x{3}+\x{3}^2)
%\right) \right)\\
%&=(1+c)c \left( \x{1}^4\x{2}+\x{1}^2\x{2}^3+c\left( 
%\x{1}^4\x{2}+\x{1}\x{2}^4+ 
%\x{1}^5+\x{1}^4\x{2}+\x{1}^3\x{2}^2
%\right) \right)\\
%&=(1+c)c \left( \x{1}^4\x{2}+\x{1}^2\x{2}^3
%+c\left( \x{1}^5+ \x{1}^3\x{2}^2+ \x{1}^5+\x{1}^3\x{2}^2 \right) \right)\\
&=(1+c)c \left( \x{1}^4\x{2}+\x{1}^2\x{2}^3 \right),\\
D_{y_1-y_3}^4(\x{1}^5\x{2}^3)
%&=(1+c)c D_{y_1-y_3}\left(\x{1}^4\x{2}+\x{1}^2\x{2}^3 \right)\\
%&=(1+c)c^2\left( \x{1}^3\x{2}+\x{1}^2\x{2}^2+\x{1}\x{2}^3+\x{1}^2\x{2}^2+
%\x{1}^4+\x{1}^2(\x{2}^2+\x{2}\x{3}+\x{3}^2) \right)\\
%&=(1+c)c^2\left( \x{1}^3\x{2}+\x{1}\x{2}^3+
%\x{1}^4+\x{1}^4+\x{1}^3\x{2}+\x{1}^2\x{2}^2 \right)\\
&=(1+c)c^2\left( \x{1}^2\x{2}^2+\x{1}\x{2}^3 \right)\\
D_{y_1-y_3}^5(\x{1}^5\x{2}^3)
%&=(1+c)c^2D_{y_1-y_3} \left( \x{1}^2\x{2}^2+\x{1}\x{2}^3 \right)\\
%&=(1+c)c^2 \left(\x{2}^3-c\left(
%\x{1}^2\x{2}+\x{1}\x{2}^2+
%\x{1}^2(\x{2}+\x{3})+
%\x{1}(\x{2}^2+\x{2}\x{3}+\x{3}^2)
%\right) \right)\\
%&=(1+c)c^2\left(\x{2}^3-c\left(
%\x{1}^2\x{2}+\x{1}\x{2}^2+
%\x{1}^3+
%\x{1}^3+\x{1}^2\x{2}+\x{1}\x{2}^2
%\right) \right)\\
&=(1+c)c^2 \x{2}^3,\\
D_{y_1-y_2}D_{y_1-y_3}^5(\x{1}^5\x{2}^3)
%&=(1+c)c^2  D_{y_1-y_2}(\x{2}^3)\\
&=(1+c)c^2\left( -3\x{2}^2-c\left( \frac{\x{2}^3-\x{2}^3}{\x{1}-\x{3}}-\frac{\x{2}^3-\x{3}^3}{\x{2}-\x{3}}
\right)\right)\\
%&=(1+c)c^2\left(\x{2}^2+ c\left( \x{2}^2+\x{2}\x{3}+\x{3}^2 \right)\right)\\
&=(1+c)c^2\left(c \x{1}^2+ c\x{1}\x{2}+(1+c)\x{2}^2 \right),\\
D_{y_1-y_2}^2D_{y_1-y_3}^5(\x{1}^5\x{2}^3)=
%&=(1+c)c^2D_{y_1-y_2} \left(c \x{1}^2+ c\x{1}\x{2}+(1+c)\x{2}^2 \right)\\
%&=(1+c)c^2 \left(c \x{1}+c\x{2}-c( c\x{1}+c\x{3}+c\x{2}+c\x{1}+(1+c)\x{2}+(1+c)\x{3}) \right)\\
%&=(1+c)c^2 \left(c \x{1}+c\x{2}+c\x{1}\right)\\
&=(1+c)c^3 \x{2},\\
D_{y_1-y_2}^3D_{y_1-y_3}^5(\x{1}^5\x{2}^3)
%&=(1+c)c^3 D_{y_1-y_2}(\x{2})\\
&=(1+c)^2c^3.
\end{align*}

So, $(1+c)^2c^3\in U$, and for $c\ne 0,1$ this implies $1\in U$. Using $S\h^*$ action this implies $U=N_{1,c}(\Triv)$ and shows that $N_{1,c}(\Triv)$ is irreducible.  
\end{proof}

\subsection{The irreducible representation $L_{0,c}(\Stand)$ in characteristic $p\ne 3$ for generic $c$}

In the following sections we aim to describe $L_{0,c}(\Stand)$ over an algebraically closed field of characteristic $2$, but the proofs work for any characteristic $p\ne 3$.

\begin{lemma}\label{lemma-matricesofDunklsStand1}
For any $p,t$ and $c$ the matrices of the Dunkl operators $D_{y_1-y_2}$ and $D_{y_2-y_3}$ on $M_{t,c}^1(\Stand)$ written in the bases $\left\{\x{1}\otimes \x{1},  \x{1}\otimes \x{2}, \x{2}\otimes \x{1},  \x{1}\otimes \x{2} \right\}$ for $M_{t,c}^1(\Stand)$ and $\left\{1\otimes \x{1}, 1\otimes \x{2}\right\}$ for $M_{t,c}^0(\Stand)$ are: 
\[\left[(D_{y_1-y_2})|_{M_{t,c}^1(\Stand)}\right]=
\begin{blockarray}{ccccc}
{\scriptstyle \x{1}\otimes \x{1}} & {\scriptstyle \x{1}\otimes \x{2}} &  {\scriptstyle \x{2}\otimes \x{1}} & {\scriptstyle \x{2}\otimes \x{2}}\\
\begin{block}{[cccc]c}
t+c &  -2c &-t+c & c & {\scriptstyle 1\otimes \x{1}} \\
-c &  t-c & 2c & -t-c & {\scriptstyle 1\otimes \x{2}} \\
\end{block}
\end{blockarray}
\]
\[\left[(D_{y_2-y_3})|_{M_{t,c}^1(\Stand)}\right]=
\begin{blockarray}{ccccc}
{\scriptstyle \x{1}\otimes \x{1}} & {\scriptstyle \x{1}\otimes \x{2}} &  {\scriptstyle \x{2}\otimes \x{1}} & {\scriptstyle \x{2}\otimes \x{2}}\\
\begin{block}{[cccc]c}
c & c & t-2c & c & {\scriptstyle 1\otimes \x{1}} \\
2c & -c & -c & t+2c & {\scriptstyle 1\otimes \x{2}} \\
\end{block}
\end{blockarray}
\] 
\end{lemma}
\begin{proof}
Direct computation.

\end{proof}

\begin{lemma}\label{pne3,t=0,cne0vectors}
\begin{enumerate}
\item For any $p$ the following vectors in $M_{0,c}^1(\Stand)$ are singular:
\begin{align*}
v_1&=-\x{1}\otimes \x{1}+\x{1}\otimes \x{2}+\x{2}\otimes \x{1}+2\x{2}\otimes \x{2}\\ v_2&= 2\x{1}\otimes \x{1}+\x{1}\otimes \x{2}+\x{2}\otimes \x{1}-\x{2}\otimes \x{2}.
\end{align*}
\item For $p\ne 3$ and $c\ne 0$, $v_1$ and $v_2$ span the set of singular vectors in $M_{0,c}^1(\Stand)$.
\end{enumerate}
\end{lemma}
\begin{proof}
As $\left\{y_1-y_2,y_2-y_3\right\}$ span $\h$, the intersection of the kernels of $D_{y}$ for all $y\in \h$ equals $\ker(y_1-y_2)\cap \ker(y_2-y_3)$. 
\begin{enumerate} 
\item This follows immediately from Lemma \ref{lemma-matricesofDunklsStand1} by setting $t=0$.
\item The determinant of the leftmost $2\times 2$ minor of $(D_{y_1-y_2})|_{M_{t,c}^1(\Stand)}$ equals
$(t+c)(t-c)-(-c)(-2c)=t^2-3c^2$, which is nonzero when $p\ne 3$, $t=0$ and $c\ne 0$. This shows that $(D_{y_1-y_2})|_{M_{t,c}^1(\Stand)}$ has rank $2$ so its kernel has dimension $2$. 
\end{enumerate}
\end{proof}

\begin{lemma}\label{pne3t=0cne0char}
Let $p\ne 3$, $t=0$, $c\ne 0$, and $v_1,v_2\in M_{0,c}(\Stand)$ from Lemma \ref{pne3,t=0,cne0vectors}.

\begin{enumerate}
\item The subrepresentation $\left<v_1,v_2\right>$  is isomorphic to $M_{0,c}(\Stand)[-1]$.  

\item The Hilbert series of the quotient $M_{0,c}(\Stand)/\left<v_1,v_2 \right>$ is $\frac{2(1-z)}{(1-z)^2}$. 

\item For $i=1,2$ the vectors $\overline{\sigma_2}\otimes \x{i}$ lie in $\left<v_1,v_2\right>$. 

\item For $i=1,2$ the vectors $\overline{\sigma_3}\otimes \x{i}$, $i=1,2$ do not lie in $\left<v_1,v_2\right>$. 

\item The quotient $M_{0,c}(\Stand)/\left<v_1,v_2,\overline{\sigma_3}\otimes \x{1},\overline{\sigma_3}\otimes \x{2} \right>$ 
has the character $$\chi_{M_{0,c}(\Stand)/\left<v_1,v_2,\overline{\sigma_3}\otimes \x{1},\overline{\sigma_3}\otimes \x{2}\right>}(z)=[\Stand]+([\Triv]+[\Sign])z+[\Stand]z^2$$
and the Hilbert polynomial 
$$h_{M_{0,c}(\Stand)/\left<v_1,v_2,\overline{\sigma_j}\otimes \x{i}, j=2,3, i=1,2\right>}(z)=2+2z+2z^2.$$ 

\item The quotient $M_{0,c}(\Stand)/\left<v_1,v_2,\overline{\sigma_3}\otimes \x{1},\overline{\sigma_3}\otimes \x{2}\right>$  is irreducible.

\end{enumerate}
\end{lemma}

\begin{proof}
\begin{enumerate}
\item The vectors $v_1,v_2$ are singular and span an $S_3$-subrepresentation isomorphic to the standard representation via the isomorphism $$\varphi:\Stand \to \mathrm{span}\left\{v_1,v_2\right\}$$ given by $\varphi(\x{i})=v_i.$
By the universal mapping property $\varphi$ extends to a homomorphism of graded $H_{0,c}(S_3)$ representations
$$\varphi:M_{0,c}(\Stand)[-1]\to M_{0,c}(\Stand),$$
given, for $f\in S\h^*$ and $i=1,2$, by $$\varphi(f\otimes \x{i})=f\cdot v_i.$$ 
We claim $\varphi$ is injective. 

Assume that some $v\in M_{0,c}(\Stand)$ is in the kernel of $\varphi$. Without loss of generality $v$ is homogeneous, so there are $A,B\in S\h^*$ homogeneous of the same degree such that 
$$v=A\otimes \x{1}+B\otimes \x{2}.$$
Applying $\varphi$ we get
$$Av_1+Bv_2=\varphi(v)=0,$$
which leads to
\begin{align*}
(-\x{1}+\x{2})A+(2\x{1}+\x{2})B & = 0 \\
(\x{1}+2\x{2})A+(\x{1}-\x{2})B & = 0.
\end{align*}
Considering this as a system of equations with unknowns $A,B\in S\h^*$, we calculate its determinant as
\begin{align*}
\det&=(-\x{1}+\x{2})(\x{1}-\x{2})-(2\x{1}+\x{2})(\x{1}+2\x{2})\\
&=-3(\x{1}^2+\x{1}\x{2}+\x{2}^2)=-3\overline{\sigma_2}. 
\end{align*}
Whenever $p\ne 3$, this is nonzero and the above system has only the trivial solution $A=B=0$, so the kernel of the $H_{0,c}(S_3)$ homomorphism $\varphi$ contains only $v=0$, and $\varphi$ is injective.

\item It immediately follows that 
$$h_{\left<v_1,v_2 \right>}(z)=z\cdot h_{M_{0,c}(\Stand)}(z)=\frac{2z}{(1-z)^2}$$
and $$h_{M_{0,c}(\Stand)/\left< v_1,v_2 \right>}(z) =\frac{2(1-z)}{(1-z)^2}.$$ 

\item Checking whether a vector $u=u_1\otimes \x{1}+u_2\otimes \x{2}$ is in the subrepresentation generated by $v_1,v_2$ is equivalent to solving the system 
$$Av_1+Bv_2=u,$$
for $A,B\in S\h^*$.
Reading off the coefficients of $\otimes \x{1}$ and $\otimes \x{2}$ we get this is equivalent to the system of equations
\begin{align*}
(-\x{1}+\x{2})A+(2\x{1}+\x{2})B & = u_1 \\
(\x{1}+2\x{2})A+(\x{1}-\x{2})B & = u_2.
\end{align*}
This system has a determinant $\det=-3\overline{\sigma_2}$. When $p\ne 3$, its unique solutions $A,B$, as rational functions on $\h$, as given as: 
\begin{align}
A&=\frac{u_1(\x{1}-\x{2})+u_2(2\x{1}+\x{2})}{-3\overline{\sigma_2}} \notag\\
B&=\frac{u_1(\x{1}+2\x{2})+u_2(\x{1}-\x{2})}{3\overline{\sigma_2}}. \label{systemp=2}
\end{align}
The question on whether $u=u_1\otimes \x{1}+u_2\otimes \x{2}$ is in the subrepresentation generated by $v_1,v_2$ is equivalent to checking whether these $A,B$ are polynomials in $\h^*$ (as opposed to rational functions).

When $u=\overline{\sigma_2}\otimes \x{1}$ we have $u_1=\overline{\sigma_2}$, $u_2=0$ so the solutions given by equations \eqref{systemp=2} are indeed in $S\h^*$, and simplify as 
$$A=\frac{-1}{3}(\x{1}-\x{2}), \quad 
B=\frac{1}{3}(\x{1}+2\x{2}). $$
This shows $\overline{\sigma_2}\otimes \x{1}\in \left< v_1,v_2\right>$, and $\overline{\sigma_2}\otimes \x{2}\in \left< v_1,v_2\right>$ follows by the action of $S_3$.

\item Analogously, to check whether $\overline{\sigma_3}\otimes \x{1}$, $i=1$ is in $\left< v_1,v_2\right>$, we substitute $u_1=\overline{\sigma_3}, u_2=0$ into \eqref{systemp=2}, getting 
\begin{align*}
A&=\frac{\overline{\sigma_3}(\x{1}-\x{2})}{-3\overline{\sigma_2}}=\frac{\x{1}\x{2}(\x{1}+\x{2})(\x{1}-\x{2})}{3(\x{1}^2+\x{1}\x{2}+\x{2}^2)} \\
B&=\frac{\overline{\sigma_3}(\x{1}+2\x{2})}{3\overline{\sigma_2}}=\frac{\x{1}\x{2}(\x{1}+\x{2})(\x{1}+2\x{2})}{-3(\x{1}^2+\x{1}\x{2}+\x{2}^2)}. 
\end{align*}
In characteristics other than $p=3$ the numerators and denominators of these rational functions are coprime, so $A,B$ are not polynomials. We conclude that $\overline{\sigma_3}\otimes \x{1}$ is not in the subrepresentation generated by $v_1,v_2$. Similar argument, or the action of $S_3$, gives that $\overline{\sigma_3}\otimes \x{2}$ is not in $\left< v_1,v_2\right>$.

\item 
Using (1), (3) and (4), and letting $O(z^4)$ denotes some series with terms of degree greater or equal to $4$, the first terms of this character are: 
\begin{align*}
&\chi_{M_{0,c}(\Stand)/\left<v_1,v_2,\overline{\sigma_j}\otimes \x{i}, j=2,3, i=1,2\right>}(z)=\\
&\quad =\chi_{M_{0,c}(\Stand)}(z)-\chi_{\left<v_1,v_2\right>}(z)-\chi_{\left<\overline{\sigma_3}\otimes \x{1},\sigma_3\otimes \x{1}\right>}(z)+O(z^4)\\
&\quad =\chi_{M_{0,c}(\Stand)}(z)-z\cdot \chi_{M_{0,c}(\Stand)}(z)-z^3\cdot \chi_{M_{0,c}(\Stand)}(z)+O(z^4)\\
&\quad=[\Stand]+([\Triv]+[\Sign])z+[\Stand]z^2+O(z^4).
\end{align*}

This shows the quotient is concentrated in degrees $0,1,2$, and its character and Hilbert polynomial are as claimed. 

\item For $p>3$, we could use \cite{DeSa14} Prop 4.1. to deduce that $L_{0,c}(\Stand)$ and $M_{0,c}(\Stand)/\left<v_1,v_2,\overline{\sigma_3}\otimes \x{1},\overline{\sigma_3}\otimes \x{2}\right>$ have the same Hilbert polynomial, and thus are equal. When $p=2$ the results of \cite{DeSa14} do not apply to $H_{t,c}(S_3,\h)$. Instead, we will prove that $M_{0,c}(\Stand)/\left<v_1,v_2,\overline{\sigma_3}\otimes \x{1},\overline{\sigma_3}\otimes \x{2}\right>$ is irreducible by a direct computation which works for all $p\ne 3$. 

If it is not irreducible, then it has singular vectors. These cannot be in degree $0$ (by definition) or degree $1$ by Lemma \ref{pne3,t=0,cne0vectors} (2), so we just examine degree $2$. 

Degree $2$ of $M_{0,c}(\Stand)$ has a basis 
$$\left\{ \x{1}^2\otimes \x{1}, \x{1}\x{2}\otimes \x{1}, \x{2}^2\otimes \x{1}, \x{1}^2\otimes \x{2}, \x{1}\x{2}\otimes \x{2}, \x{2}^2\otimes \x{2}\right\},$$
so $M^2_{0,c}(\Stand)/\left<v_1,v_2,\overline{\sigma_3}\otimes \x{1},\overline{\sigma_3}\otimes \x{2}\right>$ has a basis
$\left\{\x{1}^2\otimes \x{1}, \x{1}^2\otimes \x{2}\right\}.$
Calculate the action of a Dunkl operator $D_{y_1-y_2}$ in this basis as
\begin{align*}
D_{y_1-y_2}(\x{1}^2\otimes \x{1})
%&=-c\left(2(\x{1}+\x{2})\otimes \x{2}+(\x{1}+\x{3})\otimes \x{3}  \right)\\
%&=-c\left(2(\x{1}+\x{2})\otimes \x{2}+\x{2}\otimes \x{1}+\x{2}\otimes \x{2} \right)\\
&=-c\left(2\x{1}\otimes \x{2}+\x{2}\otimes \x{1}+3\x{2}\otimes \x{2} \right)\\
D_{y_1-y_2}(\x{1}^2\otimes \x{2})%&=-c\left(2(\x{1}+\x{2})\otimes \x{1}+(\x{1}+\x{3})\otimes \x{2}  \right)\\
%&=-c\left(2(\x{1}+\x{2})\otimes \x{1}-\x{2}\otimes \x{2} \right)\\
&=-c\left(2\x{1}\otimes \x{1}+2\x{2}\otimes \x{1}-\x{2}\otimes \x{2} \right).
\end{align*}
For $c\ne 0$ these are linearly independent so no nontrivial linear combination of $D_{y_1-y_2}(\x{1}^2\otimes \x{1})$ and $D_{y_1-y_2}(\x{1}^2\otimes \x{2})$ is zero.
This proves the statement.

\end{enumerate}
\end{proof}

\begin{comment}
Summarising the results of the last lemma, we get: 
\begin{cor}\label{cor-pne3,t=0,cne0vectors}
For any $p\ne 3$ and $c\ne 0$ the irreducible representation $L_{0,c}(\Stand)$ is the quotient of the Verma module $M_{0,c}(\Stand)$ by the submodule generated by vectors $v_1,v_2$ from Lemma \ref{pne3,t=0,cne0vectors} and $\overline{\sigma_3}\otimes \x{i}$, $i=1,2$. Its character and Hilbert series are: 
\begin{align*}
\chi_{L_{0,c}(\Stand)}(z)&=\begin{cases}
[\Stand]+2[\Triv]z+[\Stand]z^2, & p=2\\
[\Stand]+([\Triv]+[\Sign])z+[\Stand]z^2, & p>3
\end{cases} \\
h_{L_{0,c}(\Stand)}(z)&=2+2z+2z^2. 
\end{align*}
\end{cor}

\begin{proof}
By Lemma \ref{pne3t=0cne0char} part (6), $L_{0,c}(\Stand)$ is the quotient of the Verma module $M_{0,c}(\Stand)$ by the submodule generated by $v_1,v_2$ and $\overline{\sigma_3}\otimes \x{i}$, $i=1,2$. By Lemma \ref{pne3t=0cne0char} part (5), this module has the required character and Hilbert series. It remains only to observe that for $p=2$, $[\Triv]=[\Sign]$. 
\end{proof}
\end{comment}

\subsection{The irreducible representation $L_{1,c}(\Stand)$ in characteristic $2$ for generic $c$}
For this subsection, let $p=2$, $t=1$, and $c$ be generic. We will describe $L_{1,c}(\Stand)$, give the generators of the maximal proper graded submodule $J_{1,c}(\Stand)$ of the Verma modules $M_{1,c}(\Stand)$, calculate the character and Hilbert polynomial of $L_{1,c}(\Stand)$, and clarify again that $c$ generic here means $c\notin \mathbb{F}_2$. 

The strategy is similar to the strategy we use later for $L_{1,c}(\Stand)$ in characteristic $p>3$, but we need to treat the cases separately, because we cannot use the rescaled Young basis $b_+,b_-$ of $\h^*$ in characteristic $2$ (see Section \ref{sec-basisofh}).

The action of the Casimir element $\Omega$ tells us where to look for singular vectors in this case, by the following Lemma. 

\begin{lemma}\label{p=2Omega}
Let $p=2$, $M$ be an $H_{t,c}(S_3,\h)$ module equal to the Verma module $M_{1,c}(\Stand)$ or its quotient, and $\tau\subseteq M^k$ an irreducible $S_3$-subrepresentation contained in the kernel of all Dunkl operators. 
\begin{enumerate}
\item The action of the Casimir element $\Omega$ on $\tau$ is by a scalar: 
$$\Omega|_\Triv=0\cdot \mathrm{id}, \quad 
\Omega|_\Stand=c \cdot \mathrm{id}.$$
\item If $c\notin \mathbb{F}_2$ then $\tau=\Stand$ and $k$ is even. 
\item If $c=1$ then either $\tau=\Stand$ and $k$ is even, or $\tau=\Triv$ and $k$ is odd. 
\end{enumerate}
\end{lemma}
\begin{proof}
\begin{enumerate}
\item This is Lemma \ref{OmegaAction} for $p=2$. 
\item The action of $\Omega$ on $M_{1,c}^k(\Stand)$ and any of its quotients is by a scalar $c+k$. Simultaneously part (1) gives the constant for the action on $\tau\subseteq M^k$. If $\tau=\Triv$, comparing these constants leads to 
$c+k=0$
which has no solutions for $k\in \mathbb{N}_0$, $c\notin \mathbb{F}_2$.
If $\tau=\Stand$, this leads to 
$c+k=c$
which implies $k$ is even.
\item Similarly, for $\tau=\Triv$ we get an equation $c+k=0$ which is satisfied when $c=1$ and $k\in \mathbb{N}$ is odd, and for $\tau=\Stand$ we get an equation $c+k=c$ which is satisfied when $c=1$ and $k\in \mathbb{N}$ is even.
\end{enumerate}
\end{proof}

We now look for singular vectors for generic $c$, knowing they will be in the $\Stand$ isotypic components of $M_{1,c}^{2k'}(\Stand)$. The first ones we find are in degree $2$. 

\begin{lemma}\label{p=3,t=1,cne1vectors}
For every $c$ the vectors 
\begin{align*}
v_1&=c\s{2}\otimes \x{1}+(\x{1}^2\otimes \x{1}+\x{1}^2\otimes \x{2}+\x{2}^2\otimes \x{1})\\
%&=(c+1)\x{1}^2\otimes \x{1}+c\x{1}\x{2}\otimes \x{1}+(c+1)\x{2}^2\otimes \x{1}+\x{1}^2\otimes \x{2}\\
v_2&=c\s{2}\otimes \x{2}+(\x{2}^2\otimes \x{2}+\x{1}^2\otimes \x{2}+\x{2}^2\otimes \x{1})
%&=\x{2}^2\otimes \x{1}+(c+1)\x{1}^2\otimes \x{2}+c\x{1}\x{2}\otimes \x{2}+(c+1)\x{2}^2\otimes \x{2}.
\end{align*}
in $M_{1,c}^2(\Stand)$ span an $S_3$ representation isomorphic to $\Stand$ and are singular. 
\end{lemma}
\begin{proof}
By Theorem \ref{decmposeVermaStandp=2} the map $\x{i}\mapsto v_i$ is an $S_3$ isomorphism, so $v_1,v_2$ span a copy of $\Stand$. To see it consists of singular vectors, we first directly calculate
$$D_{y_1-y_3}(v_1+v_2)=0,$$
and from here
$$
D_{y_2-y_3}(v_1+v_2)=s_1D_{y_1-y_3}s_1(v_1+v_2)=s_1D_{y_1-y_3}(v_1+v_2)=s_1.0=0,$$
showing $v_1+v_2$ is singular. Then $v_2=s_2.(v_1+v_2)$ and $v_1=s_1.v_2$ are singular too.
\end{proof}

We will now need to investigate what the $H_{1,c}(S_3,\h)$ subrepresentation $\left<v_1,v_2 \right>$ looks like. Checking whether some vector $u=u_1\otimes \x{1}+u_2\otimes \x{2}\in M^k_{1,c}(\Stand)$ is in $\left<v_1,v_2 \right>$ amounts to solving the system 
$$Av_1+Bv_2=u,$$
for $A,B\in S^{k-2}\h^*$.
Reading off the coefficients of $\otimes \x{1}$ and $\otimes \x{2}$ we get 
\begin{align}
((c+1)\x{1}^2+c\x{1}\x{2}+(c+1)\x{2}^2)A+\x{2}^2 B & = u_1 \notag \\
\x{1}^2A+((c+1)\x{1}^2+c\x{1}\x{2}+(c+1)\x{2}^2)B & = u_2.\label{systemp=2,t1}
\end{align}

\begin{lemma}\label{p=3,t=1,cne1det}
For $p=2$, the determinant of the system \eqref{systemp=2,t1} is 
$$\det=(c+1)^2\overline{\sigma_2}^2.$$
When $c\ne 1$, the unique rational functions $A,B$ on $\h$ solving that system are: 
\begin{align}
A&=\frac{u_1((c+1)\x{1}^2+c\x{1}\x{2}+(c+1)\x{2}^2)-u_2\x{2}^2}{(c+1)^2\overline{\sigma_2}^2} \notag\\
B&=\frac{-u_1\x{1}^2+u_2((c+1)\x{1}^2+c\x{1}\x{2}+(c+1)\x{2}^2)}{(c+1)^2\overline{\sigma_2}^2}. \label{systemp=2,t1,sol}
\end{align}
\end{lemma}
\begin{proof}
We calculate directly: 
\begin{align*}
\det&=((c+1)\x{1}^2+c\x{1}\x{2}+(c+1)\x{2}^2)^2-\x{1}^2\x{2}^2\\
%&=((c+1)\x{1}^2+(c+1)\x{1}\x{2}+(c+1)\x{2}^2)^2\\
&=(c+1)^2(\x{1}^2+\x{1}\x{2}+\x{2}^2)^2=(c+1)^2\overline{\sigma_2}^2.
\end{align*}
\end{proof}

We proceed in a way analogous to Lemma \ref{pne3t=0cne0char}. 

\begin{lemma}
When $p=2$, $t=1$ and $c\ne 1$, the vectors $v_1,v_2\in M_{1,c}(\Stand)$ from Lemma \ref{p=3,t=1,cne1vectors}
generate a subrepresentation of $M_{1,c}(\Stand)$ isomorphic to $M_{1,c}(\Stand)[-2]$.  The Hilbert series of the quotient $M_{1,c}(\Stand)/\left<v_1,v_2\right>$ is $\frac{2(1-z^2)}{(1-z)^2}$. 
\end{lemma}
\begin{proof}
The isomorphism of $S_3$ representations $$\varphi:\Stand \to \mathrm{span}\left\{v_1,v_2\right\}$$ given by $\varphi(\x{i})=v_i$ extends to a homomorphism of $H_{1,c}(S_3,\h)$ modules 
$$\varphi:M_{1,c}(\Stand)[-2]\to M_{1,c}(\Stand).$$
Assuming some nonzero vector $A\x{1}+B\x{2}$ is in the kernel of $\varphi$ is equivalent to assuming there is a nontrivial solution $A,B\in S\h^*$ to the equation 
$$Av_1+Bv_2=0.$$
This is equivalent to there existing a nontrivial solution to the system \eqref{systemp=2,t1} with $u_1=u_2=0$, which is impossible by Lemma \ref{p=3,t=1,cne1det}. 
\end{proof}

Next, we consider $M_{1,c}(\Stand)/\left<v_1,v_2\right>$ and check whether the $p$-th powers of the invariants give us a proper submodule of it. 

\begin{lemma}\label{p=2t=1cgenHilb1}
Let $p=2$, $t=1$, $c\ne 1$. For $i=1,2$ we  have 
$$\overline{\sigma_2}^2\otimes \x{i} \in \left<v_1,v_2 \right>, \qquad \overline{\sigma_3}^2\otimes \x{i} \notin \left<v_1,v_2 \right>.$$

\end{lemma}
\begin{proof}
A vector $\overline{\sigma_j}^2\otimes \x{i}$, $i=1,2$, $j=2,3$ is in the submodule generated by $v_1,v_2$ if and only if the system \eqref{systemp=2,t1} with $u_i=\overline{\sigma_j}^2$, $u_k=0$ for $k\ne i$ has polynomial solutions $A,B\in S\h^*$. Lemma \ref{p=3,t=1,cne1det} gives these solutions explicitly as rational functions. If $j=2$, these rational functions are in fact polynomial, as the factors $\overline{\sigma_j}^2$ in the numerator and the denominator cancel. If $j=3$, these rational functions are not polynomial, as the numerator and the denominator are coprime.
\end{proof}

Next we consider $M_{1,c}(\Stand)/\left<v_1,v_2,\overline{\sigma_3}^2\otimes \x{1}, \overline{\sigma_3}^2\otimes \x{2}\right>$. First we use Lemma \ref{p=2t=1cgenHilb1} to calculate its character, and then we will show it is irreducible. 

\begin{lemma}\label{p=2,t=1,Stand,char}
Let $p=2$, $t=1$, $c\ne 1$. 
The module $M_{1,c}(\Stand)/\left<v_i, \overline{\sigma_3}^2\otimes \x{i}, i=1,2 \right>$ has the character 
%$$[\Stand]+([\Stand]+2[\Triv])z+([\Stand]+2[\Triv])z^2+2[\Stand]z^3+$$
%$$+([\Stand]+2[\Triv])z^4+([\Stand]+2[\Triv])z^5+[\Stand]z^6$$
$$\chi_{M_{1,c}(\Stand)/\left<v_i, \overline{\sigma_3}^2\otimes \x{i}, i=1,2 \right>}(z)=\chi_{S^{(2)}\h^*}(z) \left([\Stand](1+z^4)+2[\Triv]z^2   \right)$$
and the Hilbert polynomial 
$$h_{M_{1,c}(\Stand)/\left<v_i, \overline{\sigma_3}^2\otimes \x{i}, i=1,2 \right>}(z)=2\frac{(1-z^2)(1-z^6)}{(1-z)^2}.$$ 
\end{lemma}
\begin{proof}
By Lemma \ref{p=2t=1cgenHilb1}, the quotient $M_{1,c}(\Stand)/\left<v_1,v_2\right>$ has the character 
$$\chi_{M_{1,c}(\Stand)}(z)\cdot(1-z^2).$$
Using Lemma \ref{p=2t=1cgenHilb1}%, 
%$$\left<v_1,v_2\right>=\left<v_1,v_2, \sigma_2^2\otimes \x{i}, i=1,2\right>$$
%and 
%$$\left<v_1,v_2\right>\ne \left<v_1,v_2, \sigma_3^2\otimes \x{i}, i=1,2\right>.$$
%So, 
\begin{align*}
\chi&_{M_{1,c}(\Stand)/\left<v_1,v_2,\overline{\sigma_3}^2\otimes \x{1}, \overline{\sigma_3}^2\otimes \x{2} \right>}(z)=\\
&=\chi_{M_{1,c}(\Stand)}(z)-\chi_{\left<v_1,v_2\right>}(z)-\chi_{\left<\overline{\sigma_3}^2\otimes \x{1}, \overline{\sigma_3}^2\otimes \x{2} \right>}(z)+\chi_{\left<v_1,v_2\right>\cap \left<\overline{\sigma_3}^2\otimes \x{1}, \overline{\sigma_3}^2\otimes \x{2} \right>}(z)\\
&=\chi_{M_{1,c}(\Stand)}(z)-z^2\chi_{M_{1,c}(\Stand)}(z)-z^6\chi_{M_{1,c}(\Stand)}(z)+\chi_{\left<v_1,v_2\right>\cap \left<\overline{\sigma_3}^2\otimes \x{1}, \overline{\sigma_3}^2\otimes \x{2} \right>}(z).
\end{align*}

The remaining task is to describe the submodule $\left<v_1,v_2\right>\cap \left<\overline{\sigma_3}^2\otimes \x{1}, \overline{\sigma_3}^2\otimes \x{2} \right>$ in enough detail. We first claim that it contains no vectors in degree $7$. 

This intersection is a submodule of $\left<\overline{\sigma_3}^2\otimes \x{1}, \overline{\sigma_3}^2\otimes \x{2} \right>$, which is isomorphic to a quotient of $M_{1,c}(\Stand)[-6]$. If the intersection was nonzero in graded degree $7$, then $M_{1,c}(\Stand)[-6]$ would have a nontrivial submodule starting in degree $7$, so $M_{1,c}(\Stand)$ would have a nontrivial submodule starting in degree $1$. This means there would be a singular vector in degree $1$ of $M_{1,c}(\Stand)$, which is impossible by Lemma \ref{p=2Omega}.

Thus, 
\begin{align*}
\chi&_{M_{1,c}(\Stand)/\left<v_1,v_2,\overline{\sigma_j}^2\otimes \x{i}, j=2,3,i=1,2 \right>}(z)\\
&=\chi_{M_{1,c}(\Stand)}(z)-z^2\chi_{M_{1,c}(\Stand)}(z)-z^6\chi_{M_{1,c}(\Stand)}(z)+O(z^8),
\end{align*}
where $O(z^8)$ denotes a polynomial with terms of degrees $8$ and more. 
Calculating the coefficient of $z^7$ in that expression, we get it is equal to $0$. This means the module $M_{1,c}(\Stand)/\left<v_1,v_2,\overline{\sigma_j}^2\otimes \x{i}, j=2,3,i=1,2 \right>$ is concentrated in degrees $0,\ldots, 6$, and its character is equal to 
the truncation of $(1-z^2-z^6)\chi_{M_{1,c}(\Stand)}(z)$ at $z^7$, which can be easily calculated to be
as claimed. The Hilbert polynomial follows.

\end{proof}

It remains to see that $M_{1,c}(\Stand)/\left< v_1,v_2,\overline{\sigma_3}^2\otimes \x{1}, \overline{\sigma_3}^2\otimes \x{2}\right>$ is irreducible, which we do by showing it has no singular vectors. By Lemma \ref{p=2Omega}, we only need to examine its even degrees, and only look for singular vectors in the $\Stand$ isotypic components.

\begin{lemma}\label{p=2,t=1,Stand,cgen,deg2}
Let $p=2$, $t=1$, $c\ne 1$. 
The module $$M_{1,c}(\Stand)/\left< v_1,v_2,\overline{\sigma_3}^2\otimes \x{1}, \overline{\sigma_3}^2\otimes \x{2} \right>$$ has no singular vectors of type $\Stand$ in degree $2$.  
\end{lemma}
\begin{proof}
By Theorem \ref{decmposeVermaStandp=2}, the $\Stand$ isotypic component of $M^2_{1,c}(\Stand)$ has a basis:
$$\s{2}\otimes \x{1}, \, \s{2}\otimes \x{2},$$
$$\x{1}^2\otimes \x{1}+\x{1}^2\otimes \x{2}+\x{2}^2\otimes \x{1},\,\x{2}^2\otimes \x{2}+\x{1}^2\otimes \x{2}+\x{2}^2\otimes \x{1}.$$

Let us consider just the part of these copies of $\Stand$ corresponding to $\x{1}$ (this is enough as $\Stand$ is irreducible), so the part spanned by 
$$\s{2}\otimes \x{1}, \,\x{1}^2\otimes \x{1}+\x{1}^2\otimes \x{2}+\x{2}^2\otimes \x{1}.$$
These vectors satisfy 
\begin{align*}
D_{y_1-y_3}(\s{2}\otimes \x{1})&=\x{2}\otimes \x{1}\\
D_{y_1-y_3}(\x{1}^2\otimes \x{1}+\x{1}^2\otimes \x{2}+\x{2}^2\otimes \x{1})&=c\x{2}\otimes \x{1}.
\end{align*}
Their only linear combination which is singular is 
$$v_1=c\cdot \s{2}\otimes \x{1}+ (\x{1}^2\otimes \x{1}+\x{1}^2\otimes \x{2}+\x{2}^2\otimes \x{1}),$$
so there are no singular vectors in $M_{1,c}^2(\Stand)/\left< v_i,\overline{\sigma_3}^2\otimes \x{i}, i=1,2 \right>$.

\end{proof}

\begin{lemma}\label{p=2,t=1,Stand,cgen,deg4}
Let $p=2$, $t=1$, $c\ne 1$. 
The module $$M_{1,c}(\Stand)/\left< v_1,v_2,\overline{\sigma_3}^2\otimes \x{1}, \overline{\sigma_3}^2\otimes \x{2} \right>$$ has no singular vectors of type $\Stand$ in degree $4$.  
\end{lemma}
\begin{proof}
By Theorem \ref{decmposeVermaStandp=2}, the part of the $\Stand$ isotypic component of $M^4_{1,c}(\Stand)$ corresponding to $\x{1}$ has a basis
$$\s{2}^2\otimes \x{1}, \,\s{3}(\x{1}\otimes \x{1}+\x{1}\otimes \x{2}+\x{2}\otimes \x{1}), \, \s{2}(\x{1}^2\otimes \x{1}+\x{1}^2\otimes \x{2}+\x{2}^2\otimes \x{1}).$$

Using Theorem \ref{decmposeVermaStandp=2} again, we see that the submodule $\left< v_1,v_2 \right>$ which we quotient by contains the following vectors in the $\Stand$ isotypic component of degree $4$:
\begin{align*}
\sigma_2v_1&=c\cdot \s{2}^2\otimes \x{1}+ \s{2}(\x{1}^2\otimes \x{1}+\x{1}^2\otimes \x{2}+\x{2}^2\otimes \x{1})\\
\x{1}^2v_1+\x{2}^2v_1+\x{1}^2v_2&=\s{2}^2\otimes \x{1}+ c\cdot  \s{2}(\x{1}^2\otimes \x{1}+\x{1}^2\otimes \x{2}+\x{2}^2\otimes \x{1}). 
\end{align*}
As $c\ne 1$, these two vectors are linearly independent, and so $\s{2}^2\otimes \x{1}$ and  $\s{2}(\x{1}^2\otimes \x{1}+\x{1}^2\otimes \x{2}+\x{2}^2\otimes \x{1})$ lie in $\left< v_1,v_2 \right>$, while $\s{3}(\x{1}\otimes \x{1}+\x{1}\otimes \x{2}+\x{2}\otimes \x{1})$ does not and spans the $\x{1}$ part of the $\Stand$ isotypic component of $M_{1,c}^4(\Stand)/\left< v_1,v_2,\overline{\sigma_j}^2\otimes \x{i} \right>$.

It remains to check if  $\s{3}(\x{1}\otimes \x{1}+\x{1}\otimes \x{2}+\x{2}\otimes \x{1})$ is singular in the quotient $M_{1,c}(\Stand)/\left< v_1,v_2,\overline{\sigma_3}^2\otimes \x{1}, \overline{\sigma_3}^2\otimes \x{2} \right>.$ We calculate
\begin{align*}
D_{y_1-y_3}(\s{3}(\x{1}\otimes \x{1}+\x{1}\otimes \x{2}+\x{2}\otimes \x{1}))&=\x{2}^3\otimes \x{1}+\x{1}^2\x{2}\otimes \x{1}+\x{1}^2\x{2}\otimes \x{2},
\end{align*}
which can be shown either directly or using Lemma \ref{p=3,t=1,cne1det} to not lie in $\left< v_1,v_2 \right>$, so is not zero in the quotient $M_{1,c}(\Stand)/\left< v_1,v_2,\overline{\sigma_3}^2\otimes \x{1}, \overline{\sigma_3}^2\otimes \x{2}\right>$.

\end{proof}

\begin{lemma}\label{p=2,t=1,Stand,cgen,deg6}
Let $p=2$, $t=1$, $c\ne 1$. 
The module $$M_{1,c}(\Stand)/\left< v_1,v_2,\overline{\sigma_3}^2\otimes \x{1}, \overline{\sigma_3}^2\otimes \x{2} \right>$$ has no singular vectors of type $\Stand$ in degree $6$.  
\end{lemma}
\begin{proof}
Using Theorem \ref{decmposeVermaStandp=2}, the part of the $\Stand$ isotypic component of $M^6_{1,c}(\Stand)$ corresponding to $\x{1}$ has a basis
\begin{align*}
A&=\s{2}\s{3}(\x{1}\otimes \x{1}+\x{1}\otimes \x{2}+\x{2}\otimes \x{1})\\
B&=\s{2}^2(\x{1}^2\otimes \x{1}+\x{1}^2\otimes \x{2}+\x{2}^2\otimes \x{1})\\
C&=\s{2}^3\otimes \x{1}\\
D&=\s{3}^2\otimes \x{1}\\
E&=\s{3}(\x{1}^3+\x{1}^2\x{2}+\x{2}^3)\otimes \x{1}.
\end{align*}

Taking a quotient by $\s{3}^2\otimes \x{i}$ annihilates $\s{3}^2\otimes \x{1}=D,$ taking a quotient by $\left<v_1,v_2 \right>$ annihilates (again using Theorem \ref{decmposeVermaStandp=2})
\begin{align*}
\s{2}^2v_1&=B+c\cdot C\\
\s{3}(\x{1}v_1+\x{2}v_1+\x{1}v_2)&=c\cdot A + E\\
\s{2}(\x{1}^2v_1+\x{2}^2v_1+\x{1}^2v_2)&=c\cdot B + C.
\end{align*}
and the resulting basis for the $\x{1}$ part of the $\Stand$ isotypic component in the quotient $M^6_{1,c}(\Stand)/\left< v_1,v_2,\overline{\sigma_j}^2\otimes \x{i}\right>$ is $A$. To check that it is not singular in this quotient, we calculate
$$D_{y_1-y_3}(A)=\x{1}^4\x{2}\otimes (\x{1}+\x{2}),$$
which can be shown directly or using Lemma \ref{p=3,t=1,cne1det} does not lie in the submodule $\left< v_1,v_2 \right>$.

\end{proof}

\begin{lemma}\label{p=2,t=1,Stand,cgen-endproof}
Let $p=2$, $t=1$, $c\ne 0,1$. 
The quotient of the Verma module $M_{1,c}(\Stand)$ by the submodule generated by $v_1,v_2$ from Lemma \ref{p=3,t=1,cne1vectors}, $\overline{\sigma_3}^2\otimes \x{1}$ and $\overline{\sigma_3}^2\otimes \x{2}$ is irreducible, and thus equal to $L_{1,c}(\Stand)$. 
\end{lemma}
\begin{proof}

By Lemma \ref{p=2Omega}, $M_{1,c}(\Stand)/\left< v_1,v_2, \overline{\sigma_3}^2\otimes \x{1},\overline{\sigma_3}^2\otimes \x{2} \right>$ has no singular vectors in odd degrees, and only singular vectors of type $\Stand$ in even degrees. By Lemma \ref{p=2,t=1,Stand,char}, the only positive even degrees of this module are $2,4$ and $6$, in which there are no singular vectors of type $\Stand$ by Lemmas \ref{p=2,t=1,Stand,cgen,deg2}, \ref{p=2,t=1,Stand,cgen,deg4} and \ref{p=2,t=1,Stand,cgen,deg6}. Hence, this module has no singular vectors and is irreducible. 

\end{proof}

\subsection{The irreducible representation $L_{1,c}(\Stand)$ in characteristic $2$ for $c\in \mathbb{F}_2$}

The previous section worked out the characters of $L_{1,c}(\Stand)$ for all $c\ne 0,1$. For $c=0$, this character is computed in Lemma \ref{t=1,c=0}. In this section we compute it for $c=1$. For the entire section, let $p=2$, $t=1$, $c=1$. The aim of the section is to prove the following lemma.

\begin{lemma}\label{p=2,t=1,Stand,c=1}
Let $p=2$, $t=1$, $c=1$. The irreducible representation $L_{1,1}(\Stand)$ is the quotient of the Verma module $M_{1,1}(\Stand)$ by the submodule generated by the following vectors: 
\begin{align*}
v_1&=\x{1}\otimes \x{2}+\x{2}\otimes \x{1}\\
v_3&=\s{2}(\x{1}\otimes \x{1}+\x{2}\otimes \x{1}+\x{2}\otimes \x{2})\\
v_5&=\s{3}(\x{1}^2\otimes \x{2}+\x{2}^2\otimes \x{1})\\
v_7&=\s{2}\s{3}(\x{1}^2\otimes \x{1}+\x{2}^2\otimes \x{1}+\x{2}^2\otimes \x{2}).
\end{align*}
The character of $L_{1,1}(\Stand)$ is 
$$\chi_{L_{1,1}(\Stand)}(z)=[\Stand](1+z+z^2+2z^3+z^4+z^5+z^6)+[\Triv](z+2z^2+2z^4+z^5)$$
and its Hilbert polynomial is
$$h_{L_{1,1}(\Stand)}(z)=2+3z+4z^2+4z^3+4z^4+3z^5+2z^6=\frac{1-z-z^3-z^5-z^7+2z^8}{(1-z)^2}.$$
\end{lemma}

We prove this in the series of lemmas, considering one degree of $M_{1,1}(\Stand)$ at a time, and calculating the character one term at a time. 

\begin{lemma}\label{p=2,t=1,Stand,c=1,k=1}
Let $p=2$. The space of singular vectors in $M^1_{1,1}(\Stand)$ is spanned by
$$v_1=\x{1}\otimes \x{2}+\x{2}\otimes \x{1}.$$
Hence, the first terms of the character of $L_{1,1}(\Stand)$ are  
$$\chi_{L_{1,1}(\Stand)}(z)=[\Stand](1+z+\ldots )+[\Triv](z+\ldots).$$
\end{lemma}
\begin{proof}
Follows directly from Lemma \ref{lemma-matricesofDunklsStand1}, by substituting $t=c=1$ into the matrices of Dunkl operators calculated there and looking for the intersection of their kernels. 
\end{proof}

\begin{lemma}\label{p=2,t=1,Stand,c=1,k=2}
Let $p=2$. There are no singular vectors in $M^2_{1,1}(\Stand)/\left<v_1 \right>$.
The first terms of the character of $L_{1,1}(\Stand)$ are  
$$\chi_{L_{1,1}(\Stand)}(z)=[\Stand](1+z+z^2\ldots )+[\Triv](z+2z^2\ldots).$$
\end{lemma}
\begin{proof}
By Lemma \ref{p=2Omega} we only need to inspect the $\Stand$ isotypic component for singular vectors. The part of it corresponding to $\x{1}$ in $M^2_{1,1}(\Stand)$ has a basis 
$$\s{2}\otimes \x{1}, \quad \x{1}^2\otimes \x{1}+\x{1}^2\otimes \x{2}+\x{2}^2\otimes \x{1}.$$ Taking the quotient by $\left<v_1\right>$ annihilates the vector 
$$x_1v_1=(\s{2}\otimes \x{1})+ (\x{1}^2\otimes \x{1}+\x{1}^2\otimes \x{2}+\x{2}^2\otimes \x{1}).$$
So, it is enough to check that $\s{2}\otimes \x{1}$ is not singular. We calculate
$$D_{y_1-y_3}(\s{2}\otimes \x{1})=\x{2}\otimes \x{1}\notin \left<v_1\right>,$$
which proves the claim. 
\end{proof}

\begin{lemma}\label{p=2,t=1,Stand,c=1,k=3}
Let $p=2$. The space of singular vectors in $M_{1,1}^3(\Stand)/\left<v_1 \right>$ is spanned by
$$v_3=\s{2}(\x{1}\otimes \x{1}+\x{2}\otimes \x{1}+\x{2}\otimes \x{2}).$$
The character of $M_{1,1}(\Stand)/\left<v_1,v_3 \right>$ is
$$\chi_{M_{1,1}(\Stand)/\left<v_1,v_3 \right>}(z)=\chi_{M_{1,1}(\Stand)}(z)-(z+z^3)\chi_{M_{1,1}(\Triv)}(z).$$
The first terms of the character of $L_{1,1}(\Stand)$ are
$$\chi_{L_{1,1}(\Stand)}(z)=[\Stand](1+z+z^2+2z^3+\ldots)+[\Triv](z+2z^2+0\cdot z^3+\ldots).$$
\end{lemma}
\begin{proof}
By Theorem \ref{decmposeVermaStandp=2}, in the Grothendieck group $K_{0}(S_3)$ 
$$M^3_{1,1}(\Stand)=3[\Stand]+2[\Triv],$$ 
and, as $\left<v_1\right>\cong M_{1,1}(\Triv)[-1]$, we similarly get 
$$[\left<v_1\right>^3]=[M^2_{1,1}(\Triv)]=[\Stand]+[\Triv].$$ 
Consequently, 
$$M^3_{1,1}(\Stand)/\left<v_1\right>=2[\Stand]+[\Triv].$$

By Lemma \ref{p=2Omega} we only need to inspect the $\Triv$ isotypic component of $M^3_{1,1}(\Stand)/\left<v_1\right>$ for singular vectors, which by Theorem \ref{decmposeVermaStandp=2} has a basis $v_3=\s{2}(\x{1}\otimes \x{1}+\x{2}\otimes \x{1}+\x{2}\otimes \x{2})$. 
We calculate:
\begin{align*}
D_{y_1-y_3}(v_3)&=\x{2}(\x{1}\otimes \x{1}+\x{2}\otimes \x{1}+\x{2}\otimes \x{2})+\s{2}\otimes \x{1}+\\
&+\s{2}\otimes \x{1}+\s{2}\otimes \x{1}+\s{2}\otimes \x{1}+\s{2}\otimes \x{2}\\
&=(\x{1}\x{2}+\x{2}^2)\otimes \x{1}+(\x{1}^2+\x{1}\x{2})\otimes \x{2}\\
&=(\x{1}+\x{2})v_1\in \left<v_1\right>,\\
D_{y_2-y_3}(v_3)&=s_1.D_{y_1-y_3}s_1.(v_3)\\
&=s_1.D_{y_1-y_3}(v_3+\s{2}v_1)\in \left<v_1\right>.
\end{align*}
So, $v_3$ is singular in $M_{1,1}(\Stand)/\left<v_1 \right>$ (though not in $M_{1,1}(\Stand)$).

We have $\left<v_1\right>\cong M_{1,1}(\Triv)[-1]$, $\left<v_3\right>\cong M_{1,1}(\Triv)[-3]$, and $\left<v_1\right>\cap \left<v_3\right>=0$ (by direct computation). This implies the claims about the characters. 

\end{proof}

\begin{lemma}\label{p=2,t=1,Stand,c=1,k=4}
Let $p=2$.  There are no singular vectors in $M^4_{1,1}(\Stand)/\left<v_1,v_3 \right>$.
Hence, the first terms of the character of $L_{1,1}(\Stand)$ are  
$$\chi_{L_{1,1}(\Stand)}(z)=[\Stand](1+z+z^2+2z^3+z^4+\ldots)+[\Triv](z+2z^2+2z^4+\ldots).$$
\end{lemma}
\begin{proof}

Assume there was a singular vector in $M^4_{1,1}(\Stand)/\left<v_1,v_3 \right>$. By Lemma \ref{p=2Omega} it is in a $\Stand$ isotypic component. It generates a nonzero submodule $V$ of $M_{1,1}(\Stand)/\left<v_1,v_3 \right>$, which is a quotient of $M_{1,1}(\Stand)[-4]$ and thus $L_{1,1}(\Stand)[-4]$ is a quotient of $V$. Thus, $V^6$ contains at least as many copies of $\Triv$ as $L_{1,1}^2(\Stand),$ which is $2$ by Lemma \ref{p=2,t=1,Stand,c=1,k=3}. However, $V$ is a submodule of $M_{1,1}(\Stand)/\left<v_1,v_3 \right>$ which also by Lemma \ref{p=2,t=1,Stand,c=1,k=3} contains no copies of $\Triv$. This is impossible, so we conclude there is no singular vector in $M^4_{1,1}(\Stand)/\left<v_1,v_3 \right>$.

\end{proof}

\begin{lemma}\label{p=2,t=1,Stand,c=1,k=5}
Let $p=2$. The space of singular vectors in $M^5_{1,1}(\Stand)/\left<v_1, v_3 \right>$ is spanned by
$$v_5=\s{3}(\x{1}^2\otimes \x{2}+\x{2}^2\otimes \x{1}).$$
Hence, the first terms of the character of $L_{1,1}(\Stand)$ are  
$$\chi_{L_{1,1}(\Stand)}(z)=[\Stand](1+z+z^2+2z^3+z^4+z^5+\ldots)+[\Triv](z+2z^2+2z^4+z^5+\ldots).$$
\end{lemma}
\begin{proof}
By Lemma \ref{p=2Omega}, any singular vectors in degree $5$ are of type $\Triv$, so we only examine this. Using Theorem \ref{decmposeVermaStandp=2}, a basis for the $\Triv$ component of $M^5_{1,1}(\Stand)$ is 
\begin{align*}
A&=\s{2}^2(\x{1}\otimes \x{2}+\x{2}\otimes \x{1})=\s{2}^2v_1\\
B&=\s{2}^2(\x{1}\otimes \x{1}+\x{2}\otimes \x{1}+\x{2}\otimes \x{2})=\s{2}v_3\\
C&=\s{3}(\x{1}^2\otimes \x{2}+\x{2}^2\otimes \x{1})\\
D&=\s{3}(\x{1}^2\otimes \x{1}+\x{2}^2\otimes \x{1}+\x{2}^2\otimes \x{2}).
\end{align*}
Here $A,C$ span submodules, and $B,D$ their extensions. Quotienting by $\left< v_1, v_3\right>$ annihilates $A$ and $B$. The only submodule of type $\Triv$ is spanned by $C$, so we set $v_5=C$ and check that is singular: 
\begin{align*}
D_{y_1-y_3}(v_5)&=\x{2}\s{2}v_1\in \left<v_1 \right>\\
D_{y_2-y_3}(v_5)&=s_1.\x{2}\s{2}v_1=\x{1}\s{2}v_1\in \left<v_1 \right>.
\end{align*}
This proves that $v_5$ is singular, and shows that the first terms of the character of $M_{1,1}(\Stand)/\left<v_1,v_3,v_5 \right>$ are as stated in the Lemma. It remains to show that $D$ is not singular, so we calculate
\begin{align*}
D_{y_1-y_3}(D)&=\s{2}\x{1}\otimes \x{1}+\s{3}v_1 \notin \left<v_1,v_3,v_5\right>.
\end{align*}
\end{proof}

\begin{lemma}\label{p=2,t=1,Stand,c=1,k=6}
Let $p=2$.  There are no singular vectors in $M^6_{1,1}(\Stand)/\left<v_1, v_3,v_5 \right>$.
Hence, the first terms of the character of $L_{1,1}(\Stand)$ are  
$$\chi_{L_{1,1}(\Stand)}(z)=[\Stand](1+z+z^2+2z^3+z^4+z^5+z^6+\ldots)+[\Triv](z+2z^2+2z^4+z^5+0\cdot z^6+\ldots).$$
\end{lemma}
\begin{proof}
Similar to the proof of Lemma \ref{p=2,t=1,Stand,c=1,k=4}. A singular vector in $M_{1,1}^6(\Stand)/\left<v_1, v_3,v_5 \right>$ would be of type $\Stand$ and generate a submodule of $M_{1,1}(\Stand)/\left<v_1, v_3,v_5 \right>$ with one copy of $\Stand$ in degree $7$, which is impossible as $M_{1,1}(\Stand)/\left<v_1, v_3,v_5 \right>$ has no copies of $\Stand$ in degree $7$.

\end{proof}

\begin{lemma}\label{p=2,t=1,Stand,c=1,k=7}
Let $p=2$. The space of singular vectors in $M^7_{1,1}(\Stand)/\left<v_1, v_3,v_5 \right>$ is spanned by
$$v_7=\s{2}\s{3}(\x{1}^2\otimes \x{1}+\x{2}^2\otimes \x{1}+\x{2}^2\otimes \x{2}).$$
Hence, the first terms of the character of $L_{1,1}(\Stand)$ are  
$$\chi_{L_{1,1}(\Stand)}(z)=[\Stand](1+z+z^2+2z^3+z^4+z^5+z^6+0\cdot z^7+\ldots)+$$
$$+[\Triv](z+2z^2+2z^4+z^5+0\cdot z^6+0\cdot z^7+\ldots).$$
\end{lemma}
\begin{proof}
The $\Triv$ component of $M_{1,1}^7(\Stand)/\left<v_1, v_3,v_5 \right>$ has a basis
$v_7$ (by repeated use of Theorem \ref{decmposeVermaStandp=2} and \ref{decmposeShp=2}). To see this vector is singular we calculate
\begin{align*}
D_{y_1-y_3}(v_7)&=D_{y_1-y_3}(\s{2}\s{3}(\x{1}^2\otimes \x{1}+\x{2}^2\otimes \x{1}+\x{2}^2\otimes \x{2}))\\
&=\x{1}v_5+(\x{1}^2\x{2}+\x{1}\x{2}^2+\x{2}^3)v_3  \in \left<v_1, v_3,v_5 \right>.
\end{align*}
\end{proof}

\begin{proof}[Proof of Lemma \ref{p=2,t=1,Stand,c=1}]
By Lemmas \ref{p=2,t=1,Stand,c=1,k=1}, \ref{p=2,t=1,Stand,c=1,k=2}, \ref{p=2,t=1,Stand,c=1,k=3}, \ref{p=2,t=1,Stand,c=1,k=4}, \ref{p=2,t=1,Stand,c=1,k=5}, \ref{p=2,t=1,Stand,c=1,k=6} and \ref{p=2,t=1,Stand,c=1,k=7} the module $M_{1,1}(\Stand)/\left<v_1, v_3,v_5,v_7 \right>$ has no singular vectors in degrees up to $7$, and thus coincides with the module $L_{1,1}(\Stand)$ in those degrees. On the other hand Lemma \ref{p=2,t=1,Stand,c=1,k=7} shows that this module is $0$ in degree $7$, so it is zero in all higher degrees as well, and we conclude $M_{1,1}(\Stand)/\left<v_1, v_3,v_5,v_7 \right>=L_{1,1}(\Stand)$, and its character is as stated. The Hilbert polynomial follows directly. 

\end{proof}

\section{Irreducible representations of $H_{t,c}(S_3,\h)$ in characteristic $3$}\label{sectionp=3}

As explained in Section \ref{sect-IrrepsofS3}, over an algebraically closed field $\Bbbk$ of characteristic $3$ the irreducible representations of $H_{t,c}(S_3,\h)$ are $L_{t,c}(\Triv)$ and $L_{t,c}(\Sign)$. We continue to work with the spanning set $\x{1},\x{2},\x{3}$ of $\h^*$ satisfying $\x{1}+\x{2}+\x{3}=0$ and the basis $\x{1},\x{2}$. The aim of the section is to prove the following theorem. Recall that by Remark \ref{remark--restrictedpoly},
$$\chi_{S^{(3)}\h^*}(z)=[\Triv](1+z+2z^2+z^3+z^4)+[\Sign](z+z^2+z^3).$$

\begin{theorem}\label{mainthmp=3}
The characters of the irreducible representation $L_{t,c}(\tau)$ of the rational Cherednik algebra $H_{t,c}(S_3,\h)$ over an algebraically closed field of characteristic $3$, for any $c,t$ and $\tau$, are given by the following table. 

\begin{center}
{\renewcommand*{\arraystretch}{1.5}\begin{tabular}{ c |c |c }
$p=3$ & $\Triv$ & $\Sign$ \\ \hline \hline 
 $t=0$, all $c$ & $[\Triv]$ & $[\Sign]$ \\ \hline 
 $t=1$, all $c$  & $\chi_{S^{(3)}\h^*}(z)$  & $\chi_{S^{(3)}\h^*}(z)\cdot [\Sign]$   
\end{tabular}}
\end{center}

The corresponding Hilbert polynomials are: 
\begin{center}
{\renewcommand*{\arraystretch}{1.5}\begin{tabular}{ c |c |c }
$p=3$ & $\Triv$ & $\Sign$ \\ \hline \hline 
 $t=0$, all $c$ & $1$ & $1$ \\ \hline 
% $t=0$ $c$ special & - & - \\ \hline 
 $t=1$, all $c$  & $\left(\frac{1-z^3}{1-z}\right)^2$  & $\left(\frac{1-z^3}{1-z}\right)^2$   \\ 
  & \scriptsize{\cite{DeSu16}}   &  \\ %\hline 
% $t=1$, $c$ special  &  - &  -  
\end{tabular}}
\end{center}

In all cases the generators of $J_{t,c}(\tau)$ are and given explicitly in the lemmas below. 
\end{theorem}
\begin{proof}
The description of $L_{t,c}(\Triv)$ is given in: 
\begin{itemize}
\item for $t=0$ and any $c$, in Lemma \ref{Lian3.2.}; 
\item for $t=1$ and any $c$, in Lemma \ref{p=3,t=1,Triv,c=any}. 
\end{itemize}
The analogous descriptions of the irreducible representation $L_{t,c}(\Sign)$ can be deduced from the description of $L_{t,-c}(\Triv)$ and Lemma \ref{Sign}. Note that the generators of $J_{1,c}(\Triv)$ and the Hilbert polynomials of $L_{1,c}(\Triv)$ for generic $c$ are known from \cite{DeSu16}.
\end{proof}

\begin{comment}
\begin{lemma}\label{p=3,t=0,Triv,c=any}
Let $p=3$, $t=0$ and $c\in \Bbbk$ be arbitrary. Vectors $\x{1}, \x{2}\in M_{0,c}(\Triv)$ are singular. Consequently, the irreducible representation $L_{0,c}(\Triv)$ of $H_{0,c}(S_3)$ is one dimensional and concentrated in degree $0$, with the character and Hilbert polynomial
$$\chi_{L_{0,c}(\Triv)}(z)=[\Triv], \quad h_{L_{0,c}(\Triv)}(z)=1.$$
\end{lemma}
\begin{proof}
We will show that $\x{1}$ and $\x{2}$ are singular, and the other claims follow immediately. 
We compute
\begin{align*}
D_{y_1-y_3}(\x{1})&=\left(-c\frac{\id-(12)}{\x{1}-\x{2}}-c\frac{\id-(23)}{\x{2}-\x{3}}-2c\frac{\id-(13)}{\x{1}-\x{3}}\right)(\x{1})=-c-0-2c=0\\
D_{y_2-y_3}(\x{1})&=\left(c\frac{\id-(12)}{\x{1}-\x{2}}-2c\frac{\id-(23)}{\x{2}-\x{3}}-c\frac{\id-(13)}{\x{1}-\x{3}}\right)(\x{1})=c-0-c=0.
\end{align*}
By symmetry, 
\begin{align*}
D_{y_1-y_3}(\x{2})&=(12).D_{y_2-y_3}(\x{1})=0\\
D_{y_2-y_3}(\x{2})&=(12).D_{y_1-y_3}(\x{1})=0.
\end{align*}
\end{proof}
\end{comment}

\begin{lemma}\label{p=3,t=1,Triv,c=any}
Let $p=3$, and $t=1$ and $c\in \Bbbk$ be arbitrary. Vectors $\x{1}^3$ and $\x{2}^3$ in $M_{1,c}(\Triv)$ are singular. The quotient $M_{1,c}(\Triv)/\left<\x{1}^3,\x{2}^3\right>$ is irreducible, equal to $L_{1,c}(\Triv)$ and has the character and the Hilbert polynomial
\begin{align*}
\chi_{L_{1,c}(\Triv)}(z)&=[\Triv](1+z+2z^2+z^3+z^4)+[\Sign](z+z^2+z^3) \\
h_{L_{1,c}(\Triv)}(z)&=1+2z+3z^2+2z^3+z^4=\left(\frac{1-z^3}{1-z}\right)^2.
\end{align*}
\end{lemma}
\begin{proof}
Let us first show that $\x{1}^3$ and $\x{2}^3$ are singular in $M_{1,c}(\Triv)$.
We compute
\begin{align*}
D_{y_1-y_3}(\x{1}^3)&=\left(\partial_{y_1-y_3}-c\frac{\id-(12)}{\x{1}-\x{2}}-c\frac{\id-(23)}{\x{2}-\x{3}}-2c\frac{\id-(13)}{\x{1}-\x{3}}\right)(\x{1}^3)\\
&=-c\cdot\left(\x{1}^2+\x{1}\x{2}+\x{2}^2+0+2\x{1}^2+2\x{1}\x{3}+2\x{3}^2\right) \\
&=-c\cdot \left(\x{1}^2+\x{1}\x{2}+\x{2}^2+2\x{1}^2-2\x{1}^2-2\x{1}\x{2}+2\x{1}^2+4\x{1}\x{2}+2\x{2}^2\right) \\
&=-c\cdot \left(3\x{1}^2+3\x{1}\x{2}+3\x{2}^2\right)=0 \\
D_{y_2-y_3}(\x{1})&=\left(\partial_{y_1-y_3}+c\frac{\id-(12)}{\x{1}-\x{2}}-2c\frac{\id-(23)}{\x{2}-\x{3}}-c\frac{\id-(13)}{\x{1}-\x{3}}\right)(\x{1}^3)\\
&=c\left(\x{1}^2+\x{1}\x{2}+\x{2}^2-2\cdot 0 -\x{1}^2-\x{1}\x{3}-\x{3}^2  \right)\\
&=c\left(\x{1}\x{2}+\x{2}^2+\x{1}^2 +\x{1}\x{2}-\x{1}^2-2\x{1}\x{2}-\x{2}^2  \right)=0.
\end{align*}
By symmetry, 
\begin{align*}
D_{y_1-y_3}(\x{2}^3)&=(12).D_{y_2-y_3}(\x{1}^3)=0\\
D_{y_2-y_3}(\x{2}^3)&=(12).D_{y_1-y_3}(\x{1}^3)=0.
\end{align*}
So, $\x{1}^3$ and $\x{2}^3$ are singular. 

Next, we calculate the character of $M_{1,c}(\Triv)/\left<\x{1}^3,\x{2}^3 \right>$. This is easiest to do in the basis $a_+=\x{1}+\x{2}$, $a_-=\x{1}-\x{2}$ of $\h^*$, noting that $\left<\x{1}^3,\x{2}^3 \right>=\left<a_+^3,a_-^3\right>$. The quotient $M^1_{1,c}(\Triv)/\left<a_+^3,a_-^3 \right>$ has a basis 
$$\left\{a_+,a_-,a_-^2,a_-a_+,a_+^2,a_-^2a_+,a_-a_+^2,a_-^2a_+^2\right\}.$$
This and the fact that $[\Bbbk a_+]=[\Triv]$ and $[\Bbbk a_-]=[\Sign]$ let us show that the character and the Hilbert polynomial are as stated. 

It remains to show that $M_{1,c}(\Triv)/\left<\x{1}^3,\x{2}^3 \right>$ is irreducible. We do so in a way similar to the proof of Lemma \ref{p=2,t=1,generic,Triv}, using the fact that the module $M_{1,c}(\Triv)/\left<\x{1}^3,\x{2}^3 \right>\cong S\h^*/\left<\x{1}^3,\x{2}^3 \right>$ is also a Frobenius algebra. We work with the basis $\left\{\x{1}^a\x{2}^b \mid 0\le a,b < 3\right\}$. 

Assume that $U$ is a nonzero graded submodule of $M_{1,c}(\Triv)/\left<\x{1}^3,\x{2}^3 \right>$, and $v\in U$ some nonzero homogeneous vector. Multiplying by a nonzero constant if needed, $v$ can be written as $$v=\x{1}^a\x{2}^b+\sum_{a'>a}\alpha_{a'} \x{1}^{a'}\x{2}^{a+b-a'}$$
for some $a,b$ with $0\le a,b<3$ and $\alpha_{a'}\in \Bbbk$. 
As $U$ is a subrepresentation and using $\x{1}^3=0$, we get
$$\x{1}^{2-a}\x{2}^{2-b}v=\x{1}^2\x{2}^2+\sum_{a'>a}\alpha_a' \x{1}^{2+a'-a}\x{2}^{2+a-a'}=\x{1}^2\x{2}^2$$
is in $U$. Then $U$ also contains $$D_{y_1-y_2}^2D_{y_1+y_2+y_3}^2\x{1}^2\x{2}^2.$$  (Note this makes sense as in characteristic $3$ we have $Y=y_1+y_2+y_3\in \h$.) 

Using $\left<Y, \alpha_{s}\right>=0$ for all $\alpha_s$ so $D_{Y}=\partial_Y$, we calculate
\begin{align*}
D_{y_1-y_2}^2D_{y_1+y_2+y_3}^2(\x{1}^2\x{2}^2)&=2D_{y_1-y_2}^2D_{y_1+y_2+y_3}(\x{1}^2\x{2}+\x{1}\x{2}^2)\\
&=2D_{y_1-y_2}^2\left(2\x{1}\x{2}+\x{2}^2+\x{1}^2+2\x{1}\x{2} \right)\\
&=2D_{y_1-y_2}^2\left(\x{1}^2+\x{1}\x{2} +\x{2}^2\right).
\end{align*}
We now use the fact that in characteristic $3$ we have $2(\x{1}^2+\x{1}\x{2} +\x{2}^2)=\overline{\sigma_2}$ and that $D_{y_1-y_2}(\overline{\sigma_2})=\partial_{y_1-y_2}(\overline{\sigma_2})$ to get
\begin{align*}
D_{y_1-y_2}^2D_{y_1+y_2+y_3}^2(\x{1}^2\x{2}^2)&=2D_{y_1-y_2}\partial_{y_1-y_2}\left(\x{1}^2+\x{1}\x{2} +\x{2}^2\right)\\
&=2D_{y_1-y_2}\left( 2\x{1}+\x{2} -\x{1}-2\x{2}\right)\\
&=2D_{y_1-y_2}\left( \x{1}-\x{2}\right)\\
&=2\left(\partial_{y_1-y_2}-2c\frac{\id - (12)}{\x{1}-\x{2}}-c\frac{\id - (13)}{\x{1}-\x{3}}+c\frac{\id - (23)}{\x{2}-\x{3}} \right)(\x{1}-\x{2})\\
&=2\left(1+1-2c\cdot 2-c\cdot 1+c\cdot (-1)\right)\\
&=4-12c=1. 
\end{align*}
So, $1\in U$. However, $1$ generates the module $M_{1,c}(\Triv)/\left<\x{1}^3,\x{2}^3 \right>$, and consequently $U$ is the whole module. This shows that $M_{1,c}(\Triv)/\left<\x{1}^3,\x{2}^3 \right>$ is irreducible.
\end{proof}

\section{Auxiliary computations using the rescaled Young basis}\label{sect-auxiliary}

For the rest of the paper the characteristic of the field $\Bbbk$ is $p>3$. So, the category of representations of $S_3$ is semisimple, we can realise $\h^*$ as a subrepresentation of $V^*$ with the basis
$$b_+=x_1+x_2-2x_3,\quad b_-=x_1-x_2,$$
from Section \ref{sec-basisofh}, and we will be using the bases of $M_{t,c}(\tau)$ from Theorems \ref{decmposeShp>3} and \ref{decmposeVermap>3}. To manipulate those bases we will first need to know the action of all reflections on $b_+,b_-$, which can be easily calculated to be: 
\begin{align}\label{action-of-S3-b+b-}
(12).b_+=b_+, \quad &(12).b_-=-b_+ \notag \\
(23).b_+=\frac{-b_++3b_-}{2},\quad  &(23).b_-=\frac{b_++b_-}{2}\\
(13).b_+=\frac{-b_+-3b_-}{2},\quad  &(13).b_-=\frac{-b_++b_-}{2}. \notag
\end{align}

The main advantages of the basis $b_+,b_-$ for $\h^*$ is that it lets us use Theorems \ref{decmposeShp>3} and \ref{decmposeVermap>3} to write a bases of each Verma modules, itself good because: 
\begin{enumerate}
\item it is compatible with the decomposition into $S_3$ representations, so we can use Lemma \ref{OmegaAction} to reduce the space in which we look for singular vectors; 
\item every vector in this basis is either $S_2$ invariant or anti-invariant, enabling us to use Lemma \ref{KerD1}; 
\item calculating Dunkl operators in this basis is manageable. Namely, each vectors in it is a product of an $S_3$ invariant and one of a small number of fixed low degree expressions in $b_+,b_-$. The differential part of the Dunkl operator satisfies the Leibniz rule so is easy to apply to products, and the difference part of the Dunkl operator is linear with respect to invariants so will only need to be calculated on a small number of low degree vectors.
\end{enumerate}

The main disadvantage of the basis $b_+,b_-$ and resulting bases Verma modules is that multiplication (except by a symmetric polynomial) is more involved. We gather all tedious multiplication formulas among $b_+,b_-,(-b_+^2+3b_-^2), 2b_+b_-$ and $q$ we will need for our computations of singular vectors later into one lemma. 

\begin{lemma} 
The following identities hold in $S\h^*$: 
\begin{align*}
b_+^2&%=\frac{1}{2}(b_+^2+3b_-^2)-\frac{1}{2}(-b_+^2+3b_-^2)
=-6\sigma_2-\frac{1}{2}(-b_+^2+3b_-^2) \\
b_-^2&=%\frac{1}{6}(b_+^2+3b_-^2)+\frac{1}{6}(-b_+^2+3b_-^2)=
-2\sigma_2+\frac{1}{6}(-b_+^2+3b_-^2)\\
b_+\cdot 2b_+b_-&=6\cdot q-6\cdot \sigma_2 b_-\\
b_+\cdot (-b_+^2+3b_-^2)&=54\cdot \sigma_3+6\cdot   \sigma_2b_+\\
b_-\cdot 2b_+b_-&=18\cdot \sigma_3 -2\cdot   \sigma_2b_+\\
b_-\cdot (-b_+^2+3b_-^2)&=-6\cdot q-6\cdot \sigma_2b_-\\
(-b_+^2+3b_-^2)^2&=72\cdot \sigma_2^2-108\cdot \sigma_3b_+ -6 \cdot \sigma_2(-b_+^2+3b_-^2)\\
 (-b_+^2+3b_-^2)\cdot 2b_+b_-&=108\cdot\sigma_3b_-+6\cdot \sigma_2\cdot 2b_+b_- \\
 (2b_+b_-)^2&=24\cdot \sigma_2^2+ 36\cdot \sigma_3b_+ + 2 \cdot \sigma_2(-b_+^2+3b_-^2)\\
 q\cdot b_+ &= -9\cdot \sigma_3b_- -1\cdot \sigma_2\cdot 2\cdot b_+b_-\\
 q\cdot b_- &= 3\cdot \sigma_3b_+ + \frac{1}{3}\cdot \sigma_2(-b_+^2+3b_-^2)\\
 (-b_+^2+3b_-^2) \cdot q&= -12\cdot \sigma_2^2b_- +9\cdot\sigma_3\cdot 2b_+b_-\\
2b_+b_- \cdot q&=4\cdot \sigma_2^2b_+ -3\cdot\sigma_3(-b_+^2+3b_-^2)\\
q^2&=-27\sigma_3^2-4\sigma_2^3.
\end{align*}
\end{lemma}
\begin{proof}
Direct computation. 
\end{proof}

When looking for singular vectors, we will make liberal use of Lemma \ref{OmegaAction} to reduce the space where we are looking for singular vectors, and Lemma 
\ref{KerD1} to reduce the task of finding $\cap_{y\in \h}\mathrm{Ker} D_{y}$ to the task of finding $\mathrm{Ker} D_{y_1}$. For that purpose, let us make $D_{y_1}$ explicit as 
$$D_{y_1}=t\partial_{y_1}\otimes \id - c\cdot  \frac{\id - (12)}{x_1-x_2}\otimes (12)-c\cdot \frac{id-(13)}{x_1-x_3}\otimes (13).$$
For any reflection $s$, any vector $v\in M_{t,c}(\tau)$ or one of its quotients, and any $f\in (S\h^*)^{S_3}$ we have 
\begin{align}
\left( \frac{\id - s}{\alpha_s}\otimes s\right) (fv)=f\cdot \left( \frac{\id - s}{\alpha_s}\otimes s\right)(v).
\end{align}

Let us now calculate the action of the divided difference operators and the partial derivatives on all factors which appear in the bases of $S\h^*$ and for $M_{t,c}(\tau)$ from Theorems \ref{decmposeShp>3} and \ref{decmposeVermap>3}.
\begin{lemma} 
The following identities hold in $S\h^*$: 
\begin{align*}
\frac{id-(12)}{x_1-x_2}(b_+)=0, \quad &\frac{id-(13)}{x_1-x_3}(b_+)=3,\\
\frac{id-(12)}{x_1-x_2}(b_-)=2, \quad &\frac{id-(13)}{x_1-x_3}(b_-)=1,\\
\frac{id-(12)}{x_1-x_2}(-b_+^2+3b_-^2)=0, \quad &\frac{id-(13)}{x_1-x_3}(-b_+^2+3b_-^2)=3(-b_++3b_-),\\
\frac{id-(12)}{x_1-x_2}(2b_+b_-)=4b_+, \quad &\frac{id-(13)}{x_1-x_3}(2b_+b_-)=-b_++3b_-,\\
\frac{id-(12)}{x_1-x_2}(q)%&=2(x_2-x_3)(x_1-x_3)\\
=-2 \sigma_2-\frac{1}{3} (-b_+^2+3b_-^2),\quad  &\frac{id-(13)}{x_1-x_3}(q)%&=2(x_1-x_2)(x_2-x_3)\\
= 2 \sigma_2-\frac{1}{6} (-b_+^2+3b_-^2)+\frac{1}{2}\cdot 2b_+b_-.
%\frac{id-(12)}{x_1-x_2}(\sigma_2)&=0\\
%\frac{id-(12)}{x_1-x_2}(\sigma_3)&=0\\
%\frac{id-(13)}{x_1-x_3}(\sigma_2)&=0\\
%\frac{id-(13)}{x_1-x_3}(\sigma_3)&=0\\
\end{align*}
\end{lemma}
\begin{proof}
Direct computation.
\end{proof}

\begin{lemma} 
The following identities hold in $S\h^*$: 
\begin{align*}
\partial_{y_1}(b_+)=1, \quad &\partial_{y_1}(b_-)=1,\\
\partial_{y_1}(-b_+^2+3b_-^2)=-2 b_++6 b_-, \quad &\partial_{y_1}(2b_+b_-)=2 b_++2 b_-\\
\partial_{y_1}(\sigma_2)%&=\frac{-2x_1+x_2+x_3}{3}\\
=\frac{-1}{6} b_++\frac{-1}{2} b_-, \quad 
&\partial_{y_1}(\sigma_3)%&=\frac{1}{9}((x_1-x_2)(x_1+x_2-2x_3)+(x_1-x_3)(x_1-2x_2+x_3) )\\
=\frac{1}{36} (-b_+^2+3b_-^2 )+\frac{1}{12} (2b_+b_-),\\
\partial_{y_1}(q)%&=(2x_1-x_2-x_3)(x_2-x_3)\\
%&=\frac{b_++3b_-}{2}\cdot \frac{b_+-b_-}{2}\\
=\frac{-1}{4}&(-b_+^2+3b_-^2)+\frac{1}{4}(2b_+b_-).
\end{align*}
\end{lemma}
\begin{proof}
Direct computation. 
\end{proof}

Finally, in the attempt to express the value of the Dunkl operator $D_{y_1}$ on a vector in our basis as a linear combination of basis vectors, we will often need the following computations. 
\begin{lemma} 
The following identities hold in $S\h^*$: 
\begin{align*}
\partial_{y_1}(\sigma_2)\cdot b_+%&=\frac{-1}{6}b_+^2-\frac{1}{2}b_+b_-\\
&=\sigma_2+\frac{1}{12} (-b_+^2+3b_-^2)+\frac{-1}{4} (2b_+b_-)\\
\partial_{y_1}(\sigma_3)\cdot b_+&=\frac{3}{2}\sigma_3+\frac{1}{6}\sigma_2b_+ -\frac{1}{2}\sigma_2b_-+\frac{1}{2}q\\
\partial_{y_1}(\sigma_2)\cdot b_-&=\sigma_2+\frac{-1}{12}(-b_+^2+3b_-^2)+\frac{-1}{12} (2b_+b_-)\\
\partial_{y_1}(\sigma_3)\cdot b_-&=\frac{3}{2}\sigma_3+\frac{-1}{6}\sigma_2b_++\frac{-1}{6}\sigma_2b_-+\frac{-1}{6}q\\
\partial_{y_1}(\sigma_2)\cdot (-b_+^2+3b_-^2)&=-9\sigma_3 - \sigma_2b_++3\sigma_2b_-+3 q\\
\partial_{y_1}(\sigma_3)\cdot (-b_+^2+3b_-^2)&=2\sigma_2^2-3\sigma_3b_+ + 9\sigma_3b_--\frac{1}{6}\sigma_2(-b_+^2+3b_-^2)+\frac{1}{2}\sigma_2(2b_+b_-)\\
\partial_{y_1}(\sigma_2)\cdot 2b_+b_-&=-9\sigma_3+\sigma_2b_++\sigma_2b_-- q\\
\partial_{y_1}(\sigma_3)\cdot 2b_+b_-&= 2\sigma_2^2 + 3\sigma_3b_  + +3\sigma_3b_- +\frac{1}{6}\sigma_2(-b_+^2+3b_-^2) +\frac{1}{6}\sigma_2(2b_+b_-)\\
\partial_{y_1}(\sigma_2)\cdot q&=\frac{-3}{2}\sigma_3b_+ + \frac{3}{2}\sigma_3b_- +\frac{-1}{6}\sigma_2(-b_+^2+3b_-^2)+\frac{1}{6}\sigma_2(2b_+b_-)\\
\partial_{y_1}(\sigma_3)\cdot q&=\frac{1}{3} \sigma_2^2b_+ \frac{-1}{3} \sigma_2^2b_- +\frac{-1}{4}\sigma_3(-b_+^2+3b_-^2) + \frac{1}{4}\sigma_3 (2b_+b_-).
\end{align*}
\end{lemma}
\begin{proof}
Direct computation. 
\end{proof}

\section{Irreducible representations of $H_{t,c}(S_3,\h)$ in characteristic $p>3$ for generic $c$}\label{sect-p>3-generic}

Recall that for the rest of the paper the characteristic of the field $\Bbbk$ is $p>3$. The aim of this section is to prove the following theorem. 
\begin{theorem}\label{thm-p>3generic}
The characters and the Hilbert polynomials of the irreducible representation $L_{t,c}(\tau)$ of the rational Cherednik algebra $H_{t,c}(S_3,\h)$ over an algebraically closed field of characteristic $p>3$, for generic $c$, $t=0,1$, and any $\tau$, are given by the following tables. 

Characters: 
\begin{center}
{\renewcommand*{\arraystretch}{1.5}\begin{tabular}{ c |c |c |c }
$p>3$ & $\Triv$ & $\Sign$ & $\Stand$ \\ \hline \hline 
  $t=0$,  & $[\Triv]+z[\Stand]+$ & $[\Sign]+z[\Stand]+$ & $[\Stand]+z[\Triv]+$ \\ 
  $c\ne 0$ & $+ z^2[\Stand]+z^3[\Sign]$ & $+z^2[\Stand]+z^3[\Triv]$ & $+z[\Sign]+z^2[\Stand]$\\
   & \scriptsize{\cite[Prop 4.1]{DeSa14}} &  \scriptsize{\cite[Prop 4.1]{DeSa14}}  & \scriptsize{\cite[Prop 4.1]{DeSa14}} \\ \hline 
 %$t=0$, $c=0$ & $[\Triv]$ & $[\Sign]$ & $[\Stand]$ \\ \hline 
    $t=1$, & $\chi_{S\h^*}(z)\cdot $ & $\chi_{S\h^*}(z)\cdot [\Sign]\cdot $ & $\chi_{S\h^*}(z)\cdot [\Stand]$ \\ 
    $c\notin \mathbb{F}_p$   & $\cdot (1-z^{2p})(1-z^{3p})$ & $\cdot (1-z^{2p})(1-z^{3p})$ & $\cdot  (1-z^{p})(1-z^{3p})$ \\ 
  &    \scriptsize{\cite[Prop 4.2]{DeSa14}}  & \scriptsize{\cite[Prop 4.2]{DeSa14}} & \scriptsize{\cite[Prop 4.2]{DeSa14}}  
 \end{tabular}}
\end{center}

\vspace{0.5cm}
Hilbert polynomials: 
\begin{center}
{\renewcommand*{\arraystretch}{1.5}\begin{tabular}{ c |c |c |c }
$p>3$ & $\Triv$ & $\Sign$ & $\Stand$ \\ \hline \hline 
  $t=0$, $c\ne 0 $ & $1+2z+2z^2+z^3$ & $1+2z+2z^2+z^3$ & $2+2z+2z^2$ \\ 
   & \scriptsize{\cite[Prop 4.1]{DeSa14}} &  \scriptsize{\cite[Prop 4.1]{DeSa14}}  & \scriptsize{\cite[Prop 4.1]{DeSa14}} \\ \hline 
% $t=0$, $c=0$ & 1 & 1 & 2 \\ \hline 
 $t=1$, $c\notin \mathbb{F}_p$  
& $\frac{(1-z^{2p})(1-z^{3p})}{(1-z)^2}$ & $\frac{(1-z^{2p})(1-z^{3p})}{(1-z)^2}$ & $\frac{2(1-z^{p})(1-z^{3p})}{(1-z)^2}$ \\ 
  &    \scriptsize{\cite[Prop 4.2]{DeSa14}}  & \scriptsize{\cite[Prop 4.2]{DeSa14}} & \scriptsize{\cite[Prop 4.2]{DeSa14}}  
 \end{tabular}}
\end{center}
The generators of $J_{t,c}(\tau)$ are:
\begin{itemize}
\item For all $\tau$ and $t=0$, $\sigma_i\otimes v$ for $i=2,3$ and all $v\in \tau$; 
\item For all $\tau$ and $t=1$, $\sigma_i^p\otimes v$ for $i=2,3$ and all $v\in \tau$; 
\item For $\tau=\Stand$, $t=0$, additionally, the two vectors from Lemmas \ref{pne3,t=0,cne0vectors} and \ref{pne3t=0cne0char}; %\ref{p>3,t=0,generic,Stand}; 
\item For $\tau=\Stand$, $t=1$, additionally, the two vectors from Lemma \ref{p>3,t=1,generic,Stand}.
\end{itemize}
\end{theorem}

To write the above characters explicitly, recall from Corollary \ref{characterShp>3} that 
$$\chi_{S\h^*}(z)=\frac{1}{(1-z^2)(1-z^3)}\left([\Triv]+[\Stand](z+z^2)+[\Sign]z^3 \right).$$

We note that \cite{DeSa14} already provide these characters and Hilbert polynomials in their Propositions 4.1 and 4.2, in much greater generality of $G(m,r,n)$. We do not use their results. Our main contribution is to give explicit singular vectors, and their results are reproved along the way. Other than for its general interest, these singular vectors are needed in Section \ref{sect-p>3-special} where we describe $L_{t,c}(\tau)$ for special $c$ - a case \cite{DeSa14} do not discuss.

\begin{proof}
The generic values of $c$ for $t=0$ are $c\ne 0$ and for $t=1$ they are $c\notin \mathbb{F}_p$ by Proposition \ref{WhatIsGeneric}. For those values, \cite[Thm 1.7]{EtGi02} and \cite[Rmk 1.2.3]{BFG06} (see also proof of Prop 4.1. and 4.2. in \cite{DeSa14}) imply that $\dim L_{0,c}(\tau)=|S_3|=6$, and $\dim L_{1,c}(\tau)=|S_3|\cdot \dim S^{(p)}\h^*=6p^2$. For $t=0,1$ and $\tau=\Triv,\Sign$, this is equal to the dimension of $N_{t,c}(\tau)$, whose characters have been computed in Corollary \ref{characterShp>3}. This implies that $L_{t,c}(\tau)=N_{t,c}(\tau)$ and proves the theorem for $\tau=\Triv,\Sign$. 

When $\tau=\Stand$, the same reasoning tells us that $L_{t,c}(\tau)$ is a proper quotient of $N_{t,c}(\tau)$.

For $t=0$ the additional singular vectors are computed in Lemma \ref{pne3,t=0,cne0vectors}, and the module $L_{0,c}(\Stand)$ and its character described in Lemma \ref{pne3t=0cne0char}. We rewrite them in the basis $b_+,b_-$ in Remark \ref{rmk-cob-v1v2}. This proves the theorem for $t=0$, $\tau=\Stand$.

For $t=1$ the additional singular vectors are computed in Lemma \ref{p>3,t=1,generic,Stand}, and the character of the quotient of $N_{1,c}(\Stand)$ by these vectors is computed in Lemma \ref{p>3,t=1,generic,Stand-character}. Observing the dimension of this module is equal to $6p^2$ we conclude it is irreducible and thus equal to $L_{1,c}(\Stand)$, which proves the theorem for $t=1$, $\tau=\Stand$.

\end{proof}

\begin{remark}\label{rmk-cob-v1v2} Let $\tau=\Stand$. For $t=1$ the space of additional singular vectors $\mathrm{span}\{v_+,v_-\}$ is described in terms $v_+,v_-$ written in the basis $b_+,b_-$ in Lemma \ref{p>3,t=1,generic,Stand}. For $t=0$ we do not do the analogue because the space of additional singular vectors is already described in terms of $v_1,v_2$ in Lemmas \ref{pne3,t=0,cne0vectors} and \ref{pne3t=0cne0char}. However, an analogous computation would yield $v_+,v_-\in M_{0,c}(\Stand)$, related to $v_1,v_2$ via
\begin{align*}
v_+&=-b_+\otimes b_++3b_-\otimes b_-=-6\pi(v_1+v_2)\\
v_-&=b_+\otimes b_-+b_-\otimes b_+=-2\pi(v_1-v_2).
\end{align*}

\end{remark}

\subsection{Singular vectors in $M_{1,c}(\Stand)$}

The rest of the section is dedicated to the remaining case of $t=1$, $\tau=\Stand$ and generic $c$. We will describe the generators of $J_{1,c}(\Stand)$ by explicit computations of Dunkl operators in a basis from Theorem \ref{decmposeVermap>3}. 

By \cite[Prop 3.4]{BaCh13a}, for generic $c$ singular vectors only appear in degrees divisible by $p$. By Lemma \ref{OmegaAction}, any singular vectors in degrees divisible by $p$ of $M_{1,c}(\Stand)$ are in the isotypic component of $\Stand$. As $\Stand$ is irreducible, it is enough to look for singular vectors in the one dimensional subspace of $\Stand$ which restricts to the trivial representation of $S_2$ (in other words, the image of $b_+$ under any isomorphism from $\Stand$). 
By Theorem \ref{decmposeVermap>3}, a basis of this part of $\Stand$ in degree $kp$ is: 
\begin{itemize}
\item $\sigma_2^a\sigma_3^b \otimes b_+$, $2a+3b=kp$; 
\item $\sigma_2^a\sigma_3^b \cdot (-b_+\otimes b_++3b_-\otimes b_-)$, $2a+3b=kp-1$;
\item $\sigma_2^a\sigma_3^b \cdot \left(-(-b_+^2+3b_-^2)\otimes b_++3(2b_+b_-)\otimes b_- \right)$, $2a+3b=kp-2$;
\item $\sigma_2^a\sigma_3^b q \otimes b_-$, $2a+3b=kp-3$.
\end{itemize}

By Lemma \ref{KerD1}, looking for singular vectors in degree $kp$ means looking for a linear combination of the above vectors which is in the kernel of $D_{y_1}$. We first calculate the values of the Dunkl operator $D_{y_1}$ on the vectors listed above. 

\begin{lemma}\label{valuesDy1-Stand+}
\begin{enumerate}
\item Let $a,b,k\in \mathbb{N}_0$, $2a+3b=kp$. Then 
\begin{align*}
D_{y_1}(\sigma_2^a\sigma_3^b\otimes b_+)&=\left( \frac{-a}{6}\right) \sigma_2^{a-1}\sigma_3^b\left(b_+\otimes b_+ + 3b_-\otimes b_+\right) +\\
&\hskip 5 pt +\left(\frac{b}{36}\right)\sigma_2^a\sigma_3^{b-1}\left( (-b_+^2+3b_-^2)\otimes b_+ +3(2b_+b_-)\otimes b_+\right).
\end{align*}
\item Let $a,b,k\in \mathbb{N}_0$, $2a+3b=kp-1$. Then 
\begin{align*}
&\hskip -30pt D_{y_1}(-\sigma_2^a\sigma_3^bb_+\otimes b_+ + 3\sigma_2^a\sigma_3^bb_-\otimes b_-)=\hskip 30pt \\
&\hskip 10 pt=\left(\frac{-1}{2}\right)\sigma_2^a\sigma_3^b\otimes \left( b_+-3b_-\right)+\\
&\hskip 20pt +\left(\frac{-b}{6}\right) \sigma_2^{a+1}\sigma_3^{b-1} \left( b_+\otimes b_+ -3b_-\otimes b_+ + 3b_+\otimes b_-+ 3b_-\otimes b_- \right)+\\
&\hskip 20pt +\left(\frac{-a}{12}\right)\sigma_2^{a-1}\sigma_3^{b}\left( (-b_+^2+3b_-^2)\otimes b_+ -3(2b_+b_-)\otimes b_++\right. \\
& \hskip 120pt \left. +3 (-b_+^2+3b_-^2)\otimes b_- + 3(2b_+b_-)\otimes b_-\right)+ \\
&\hskip 20pt +\left(\frac{-b}{2}\right)\sigma_2^a\sigma_3^{b-1}q\otimes \left( b_++b_-\right).
\end{align*}
\item Let $a,b,k\in \mathbb{N}_0$, $2a+3b=kp-2$. Then 
\begin{align*}
&\hskip -30pt D_{y_1}(-\sigma_2^a\sigma_3^b(-b_+^2+3b_-^2)\otimes b_+ + 3\sigma_2^a\sigma_3^b(2b_+b_-)\otimes b_-)=\\[5pt]
&=\hskip 5pt (-2b)\sigma_2^{a+2}\sigma_3^{b-1}\otimes (b_+-3b_-)+\\
&\hskip 15 pt +(9a)\sigma_2^{a-1}\sigma_3^{b+1}\otimes (b_+-3b_-)+\\
&\hskip 15pt + (-a)\sigma_2^a\sigma_3^b \left(b_+\otimes b_+ -3b_-\otimes b_+ +3 b_+\otimes b_-+3b_+\otimes b_- \right)+ \\
&\hskip 15pt + (18c)\sigma_2^a\sigma_3^b\left( b_+\otimes b_-- b_-\otimes b_-\right)+\\
&\hskip 15pt +\left(\frac{b}{6}\right)\sigma_2^{a+1}\sigma_3^{b-1}\left( (-b_+^2+3b_-^2)\otimes b_+ -3(2b_+b_-)\otimes b_++\right. \\
& \hskip 120pt \left.+3 (-b_+^2+3b_-^2)\otimes b_- + 3(2b_+b_-)\otimes b_-\right)+ \\
&\hskip 15pt +(-3a) \sigma_2^{a-1}\sigma_3^{b}q\otimes \left( b_++ b_-\right).
\end{align*}
\item Let $a,b,k\in \mathbb{N}_0$, $2a+3b=kp-3$. Then 
\begin{align*}
D_{y_1}(\sigma_2^a\sigma_3^bq\otimes b_-) &= c \sigma_2^{a+1}\sigma_3^{b}\otimes \left( b_+-3b_-\right)+ \\
&\hskip 5 pt + \frac{b}{3}\sigma_2^{a+2}\sigma_3^{b-1}\left( b_+\otimes b_- - b_-\otimes b_-\right)+ \\
&\hskip 5 pt + \frac{-3a}{2}\sigma_2^{a-1}\sigma_3^{b+1}\left( b_+\otimes b_- - b_-\otimes b_-\right)+ \\
&\hskip 5 pt + \frac{-c}{12}\sigma_2^{a}\sigma_3^{b}\left( (-b_+^2+3b_-^2)\otimes b_+ -3(2b_+b_-)\otimes b_++\right. \\
& \hskip 120pt \left.+3 (-b_+^2+3b_-^2)\otimes b_- + 3(2b_+b_-)\otimes b_-\right).
\end{align*}
\end{enumerate}
\end{lemma}
\begin{proof}
Direct computation using the auxiliary results from Section \ref{sect-auxiliary}. 
\end{proof}

Let us first look for singular vectors in degree $p$ first (this will turn out to be enough). 
Parametrising the integers $a,b$ which label the basis from the start of this section like in the proof of Lemma \ref{numberinbasis}, we are looking for a vector in degree $p$ of the form 
\begin{align}\label{eqn-v+-unknowns}
v_+&=\sum_{0\le j\le \lfloor \frac{p-3}{6}\rfloor}  \alpha_j \sigma_2^{\frac{p-3}{2}-3j}\sigma_3^{2j+1} \otimes b_+ +\\
&+\sum_{0\le j\le \lfloor \frac{p-1}{6}\rfloor}\beta_j \sigma_2^{\frac{p-1}{2}-3j}\sigma_3^{2j} \cdot (-b_+\otimes b_++3b_-\otimes b_-)+ \notag\\
&+\sum_{0\le j\le \lfloor \frac{p-5}{6}\rfloor}  \gamma_j \sigma_2^{\frac{p-5}{2}-3j}\sigma_3^{2j+1} \left(-(-b_+^2+3b_-^2)\otimes b_++3(2b_+b_-)\otimes b_- \right)+\notag\\
&+\sum_{0\le j\le \lfloor\frac{p-3}{6}\rfloor}\delta_j \sigma_2^{\frac{p-3}{2}-3j}\sigma_3^{2j} \cdot q \otimes b_-\notag
\end{align}
for some $\alpha_j,\beta_j,\gamma_j,\delta_j\in \Bbbk$ (depending on $c$) with the property that $D_{y_1}(v_+)=0$.

\begin{lemma}\label{Stand-generic-alpha,gamma}
If a vector $v_+$ of the form \eqref{eqn-v+-unknowns} satisfies $D_{y_1}(v_+)=0$, then 
\begin{enumerate}
\item $\alpha_j=0$ for all $j$; 
\item For all $0\le j\le \lfloor \frac{p-7}{6}\rfloor$
$$ \gamma_j=\frac{2(j+1)}{3(6j+5)} \beta_{j+1}.$$
\end{enumerate}
\end{lemma}
Note that if $p\equiv 2 \pmod 3$ then there is also a coefficient $\gamma_{\frac{p-5}{6}}$, which this Lemma puts no conditions on. Otherwise this Lemma determines all $\alpha_j,\gamma_j$ in terms of $\beta_j,\delta_j$.

\begin{proof}
The right hand sides in each of the expressions in Lemma \ref{valuesDy1-Stand+} are linearly independent, so we will be reading off their coefficients in the expansion of the equation $D_{y_1}(v_+)=0$ using Lemma \ref{valuesDy1-Stand+}. 
\begin{enumerate}
\item 
The coefficient of $\left(b_+\otimes b_+ + 3b_-\otimes b_+\right)$ in the equation $D_{y_1}(v_+)=0$ expanded as in Lemma \ref{valuesDy1-Stand+} equals
$$\sum_{0\le j\le \lfloor \frac{p-3}{6}\rfloor }  \alpha_j \left( - \frac{\frac{p-3}{2}-3j}{6}\right) \sigma_2^{\frac{p-3}{2}-3j-1}\sigma_3^{2j+1}=0.$$
This tells us that $\alpha_j=0$ for all $0\le j\le \lfloor \frac{p-3}{6}\rfloor $, except maybe for $j$ such that 
$$\frac{p-3}{2}-3j=0, \quad 0\le j\le \lfloor \frac{p-3}{6}\rfloor .$$
Keeping in mind that the equation $\frac{p-3}{2}-3j=0$ is in $\Bbbk$ but the inequality $0\le j\le \lfloor \frac{p-3}{6}\rfloor $ is in $\mathbb{N}_0$, this is equivalent to 
$$j=\frac{p-3}{6}\in \mathbb{N}_0.$$
However, this means that 
$6|p-3$ so $3|p,$
which is impossible as $p$ is a prime and $p>3$. In conclusion, no such $j$ exists, so $\alpha_j=0$ for all $j$. 

\item The coefficient of $q \otimes \left(b_++b_-\right)$ in the equation $D_{y_1}(v_+)=0$ expanded as in Lemma \ref{valuesDy1-Stand+} equals
$$\sum_{0\le j\le \lfloor  \frac{p-1}{6}\rfloor }\beta_j \left(\frac{-2j}{2}\right) \sigma_2^{\frac{p-1}{2}-3j}\sigma_3^{2j-1}  
+\sum_{0\le j\le \lfloor  \frac{p-5}{6}\rfloor }  \gamma_j (-3) \left(\frac{p-5}{2}-3j \right) \sigma_2^{\frac{p-5}{2}-3j-1}\sigma_3^{2j+1} =0.$$
We rewrite this as
$$-\sum_{0\le j\le \lfloor \frac{p-7}{6}\rfloor }\beta_{j+1} \left(j+1\right) \sigma_2^{\frac{p-7}{2}-3j}\sigma_3^{2j+1}  
+\left(\frac{-3}{2}\right) \sum_{0\le j\le \lfloor  \frac{p-5}{6}\rfloor }  \gamma_j  \left(p-5-6j \right) \sigma_2^{\frac{p-7}{2}-3j}\sigma_3^{2j+1} =0.$$
The boundaries of these two sums are different precisely when there exists an integer $j$ with $\frac{p-7}{6}<j\le \frac{p-5}{6}$, which is equivalent to $6j=p-6$ (impossible as then $6|p$ and $p$ is prime) or $6j=p-5$ (happens exactly when $p\equiv 2 \pmod 3$). In that case, the coefficient of $\gamma_{\frac{p-5}{6}}$ is zero, so for any $p$ the above equation becomes
$$\sum_{0\le j\le \lfloor \frac{p-7}{6}\rfloor } \left( -\beta_{j+1} \left(j+1\right) +\left(\frac{-3}{2}\right) \gamma_j  \left(p-5-6j \right)\right) \sigma_2^{\frac{p-7}{2}-3j}\sigma_3^{2j+1} =0.$$
This means that for all $0\le j\le \lfloor \frac{p-7}{6}\rfloor$ we have 
$$ \gamma_j=\frac{2(j+1)}{3(6j+5)} \beta_{j+1}.$$
\end{enumerate}
\end{proof}

We use the notation $\delta_{p\equiv 2 \pmod 3}$ for the delta function. Using Lemma \ref{Stand-generic-alpha,gamma}, 
any vector $v_+$ of the form \eqref{eqn-v+-unknowns} which satisfies $D_{y_1}(v_+)=0$ is of the form 
\begin{align}\label{eqn-v+-unknowns-2}
v_+&=\sum_{0\le j\le \lfloor \frac{p-1}{6}\rfloor}\beta_j \sigma_2^{\frac{p-1}{2}-3j}\sigma_3^{2j} \cdot (-b_+\otimes b_++3b_-\otimes b_-)+ \\
&+\sum_{0\le j\le \lfloor \frac{p-7}{6}\rfloor}  \frac{2(j+1)}{3(6j+5)} \beta_{j+1} \sigma_2^{\frac{p-5}{2}-3j}\sigma_3^{2j+1} \left(-(-b_+^2+3b_-^2)\otimes b_++3(2b_+b_-)\otimes b_- \right)+\notag\\
&+\delta_{p\equiv 2 \bmod 3}\cdot \gamma_{\frac{p-5}{6}}\sigma_3^{\frac{p-2}{3}} \left(-(-b_+^2+3b_-^2)\otimes b_++3(2b_+b_-)\otimes b_- \right)+\notag\\
&+\sum_{0\le j\le \lfloor\frac{p-3}{6}\rfloor}\delta_j \sigma_2^{\frac{p-3}{2}-3j}\sigma_3^{2j} \cdot q \otimes b_-.\notag
\end{align}

Next, we want to establish the conditions on $\beta_j$, $\delta_j$, and (where needed) $\gamma_{\frac{p-5}{6}}$.

\begin{lemma}\label{p>3,t=1,generic,Stand,equations}
For a vector $v_+$ of the form \eqref{eqn-v+-unknowns-2} the condition $D_{y_1}(v_+)=0$ is equivalent to the following systems of equations: 
\begin{enumerate}
\item If $p\equiv 1 \bmod 3$: 

For all $0\le j\le \frac{p-1}{6}-1$
\begin{equation}
-\frac{6j+1}{2}\beta_j-\frac{4(j+1)(2j+1)}{3(6j+5)} \beta_{j+1}+c\delta_j=0 \tag{I} \end{equation}

For all $0\le j\le \frac{p-1}{6}-2$
\begin{equation}
\tag{II}
\frac{12(j+1)c}{(6j+5)} \beta_{j+1}  +  \frac{2(j+1)}{3}\delta_{j+1} +\frac{3(6j+3)}{4}\delta_j =0, 
 \end{equation}

\begin{equation}
\tag{II'}  c \cdot \beta_{\frac{p-1}{6}}  -3\delta_{\frac{p-1}{6}-1} =0. 
 \end{equation}

\item If $p\equiv 2 \bmod 3$:

For all $0\le j\le \frac{p-5}{6}-1$
\begin{equation}
-\frac{6j+1}{2}\beta_j-\frac{4(j+1)(2j+1)}{3(6j+5)} \beta_{j+1}+c\delta_j=0\tag{I}
 \end{equation}
 
 \begin{equation}
2\beta_{\frac{p-5}{6}}+\frac{4}{3}\gamma_{\frac{p-5}{6}}+c\delta_{\frac{p-5}{6}}=0 \tag{I'}
 \end{equation}

For all $0\le j\le \frac{p-5}{6}-1$
 \begin{equation}
\tag{II}\frac{12(j+1)c}{(6j+5)} \beta_{j+1} + \frac{2(j+1)}{3} \delta_{j+1}  +\frac{3(6j+3)}{4} \delta_j=0
 \end{equation}

 \begin{equation}
\tag{II'} \gamma_{\frac{p-5}{6}} \cdot 18c +\frac{-3}{2}\delta_{\frac{p-5}{6}} =0
 \end{equation}

\end{enumerate}
\end{lemma}
\begin{proof}
We look at all the right hand sides of the values of Dunkl operators on basis vectors from Lemma \ref{valuesDy1-Stand+}, and setting their coefficients in $D_{y_1}(v_+)$ to be $0$ get the following equations: 
\begin{itemize}
\item[(I)] From the coefficient of $\otimes (b_+-3b_-)$:  
\begin{align*}%\label{eqn-v+-unknowns-2}
0=&\sum_{0\le j\le \lfloor \frac{p-1}{6}\rfloor}(\frac{-1}{2})\beta_j \sigma_2^{\frac{p-1}{2}-3j}\sigma_3^{2j} + \\
&+\sum_{0\le j\le \lfloor \frac{p-7}{6}\rfloor}  \frac{2(j+1)}{3(6j+5)} \beta_{j+1}\cdot (-2)(2j+1) \sigma_2^{\frac{p-5}{2}-3j+2}\sigma_3^{2j} +\\
&+\sum_{0\le j\le \lfloor \frac{p-7}{6}\rfloor}  \frac{2(j+1)}{3(6j+5)} \beta_{j+1}\cdot 9(\frac{p-5}{2}-3j) \sigma_2^{\frac{p-5}{2}-3j-1}\sigma_3^{2j+2} +\\
&+\delta_{p\equiv 2 \bmod 3}\cdot \gamma_{\frac{p-5}{6}}(-2)(\frac{p-2}{3})\sigma_2^2\sigma_3^{\frac{p-2}{3}-1} +\\
&+\sum_{0\le j\le \lfloor\frac{p-3}{6}\rfloor}\delta_j \cdot c \sigma_2^{\frac{p-3}{2}-3j+1}\sigma_3^{2j}=\\
=&\sum_{0\le j\le \lfloor \frac{p-1}{6}\rfloor}(\frac{-1}{2})\beta_j \sigma_2^{\frac{p-1}{2}-3j}\sigma_3^{2j} +\sum_{0\le j\le \lfloor \frac{p-7}{6}\rfloor}  \frac{-4(j+1)(2j+1)}{3(6j+5)} \beta_{j+1} \sigma_2^{\frac{p-1}{2}-3j}\sigma_3^{2j} +\\
&+\sum_{0\le j\le \lfloor \frac{p-1}{6}\rfloor}  (-3j) \beta_{j} \sigma_2^{\frac{p-1}{2}-3j}\sigma_3^{2j} +\delta_{p\equiv 2 \bmod 3}\cdot \gamma_{\frac{p-5}{6}}\frac{(-2)(p-2)}{3}\sigma_2^2\sigma_3^{2\cdot \frac{p-5}{6}} +\\
&+\sum_{0\le j\le \lfloor\frac{p-3}{6}\rfloor}\delta_j \cdot c \sigma_2^{\frac{p-1}{2}-3j}\sigma_3^{2j}.
\end{align*}
We now distinguish two cases: 
\begin{enumerate}
\item If $p\equiv 1 \bmod 3$, then $\lfloor\frac{p-3}{6}\rfloor=\frac{p-7}{6}$, $\delta_{p\equiv 2 \bmod 3}=0$. Reading the coefficient of $\sigma_2^{\frac{p-1}{2}-3j}\sigma_3^{2j}$ we get that the above equation is equivalent to requiring that for all $0\le j\le \frac{p-7}{6}$
\begin{equation*}%\label{Stand-gen-recursion-I-case1-v1}
-\frac{6j+1}{2}\beta_j+\frac{-4(j+1)(2j+1)}{3(6j+5)} \beta_{j+1}+c\delta_j=0.\end{equation*}
For $j=\frac{p-1}{6}$ the coefficient of $\sigma_2^{\frac{p-1}{2}-3j}\sigma_3^{2j}$ gives $0=0$, which is always satisfied.
\item If $p\equiv 2 \bmod 3$, then $\lfloor\frac{p-1}{6}\rfloor=\lfloor\frac{p-3}{6}\rfloor=\frac{p-5}{6}$, $\lfloor\frac{p-7}{6}\rfloor=\frac{p-5}{6}-1$ and $\delta_{p\equiv 2 \bmod 3}=1$. Reading the coefficient of $\sigma_2^{\frac{p-1}{2}-3j}\sigma_3^{2j}$ we get that the above equation is equivalent to requiring that for all $0\le j\le \frac{p-5}{6}-1$
\begin{equation*}%\label{Stand-gen-recursion-I-case2-v1}
\frac{-(6j+1)}{2}\beta_j+\frac{-4(j+1)(2j+1)}{3(6j+5)} \beta_{j+1}+c\delta_j=0.\end{equation*}
For $j=\frac{p-5}{6}$ the coefficient of $\sigma_2^{2}\sigma_3^{2j}$ gives one additional equation
\begin{equation*}%\label{Stand-gen-recursion-I-case2-v2}
\frac{-p+4}{2}\beta_j+\frac{-2}{3}(p-2) \gamma_{\frac{p-5}{6}}+c\delta_{\frac{p-5}{6}}=0.\end{equation*}
Using that $p=0\in \Bbbk$ we get exactly the equations (I) and (I') from the Lemma.
\end{enumerate}

\item[(II)] From the coefficient of $(b_+\otimes b_- -b_-\otimes b_-)$: 
\begin{align*}\label{eqn-v+-unknowns-2}
0&=\sum_{0\le j\le \lfloor \frac{p-7}{6}\rfloor}  \frac{2(j+1)}{3(6j+5)} \beta_{j+1} \cdot 18c \cdot  \sigma_2^{\frac{p-5}{2}-3j}\sigma_3^{2j+1}+\delta_{p\equiv 2 \bmod 3}\cdot \gamma_{\frac{p-5}{6}} \cdot 18c \cdot \sigma_3^{2\frac{p-5}{6}+1}+ \\
&+\sum_{0\le j\le \lfloor\frac{p-9}{6}\rfloor}\delta_{j+1} \frac{2(j+1)}{3} \sigma_2^{\frac{p-5}{2}-3j}\sigma_3^{2j+1} +\sum_{0\le j\le \lfloor\frac{p-3}{6}\rfloor}\delta_j \frac{-3(\frac{p-3}{2}-3j)}{2} \sigma_2^{\frac{p-5}{2}-3j}\sigma_3^{2j+1}.
\end{align*}
We again distinguish two cases: 
\begin{enumerate}
\item If $p\equiv 1 \bmod 3$, then for all $0\le j\le \frac{p-1}{6}-2$
\begin{equation*}
\frac{12(j+1)c}{(6j+5)} \beta_{j+1}  +  \frac{2(j+1)}{3}\delta_{j+1} +\frac{3(6j+3)}{4}\delta_j =0,
\end{equation*}
and for $j=\frac{p-1}{6}-1$ we get an additional equation 
\begin{equation*}
c \cdot \beta_{\frac{p-1}{6}}  -3\delta_{\frac{p-1}{6}-1} =0.
\end{equation*}
\item If $p\equiv 2 \bmod 3$, then for all $0\le j\le \frac{p-5}{6}-1$
\begin{equation*}
\frac{12(j+1)c}{(6j+5)} \beta_{j+1} + \frac{2(j+1)}{3} \delta_{j+1}  +\frac{3(6j+3)}{4} \delta_j=0,
\end{equation*}
and for $j=\frac{p-5}{6}$ we get an additional equation 
\begin{equation*}
 \gamma_{\frac{p-5}{6}} \cdot 18c +\frac{-3}{2}\delta_{\frac{p-5}{6}} =0.
\end{equation*}
This gives us the equations (II) and (II') from the Lemma.

\item[(III)] The coefficient of $b_+\otimes b_+ - 3b_-\otimes b_+ + 3 b_+\otimes b_- + 3 b_-\otimes b_- $ in $D_{y_1}(v_+)=0$ is always $0$ when $v_+$ is of the form \eqref{eqn-v+-unknowns-2}. 

\item[(IV)] The coefficient of $(-b_+^2+3b_-^2)\otimes b_+ - 3(2b_+b_-)\otimes b_+ + 3 (-b_+^2+3b_-^2)\otimes b_- + 3 (2b_+b_-)\otimes b_- $ is $0$ precisely when conditions (I) and (I') are satisfied. 
\end{enumerate}
\end{itemize}
\end{proof}

\begin{lemma}\label{p>3,t=1,generic,Stand,solutions}
For every $p>3$, the system of equations from the statement of Lemma \ref{p>3,t=1,generic,Stand,equations} has a unique solution up to overall scaling. 
\end{lemma}
\begin{proof}
\begin{enumerate}
\item If $p\equiv 1 \bmod 3$ and we write $p=6k+1$, then the unknowns are $\beta_0, \beta_1, \ldots, \beta_k$ and $\delta_0, \delta_1, \ldots, \delta_{k-1}$. Ordering them as 
$$\beta_k,\delta_{k-1}, \beta_{k-1}, \ldots, \delta_0,\beta_0,$$
we can treat equations (I),(II),(II') as recursions that let us calculate each unknown from the previous two. More precisely, choose an arbitrary nonzero $\beta_k$, then use (II') to calculate
$$\delta_{k-1}=\frac{c}{3} \cdot \beta_{k}.$$
After that, alternating (I) and (II) lets us calculate the remaining unknowns recursively: 
$$\delta_j=\frac{-4}{3(6j+3)} \left( \frac{12(j+1)c}{(6j+5)} \beta_{j+1}  +  \frac{2(j+1)}{3}\delta_{j+1}\right),$$
$$\beta_j=\frac{2}{6j+1} \left(\frac{-4(j+1)(2j+1)}{3(6j+5)} \beta_{j+1}+c\delta_j\right).$$

\item If $p\equiv 2 \bmod 3$ and we write $p=6k+5$, then the unknowns are $\beta_0, \beta_1, \ldots, \beta_k$, $\gamma_k$, and $\delta_0, \delta_1, \ldots, \delta_{k}$. We choose an arbitrary nonzero $\gamma_k$,
use (II') to calculate
$$\delta_{k} =12c \cdot \gamma_{k}$$
and (I') to calculate
$$\beta_{k}=-\frac{2}{3}\gamma_{k}-\frac{c}{2}\delta_{k}.$$
After that, alternating (I) and (II) lets us calculate the remaining unknowns recursively, for $0\le j\le k-1$ in decreasing order of $j$, as
$$ \delta_j=\frac{-16(j+1)c}{3(6j+5)(2j+1)} \beta_{j+1} - \frac{8(j+1)}{27(2j+1)} \delta_{j+1}  $$
$$\beta_j=\frac{-8(j+1)(2j+1)}{3(5+6j)(6j+1)} \beta_{j+1}+\frac{2c}{6j+1}\delta_j.$$
\end{enumerate}
\end{proof}

Summarising, we get the following: 
\begin{lemma}\label{p>3,t=1,generic,Stand}
For every $p>3$ and $c\notin \mathbb{F}_p$, the space of singular vectors in $M^p_{1,c}(\Stand)$ is two dimensional, spanned by $v_+,v_-$, where $v_+$ is of the form \eqref{eqn-v+-unknowns-2} with coefficients satisfying the equations from Lemma \ref{p>3,t=1,generic,Stand,equations}, and $v_-=\frac{2}{3}(s_2+\frac{1}{2}).v_+.$
\end{lemma}

Lemma \eqref{eqn-v+-unknowns-2} only determines $v_+$ up to overalll scaling, exhibited by the choice of $\beta_k$ or $\gamma_k$ at the start of the recursion. Without loss of generality those $\beta_k$ or $\gamma_k$ can be chosen as nonzero constants for every $p$, so that they don't depend on $c$. The other coefficients will then be polynomials in $c$.

\begin{example}When $p=5$ a solution is 
$$\gamma_0=1,\delta_0=2c, \beta_0=1-c^2,$$
leading to 
\begin{align*}
v_+&=(1-c^2) \sigma_2^2 \cdot (-b_+\otimes b_++3b_-\otimes b_-)+ \\
&+\sigma_3\left(-(-b_+^2+3b_-^2)\otimes b_++3(2b_+b_-)\otimes b_- \right)+2c \sigma_2\cdot q \otimes b_-,\\
v_-&=(1-c^2) \sigma_2^2 \cdot (b_+\otimes b_-+b_-\otimes b_+)+ \\
&+\sigma_3\left((-b_+^2+3b_-^2)\otimes b_-+(2b_+b_-)\otimes b_+ \right)+c \sigma_2\cdot q \otimes b_+.
\end{align*}
\end{example}

\begin{example}When $p=7$ a solution is
$$\beta_1=2,\,\, \delta_0=3c,\,\,\beta_0=6c^2-2$$
leading to 
\begin{align*}
v_+&=\left((6c^2-2) \sigma_2^{3}+2 \sigma_3^2 \right)\cdot (-b_+\otimes b_++3b_-\otimes b_-)+ \\
&+4 \sigma_2^{\frac{p-5}{2}}\sigma_3 \left(-(-b_+^2+3b_-^2)\otimes b_++3(2b_+b_-)\otimes b_- \right)+\\
&+3c \sigma_2^{2}\cdot q \otimes b_-,\\
v_-&=\left((6c^2-2) \sigma_2^{3}+2 \sigma_3^2 \right)\cdot (b_+\otimes b_-+b_-\otimes b_+)+ \\
&+4 \sigma_2^{\frac{p-5}{2}}\sigma_3 \left((-b_+^2+3b_-^2)\otimes b_-+(2b_+b_-)\otimes b_+ \right)+\\
&-c \sigma_2^{2}\cdot q \otimes b_-.
\end{align*}
\end{example}

\begin{example}When $p=11$ a solution is
$$\gamma_1=3,\,\, \delta_1=3c,\,\, \beta_1=4c^2+9,\,\, \delta_0=c(6c^2+1),\,\, \beta_0=c^4+5c^2+4$$
leading to 
\begin{align*}
v_+&=\left((c^4+5c^2+4)\sigma_2^5+(4c^2+9)\sigma_2^{2}\sigma_3^{2}\right) \cdot (-b_+\otimes b_++3b_-\otimes b_-)+ \\
&+ (2c^2-1) \sigma_2^{3}\sigma_3 \left(-(-b_+^2+3b_-^2)\otimes b_++3(2b_+b_-)\otimes b_- \right) \\
&+c(6c^2+1) \sigma_2^{4} \cdot q \otimes b_-+3c \sigma_2\sigma_3^{2} \cdot q \otimes b_-,\\
v_-&=\left((c^4+5c^2+4)\sigma_2^5+(4c^2+9)\sigma_2^{2}\sigma_3^{2}\right) \cdot (b_+\otimes b_-+b_-\otimes b_+)+ \\
&+ (2c^2-1) \sigma_2^{3}\sigma_3 \left((-b_+^2+3b_-^2)\otimes b_-+(2b_+b_-)\otimes b_+ \right) \\
&-4c(6c^2+1) \sigma_2^{4} \cdot q \otimes b_+-c \sigma_2\sigma_3^{2} \cdot q \otimes b_+.
\end{align*}
\end{example}

\begin{example}When $p=13$ a solution is
$$\beta_2=3,\delta_1=c,\beta_1=9 +4 c^2,\delta_0=c(7c^2-2),\beta_0=c^4+6c^2+3,$$
leading to 
\begin{align*}
v_+&=\left( (c^4+6c^2+3) \sigma_2^{6} +(9 +4 c^2) \sigma_2^{3}\sigma_3^{2} +3 \sigma_3^{4} \right)\cdot (-b_+\otimes b_++3b_-\otimes b_-)+ \\
&+\left(  (9 +4 c^2) \sigma_2^{4}\sigma_3-2 \sigma_2\sigma_3^{3} \right)\cdot \left(-(-b_+^2+3b_-^2)\otimes b_++3(2b_+b_-)\otimes b_- \right)+\notag\\
&+\left( c(7c^2-2) \sigma_2^{5}+ c \sigma_2^{2}\sigma_3^{2}\right)  \cdot q \otimes b_-,\\
v_-&=\left( (c^4+6c^2+3) \sigma_2^{6} +(9 +4 c^2) \sigma_2^{3}\sigma_3^{2} +3 \sigma_3^{4} \right)\cdot (b_+\otimes b_-+b_-\otimes b_+)+ \\
&+\left(  (9 +4 c^2) \sigma_2^{4}\sigma_3-2 \sigma_2\sigma_3^{3} \right)\cdot \left((-b_+^2+3b_-^2)\otimes b_-+(2b_+b_-)\otimes b_+ \right)+\notag\\
&+\frac{-1}{3}\left( c(7c^2-2) \sigma_2^{5}+ c \sigma_2^{2}\sigma_3^{2}\right)  \cdot q \otimes b_+.
\end{align*}
\end{example}

\subsection{Calculating a determinant}
The previous section shows that the space of singular vectors in $M^p_{1,c}(\Stand)$ is two dimensional, spanned by $v_+,v_-$ from Lemma \ref{p>3,t=1,generic,Stand}, and isomorphic to $\Stand$ via $b_\pm\mapsto v_\pm$.
We would now like to understand the submodule generated by $v_+,v_-$, using an argument similar to that in the proof of Lemma \ref{pne3t=0cne0char}.

Let us name their components as
$$v_+=a_{++}\otimes b_++a_{+-}\otimes b_-$$
$$v_-=a_{-+}\otimes b_++a_{--}\otimes b_-.$$
The determinant that we will have to compute to understand the submodule $\left<v_+,v_-\right>$ is 
$$Det=a_{++}a_{--}-a_{+-}a_{-+} \in S^{2p}\h^*.$$

\begin{prop}\label{p>3,Det1}
\begin{enumerate}
\item $Det$ is an invariant in degree $2p$.
\item $Det$ is in the kernel of $\partial_{y_1}$.
\item $Det$ is of the form $Det=f(c)\cdot \sigma_2^p$, where $f(c)\in \Bbbk[c]$ is a polynomial in $c$.  
\end{enumerate}
\end{prop}
\begin{proof}
The proof does not use the explicit form of $v_+$ from the previous subsection; everything follows from the conditions
$$D_{y_1}v_+=D_{y_1}v_-=0, \quad s_1.v_+=v_+, s_1.v_-=-v_-, \quad v_-=\frac{2}{3}(s_2+\frac{1}{2}).v_+.$$
\begin{enumerate}
\item 
The facts that 
$s_1.v_+=v_+, s_1.v_-=-v_-$
imply that 
$$s_1.a_{++}=a_{++}, \quad s_1.a_{+-}=-a_{+-}, \quad s_1.a_{-+}=-a_{-+}, \quad s_1.a_{--}=a_{--}.$$
From this immediately follows that 
\begin{align*}
s_1.Det%&=s_1.(a_{++}a_{--}-a_{+-}a_{-+})\\
%&=(s_1.a_{++})( s_1.a_{--})-(s_1.a_{+-})(s_1.a_{-+})\\
&=a_{++}a_{--}-(-a_{+-})(-a_{-+})%\\
%&=a_{++}a_{--}-a_{+-}a_{-+}\\
=Det.
\end{align*}

To prove the $s_2$ invariance of $Det$, we expand
$v_-=\frac{2}{3}(s_2+\frac{1}{2})v_+$
to get
%\begin{align*}
%a_{-+}\otimes b_++a_{--}\otimes b_-&=\frac{2}{3}(s_2+\frac{1}{2}id)(a_{++}\otimes b_++a_{+-}\otimes b_-)\\
%&=\frac{2}{3}\left( (s_2.a_{++})\otimes \frac{-b_++3b_-}{2}+(s_2.a_{+-})\otimes \frac{b_++b_-}{2}+ \frac{1}{2}a_{++}\otimes b_++ \frac{1}{2}a_{+-}\otimes b_-\right)
%\end{align*}
%which leads to: 
\begin{align}\label{a-}
a_{-+}&=\frac{1}{3}\bigg( -(s_2.a_{++})+(s_2.a_{+-})+a_{++} \bigg) \\
a_{--}&=\frac{1}{3}\bigg( 3(s_2.a_{++})+(s_2.a_{+-})+a_{+-} \bigg) \notag
\end{align}
We can now rewrite $Det$ as: 
\begin{align*}
Det%&=a_{++}a_{--}-a_{+-}a_{-+}\\
%&=a_{++}\cdot \frac{1}{3}\bigg( 3(s_2.a_{++})+(s_2.a_{+-})+a_{+-} \bigg) - a_{+-}\cdot \frac{1}{3}\bigg( -(s_2.a_{++})+(s_2.a_{+-})+a_{++} \bigg)\\
&=a_{++}(s_2.a_{++})+\frac{1}{3}a_{++}(s_2.a_{+-})+\frac{1}{3}a_{+-}(s_2.a_{++})-\frac{1}{3}a_{+-}(s_2.a_{+-}),
\end{align*}
making it clear that it is an $s_2$ invariant. %So, $Det$ an $S_3$ invariant. 

\item 
We can use formulas \eqref{a-} to deduce: 
\begin{align}\label{s2a+}
(s_2.a_{++})&=\frac{1}{4}\bigg( a_{++}-a_{+-}-3a_{-+}+3a_{--} \bigg) \\
(s_2.a_{+-})&=\frac{1}{4}\bigg( -3a_{++}-a_{+-}+9a_{-+}+3a_{--} \bigg).  \notag
\end{align}
Similarly, the formula for the action of $s_2$ on $v_-$ gives us: 
\begin{align}\label{s2a-}
(s_2.a_{-+})&=\frac{1}{4}\bigg( -a_{++}+a_{+-}-a_{-+}+a_{--} \bigg) \\
(s_2.a_{--})&=\frac{1}{4}\bigg( 3a_{++}+a_{+-}+3a_{-+}+a_{--} \bigg).  \notag
\end{align}
From this we can also easily deduce formulas for $(13).a_{++}$ etc. 

As we know the action of the group on all $a_{\pm \pm}$, we can expand $D_{y_1}(v_{\pm})=0$ to get
%\begin{align*}
%\partial_1(a_++)\otimes b_++\partial_1(a_+-)\otimes b_-&=\frac{-c}{x_1-x_2}\cdot 2a_{+-}\otimes b_-+\\
%&+\frac{c}{x_1-x_3}\cdot\bigg( (a_{++}-s_1s_2a_{++})\otimes \frac{-b_+-3b_-}{2}+(a_{+-}+s_1s_2a_{+-})\otimes \frac{-b_++b_-}{2} \bigg)
%\end{align*}
%from which it follows that
\begin{align}\label{partiala}
\partial_1(a_{++})&= \frac{-c}{2(x_1-x_3)}\bigg(a_{+-}-3a_{-+} \bigg)\\
\partial_1(a_{+-})&= \frac{-2c}{x_1-x_2} a_{+-}+\frac{c}{2(x_1-x_3)}\bigg(-3a_{++}+2a_{+-}+3a_{--} \bigg) \notag \\
%\end{align}
%Similarly expanding $D_{y_1}(v_-)=0$ carefully gives us
%\begin{align}\label{partiala-}
\partial_1(a_{-+})&= \frac{2c}{x_1-x_2} a_{-+}-\frac{c}{2(x_1-x_3)}\bigg(-a_{++}+2a_{-+}+a_{--} \bigg) \notag \\
\partial_1(a_{--})&= \frac{c}{2(x_1-x_3)}\bigg(a_{+-}-3a_{-+}\bigg).  \notag
\end{align}

Using \eqref{partiala} we calculate
\begin{align*}
\partial_1(Det)&=\partial_1(a_{++})a_{--}+\partial_1(a_{--})a_{++}-\partial_1(a_{+-})a_{-+}-\partial_1(a_{-+})a_{+-}=0.
\end{align*}

\item By (1) $Det$ is an invariant of degree $2p$, and by (2) it is in the kernel by $\partial_1$. As it is symmetric, it follows that it is also in the kernel by $\partial_2$ and $\partial_3$, so $Det$ is a $p$-th power of a polynomial. The only symmetric polynomials in degree $2p$ which are also $p$-th powers are scalar multiples of $\sigma_2^p$. This scalar depends on $c$. As the coefficients of $v_+,v_-$ depend on $c$ polynomially, so does $Det$.  
\end{enumerate}
\end{proof}

In order to analyse the module $M_{1,c}(\Stand)/\left< v_+,v_-\right>$, we will need to show that $v_+,v_-$ are (for generic $c$ and some special $c$) as independent as they can possibly be, and this will be done by showing that $Det$ is nonzero (for generic $c$ and some special $c$). This will include analysing the polynomial $f$ in $c$ and its zeroes. More on that in the next section, but for now let us prove the following Lemma.

\begin{lemma}\label{p>3,Det2}
Let $f(c)\in \Bbbk[c]$ be the polynomial from Proposition \ref{p>3,Det1}. 
\begin{enumerate}
\item It satisfies $f(c)=\frac{4}{3}\left( 3\beta_0-\delta_0\right)\left( 3\beta_0+\delta_0\right)$; 
\item All coefficients $\beta_j$ are even and all $\delta_j$ are odd polynomials in $\Bbbk[c]$.  
\item If $p=6k+1$ then $\deg \beta_j=2k-2j$, $\deg \delta_j=2k-2j-1$, and $\deg f=\frac{p-1}{3}$.
\item If $p=6k+5$ then $\deg \beta_j=2k-2j+2$, $\deg \delta_j=2k-2j+1$, and $\deg f=\frac{p+1}{3}$.
\item The polynomial $f(c)$ is not identically zero. 
\end{enumerate}
\end{lemma}
\begin{proof}
\begin{enumerate}
\item Proposition \ref{p>3,Det1} (3) showed that $Det=a_{++}a_{--}-a_{+-}a_{-+}$ is a scalar multiple of $\sigma_2^p$, so we now calculate the scalar $f(c)$ by considering only the parts of $a_{++}$, $a_{--}$, $a_{+-}$ and $a_{-+}$ which will contribute to a power of $\sigma_2$. More specifically, any term divisible by $\sigma_3$ will not contribute to $Det$. Disregarding all such terms, we can write: 
\begin{align*}
a_{++}&=\beta_0\sigma_2^{\frac{p-1}{2}}\cdot (-b_+)+\sigma_3 \cdot (\ldots)\\
a_{+-}&=\beta_0\sigma_2^{\frac{p-1}{2}}\cdot 3b_-+\delta_0\sigma_2^{\frac{p-3}{2}}\cdot q +\sigma_3 \cdot (\ldots)\\
a_{-+}&=\beta_0\sigma_2^{\frac{p-1}{2}}\cdot b_--\frac{1}{3}\delta_0\sigma_2^{\frac{p-3}{2}}\cdot q+\sigma_3 \cdot (\ldots)\\
a_{--}&=\beta_0\sigma_2^{\frac{p-1}{2}}\cdot b_+ +\sigma_3 \cdot (\ldots).
\end{align*} 
From here we then compute: 
\begin{align*}
Det%&= a_{++}a_{--}-a_{+-}a_{-+}\\
&=\beta_0\sigma_2^{\frac{p-1}{2}} (-b_+)  \beta_0\sigma_2^{\frac{p-1}{2}} b_+ - \left(\beta_0\sigma_2^{\frac{p-1}{2}} 3b_-+\delta_0\sigma_2^{\frac{p-3}{2}} q \right) \left( \beta_0\sigma_2^{\frac{p-1}{2}} b_--\frac{1}{3}\delta_0\sigma_2^{\frac{p-3}{2}} q\right)+\sigma_3 \cdot (\ldots)\\
&=\beta_0^2\sigma_2^{p-1}(-b_+^2-3b_-^2)+\frac{1}{3}\delta_0^2\sigma_2^{p-3}q^2+\sigma_3 \cdot (\ldots)\\
&=12 \beta_0^2 \sigma_2^p+\frac{1}{3}\delta_0^2\sigma_2^{p-3}(-27\sigma_3^2-4\sigma_2^3)+\sigma_3 \cdot (\ldots)\\
&=\left( 12 \beta_0^2 -\frac{4}{3}\delta_0^2 \right) \sigma_2^p+\sigma_3 \cdot (\ldots).
\end{align*}
Comparing this with $Det=f(c)\sigma_2^p$ we get that 
$$f(c)=12 \beta_0^2 -\frac{4}{3}\delta_0^2 =\frac{4}{3}\left( 3\beta_0-\delta_0\right)\left( 3\beta_0+\delta_0\right).$$

\item Clear from the recursion in Lemma \ref{p>3,t=1,generic,Stand,equations}.
\item Clear from the recursion in Lemma \ref{p>3,t=1,generic,Stand,equations}.
\item It is enough to see that the polynomial $f(c)$ is nonzero at a specific point, so set  $c=0$. Then $\delta_j=0$ for all $j$ and $f(0)=12\beta_0^2$, which is nonzero by the recursion in Lemma \ref{p>3,t=1,generic,Stand,equations}.
\end{enumerate}
\end{proof}

We will now use this determinant to analyse the modules $M_{1,c}(\Stand)/\left< v_\pm\right>$ and $M_{1,c}(\Stand)/\left< v_\pm, \sigma_3^p\otimes b_{\pm}\right>$.

\begin{lemma}\label{p>3,t=1,generic,Stand-character} Let $c$ be generic. 
\begin{enumerate}
\item The submodule $\left<v_+,v_-\right>$ of $M_{1,c}(\Stand)$ is isomorphic to $M_{1,c}(\Stand)[-p]$.
\item The vectors $\sigma_2^p\otimes b_+, \sigma_2^p\otimes b_-$ are in $\left<v_+,v_-\right>$. 
\item The vectors $\sigma_3^p\otimes b_+$ and $\sigma_3^p\otimes b_-$ are not in $\left<v_+,v_-\right>$. 
\item The module $M_{1,c}(\Stand)/\left<v_\pm,\sigma_3^p\otimes b_\pm \right>$ has the character 
$$\chi_{M_{1,c}(\Stand)/\left<v_\pm,\sigma_3^p\otimes b_\pm \right>}(z)=\chi_{S\h^*}(z)\cdot [\Stand]\cdot (1-z^p)(1-z^{3p})$$
and the Hilbert series 
$$h_{M_{1,c}(\Stand)/\left<v_\pm,\sigma_3^p\otimes b_\pm \right>}(z)=2\frac{(1-z^p)(1-z^{3p})}{(1-z)^2}.$$
\item The module $M_{1,c}(\Stand)/\left<v_\pm,\sigma_3^p\otimes b_\pm \right>$ is irreducible and thus equal to $L_{1,c}(\Stand)$.
\end{enumerate}
\end{lemma}

\begin{proof}
\begin{enumerate}
\item The $S_3$-map $\varphi:\Stand\to M_{1,c}^p(\Stand)$ given by $\varphi(b_{\pm})=v_{\pm}$ induces a map of $H_{1,c}(S_3,\h)$ modules $\varphi:M_{1,c}(\Stand)[-p]\to M_{1,c}(\Stand)$ whose image is $\left< v_+,v_-\right>$. Let us show its kernel is zero.

Assume that $A,B\in S\h^*$ are homogeneous and that $\varphi(A\otimes b_++B\otimes b_-)=0$. This can be written as
%$$A\cdot v_++B\cdot v_-=0,$$
%and rewritten again as the system in $S\h^*$
\begin{align*}
A\cdot a_{++}+B\cdot a_{-+}&=0\\
A\cdot a_{+-}+B\cdot a_{--}&=0.
\end{align*}
Considering this as a linear system of equations, we note that its determinant is, by Theorem \ref{p>3,Det1} and Lemma \ref{p>3,Det2}, equal to $f(c)\sigma_2^p$, which is nonzero for generic $c$, and thus for generic $c$ the only solution of this system is $A=B=0$.

\item Checking if $\sigma_2^p\otimes b_+$ is in the $H_{t,c}(S_3,\h)$ submodule $\left<v_+,v_-\right>$ is equivalent to trying to find $A,B\in S\h^*$ such that
$$Av_++Bv_-=\sigma_2^p\otimes b_+.$$
This can be rewritten as the system
\begin{align*}
A\cdot a_{++}+B\cdot a_{-+}&=\sigma_2^p\\
A\cdot a_{+-}+B\cdot a_{--}&=0,
\end{align*}
which, for generic $c$ where $f(c)\ne 0$, has a unique solution
\begin{align*}
A&=\frac{\sigma_2^p\cdot a_{--}}{Det}=\frac{1}{f(c)}\cdot a_{--}\\
B&=\frac{-\sigma_2^p\cdot a_{+-}}{Det}=-\frac{1}{f(c)}\cdot a_{+-}.
\end{align*}
This shows that whenever $f(c)\ne 0$, $\sigma_2^p\otimes b_+$ is in $\left<v_+,v_-\right>$. The vector $\sigma_2^p\otimes b_-$ can be obtained from $\sigma_2^p\otimes b_+$ by the $S_3$ action, so the claim follows.

\item Similarly, checking if the vector $\sigma_3^p\otimes b_+$ is in $\left<v_+,v_-\right>$ is equivalent to checking if there arehomogeneous $A,B\in S\h^*$ such that 
$$Av_++Bv_-=\sigma_3^p\otimes b_+,$$
which is a system with a unique rational solution
\begin{align*}
A&=\frac{\sigma_3^p\cdot a_{--}}{f(c)\sigma_2^p}\\
B&=\frac{-\sigma_3^p\cdot a_{+-}}{f(c)\sigma_2^p}.
\end{align*}
These $A,B$ are rational function in $\h^*$ and not elements of $S\h^*$, and they are the only solutions when $f(c)\ne 0$. So, we conclude that $\sigma_3^p\otimes b_\pm$ do not lie in $\left<v_\pm\right>$.

\item The character $\chi_{M_{1,c}(\Stand)/\left<v_\pm,\sigma_3^p\otimes b_\pm \right>}(z)$ equals
\begin{align*}
\chi_{M_{1,c}(\Stand)}(z)-\chi_{\left<v_\pm\right>}(z)-\chi_{\left<\sigma_3^p\otimes \pm \right>}(z)+\chi_{\left<v_\pm\right>\cap \left<\sigma_3^p\otimes b_\pm \right>}(z).
\end{align*}
By (1), $\left<v_+,v_- \right>\cong M_{1,c}(\Stand)[-p]$ and $\left<\sigma_3^p\otimes b_+, \sigma_3^p\otimes b_- \right>\cong M_{1,c}(\Stand)[-3p]$, so
$$\chi_{\left<v_\pm\right>}(z)=\chi_{S\h^*}[\Stand]\cdot z^p, \,\, \chi_{\left<\sigma_3^p\otimes b_\pm \right>}(z)=\chi_{S\h^*}[\Stand]\cdot z^{3p},$$
and the task is to describe $\left<v_\pm\right>\cap \left<\sigma_3^p\otimes b_\pm \right>$.

Let $v\in \left<v_\pm\right>\cap \left<\sigma_3^p\otimes b_\pm \right>$ be arbitrary. There exist $A,B,C,D\in S\h^*$ such that
\begin{equation}\label{eq-intersection}
v=Av_++Bv_-=C\sigma_3^p\otimes b_++D\sigma_3^p\otimes b_-.
\end{equation}
As before, we get
\begin{align*}
A&=\frac{C \cdot a_{--}+D \cdot a_{-+}}{f(c)\sigma_2^p}\sigma_3^p \\
B&=\frac{-C \cdot a_{+-}+D \cdot a_{++}}{f(c)\sigma_2^p}\sigma_3^p .
\end{align*}
Given that $\sigma_2$ and $\sigma_3$ are algebraically independent so their powers are coprime, it follows that $\sigma_2^p$ needs to divide $C \cdot a_{--}+D \cdot a_{-+}$ and $-C \cdot a_{+-}+D \cdot a_{++}$, and $\sigma_3^p$ needs to divide both $A$ and $B$. In other words, there are $A',B'\in S\h^*$ such that 
$$A=A'\cdot \sigma_3^p, \quad B=B'\cdot \sigma_3^p.$$
Rewriting equation \eqref{eq-intersection} we get
\begin{equation*}
v=A' \sigma_3^p  v_++B' \sigma_3^p  v_-=C\sigma_3^p\otimes b_++D\sigma_3^p\otimes b_-.
\end{equation*}
For any choice of $A',B'\in S\h^*$ there are unique $C,D\in S\h^*$ satisfying this, given by 
$$C=A' a_{++}+B' a_{-+}, \quad 
D=A' a_{+-}+B' a_{--}.$$
This lets us conclude that 
$$\left<v_\pm\right>\cap \left<\sigma_3^p\otimes b_\pm \right>=\left<\sigma_3^p  v_\pm \right>, \quad \chi_{\left<\sigma_3^pv_\pm\right>}(z)=\chi_{M_{1,c}(\Stand)}(z)\cdot z^{4p},$$
and the character of $M_{1,c}(\Stand)/\left<v_\pm,\sigma_3^p\otimes b_\pm \right>$ follows.

\item The module $L_{1,c}(\Stand)$ is a quotient of $M_{1,c}(\Stand)/\left<v_\pm,\sigma_3^p\otimes b_\pm \right>$ of dimension $6p^2$. The above shows that 
$$\dim M_{1,c}(\Stand)/\left<v_\pm,\sigma_3^p\otimes b_\pm \right>(z)=h_{M_{1,c}(\Stand)/\left<v_\pm,\sigma_3^p\otimes b_\pm \right>}(1)=6p^2,$$ so they are equal.
\end{enumerate}
\end{proof}

\section{Irreducible representations of $H_{t,c}(S_3,\h)$ in characteristic $p>3$ for special $c$}
\label{sect-p>3-special}

We now turn our attention to special values of the parameter $c$, which are not covered by Theorem \ref{thm-p>3generic}. The value of the characteristic remains $p>3$. When $t=0$ the only special value of $c$ is $c=0$, and when $t=1$ the special values of $c$ are $
0,1,2\ldots, p-1$. We will sometimes consider these $c$ as elements of $\mathbb{Z}$ to allow us to write inequalities such as $0<c<p/6$, and at other times (when they are coefficients in the algebra), we consider them as elements of $\mathbb{F}_p\subseteq \Bbbk$. 

The aim of this section is to prove the following theorem. 
\begin{theorem}\label{thm-p>3special}
The characters of the irreducible representations $L_{t,c}(\tau)$ of the rational Cherednik algebra $H_{t,c}(S_3,\h)$ over an algebraically closed field of characteristic $p>3$, for all special $c$, $t=0,1$, and any $\tau$, are given by the following tables. 

Characters: 
\begin{center}
{\renewcommand*{\arraystretch}{1.5}\begin{tabular}{ c |c }
$p>3$ & $\Triv$   \\ \hline \hline 
 $t=0$, $c=0$ & $[\Triv]$  \\ \hline \hline
 $t=1$, $c=0$  & $\chi_{S^{(p)}\h^*}(z)$     \\ \hline
 $t=1, 0<c<p/3 $ & $\chi_{S\h^*}(z)\cdot (1-z^{3c+p}[\Stand]+z^{2(3c+p)}[\Sign])$  \\ 
  &    \scriptsize{[Lian]}     \\  \hline
 $t=1, p/3<c<p/2$ & $\chi_{S\h^*}(z)\cdot (1-z^{3c-p}[\Stand]+z^{2(3c-p)}[\Sign])$  \\
   &    \scriptsize{[Lian]}      \\ \hline 
 $t=1,p/2<c<2p/3$  & $\chi_{S\h^*}(z)(1-z^{6c-3p}[\Sign])(1-z^p)$  \\
    &          \\  \hline
 $t=1,2p/3<c<p$  & $\chi_{S\h^*}(z)\cdot (1-z^{3c-2p}[\Stand]+z^{2(3c-2p)}[\Sign])$   \\
   &    \scriptsize{[Lian]}  
 \end{tabular}}
\end{center}

\begin{center}
{\renewcommand*{\arraystretch}{1.5}\begin{tabular}{ c |c }
$p>3$ &  $\Sign$  \\ \hline \hline 
 $t=0$, $c=0$ & $[\Sign]$  \\ \hline \hline
 $t=1$, $c=0$  & $\chi_{S^{(p)}\h^*}(z)[\Sign]$    \\ \hline
 $t=1, 0<c<p/3$  & $\chi_{S\h^*}(z)\cdot ([\Sign]-z^{p-3c}[\Stand]+z^{2(p-3c)}[\Triv])$  \\\hline
 $t=1, p/3<c<p/2$  & $\chi_{S\h^*}(z)([\Sign]-z^{3p-6c}[\Triv])(1-z^p)$  \\ \hline
 $t=1, p/2<c<2p/3$  & $\chi_{S\h^*}(z)\cdot ([\Sign]-z^{2p-3c}[\Stand]+z^{2(2p-3c)}[\Triv])$ \\ \hline
 $t=1, 2p/3<c<p$   & $\chi_{S\h^*}(z)\cdot ([\Sign]-z^{4p-3c}[\Stand]+z^{2(4p-3c)}[\Triv])$    
 \end{tabular}}
\end{center}

\begin{center}

{\renewcommand*{\arraystretch}{1.5}\begin{tabular}{ c  |c }
$p>3$ &  $\Stand$ \\ \hline \hline 
 $t=0$, $c=0$ &  $[\Stand]$ \\ \hline \hline
 $t=1, c=0$   & $\chi_{S^{(p)}\h^*}(z)\cdot [\Stand]$  \\ \hline 
 $t=1,$ &  $\chi_{S\h^*}(z)\cdot \left([\Stand]-[\Triv]z^{p-3c}-\right.$ \\
  $0<c<p/3$ &  $\left. -[\Stand]z^p-[\Sign]z^{3c+p}+[\Stand]z^{2p}\right)$ \\ \hline    
 $t=1,$ & $\chi_{S\h^*}(z)\cdot \left( [\Stand]-z^{-p+3c}[\Sign]-\right.$ \\
$ p/3<c<p/2$ &  $ \left. -z^{3p-3c}[\Triv]-z^{p+3c}[\Sign]-z^{5p-3c}[\Triv]+z^{4p}[\Stand] \right)$\\ 
 & (Conjecture) \\ \hline
  $t=1,$ & $\chi_{S\h^*}(z)\left( [\Stand]-z^{2p-3c}[\Triv]-\right.$ \\
$ p/2<c<2p/3$ &   $ \left.- z^{3c}[\Sign]-z^{4p-3c}[\Triv]-z^{2p+3c}[\Sign]+z^{4p}[\Stand] \right)$\\
 & (Conjecture) \\ \hline
 $t=1,$ & $\chi_{S\h^*}(z)\cdot \left([\Stand]-[\Sign]z^{3c-2p}-\right.$ \\
 $ 2p/3<c<p$ & $ \left. -[\Stand]z^p-[\Triv]z^{4p-3c}+[\Stand]z^{2p}\right)$ 
 \end{tabular}}
\end{center}

\vspace{0.5cm}
Hilbert polynomials: 
\begin{center}
{\renewcommand*{\arraystretch}{1.5}\begin{tabular}{ c |c |c |c }
$p>3$ & Triv & Sign & Stand \\ \hline \hline 
 $t=0$, $c=0$ & 1 & 1 & 2 \\ \hline 
 $t=1$, $c\in \mathbb{F}_p$  & &  &   \\
 $c=0$ & $\left(\frac{1-z^{p}}{1-z}\right)^2$ & $\left(\frac{1-z^{p}}{1-z}\right)^2$ & $2\left(\frac{1-z^{p}}{1-z}\right)^2$ \\
 $0<c<p/3$ & $\left(\frac{1-z^{3c+p}}{1-z}\right)^2$ & $\left(\frac{1-z^{p-3c}}{1-z}\right)^2$  & $\frac{1-z^{p-3c}-2z^p -z^{p+3c}+2z^{2p}}{(1-z)^2}$ \\
& \scriptsize{\cite[Thm 3.3]{Li14}} & & \\ 
 $p/3<c<p/2$ & $\left(\frac{1-z^{3c-p}}{1-z}\right)^2$  & $\frac{(1-z^{3p-6c})(1-z^p)}{(1-z)^2}$  & $\frac{1-z^{-p+3c}-z^{3p-3c}-z^{p+3c}-z^{5p-3c}+2z^{4p}}{(1-z)^2}$\\
 & \scriptsize{\cite[Thm 3.3]{Li14}} & & \scriptsize{(Conjecture)} \\  
 $p/2<c<2p/3$ & $\frac{(1-z^{6c-3p})(1-z^p)}{(1-z)^2}$  & $\left(\frac{1-z^{2p-3c}}{1-z}\right)^2$ & $\frac{1-z^{2p-3c}-z^{3c}-z^{4p-3c}-z^{2p+3c}+2z^{4p}}{(1-z)^2}$\\
&  & & \scriptsize{(Conjecture)} \\
 $2p/3<c<p$ & $\left(\frac{1-z^{3c-2p}}{1-z}\right)^2$  & $\left(\frac{1-z^{4p-3c}}{1-z}\right)^2$  & $\frac{1-z^{3c-2p}-2z^p -z^{4p-3c}+2z^{2p}}{(1-z)^2}$ \\
& \scriptsize{\cite[Thm 3.3]{Li14}} & & \\ 
 \end{tabular}}
\end{center}

In all cases, the singular vectors are known and given in the proofs. 
\end{theorem}

To write these characters explicitly, recall that, from Corollaries \ref{characterShp>3} and \ref{characterSphp>3},
\begin{align*}\chi_{S\h^*}(z)&=\frac{1}{(1-z^2)(1-z^3)}\left([\Triv]+(z+z^2)[\Stand]+z^3[\Sign]\right),\\
\chi_{S^{(p)}\h^*}(z)&=\chi_{S\h^*}(z)\cdot (1-z^p[\Stand]+z^{2p}[\Sign]).
\end{align*}

\begin{proof} For all $\tau$ and $t$, the case $c=0$ is standard, and explained in Lemma \ref{t=c=0} for $t=0$ and Lemmas \ref{t=1,c=0} for $t=1$. 

For $\tau=\Triv$, the remaining computations for $t=1$, $c\ne 0 \in \mathbb{F}_p$ fall into several cases depending on where $c$ lies in the set $0,1,\ldots, p-1$. The paper \cite{Li14} deals with all these intervals except one, $p/2<c<2p/3$, where they give conjectured degrees of the generators. We explain the work of \cite{Li14} in Section \ref{sect-Lian} and work out the remaining case $p/2<c<2p/3$ in Theorem \ref{thm-sect-p>3,Triv,specialc} (giving the generators of the maximal proper graded submodule, the Hilbert polynomial, the character, and showing this is a complete intersection).

For $\tau=\Sign$, the character formulas follow from the character formulas for $\Triv$ by Lemma \ref{Sign}. 

For $\tau=\Stand$ $t=1$ and $c\in \mathbb{F}_p$, the case $0<c<p/3$ is done with all the proofs in Section \ref{section-p>3,Stand,specialc-1}, the case $p/3<c<p/2$ is done conjecturally and with no proofs in Section \ref{section-p>3,Stand,specialc-2}, and the cases $p/2<c<p$ follow from them using Lemma \ref{Sign} and the fact that $\Stand\otimes \Sign\cong \Stand$.

\end{proof}

\subsection{Results of Lian \cite{Li14}}\label{sect-Lian}
The paper \cite{Li14} deals with some special values of $c$ for $\tau=\Triv$. Here we restate their Theorem 3.3. in our conventions. 

\begin{prop}[\cite{Li14} Theorem 3.3.]
Let $\Bbbk$ be an algebraically closed field of characteristic $p>3$, $t=1$ and $c\in \mathbb{F}_p$, and consider the rational Cherednik algebra $H_{1,c}(S_3,\h)$. The degrees of the generators of the maximal proper graded submodule of the Verma module $M_{1,c}(\Triv)$ are as follows: 
\begin{enumerate}
\item $0<c<p/3$: two generators of degree $3c+p$, 
\item $p/3<c<p/2$: two generators of degree $3c-p$,  
\item $2p/3<c<p$: two generators of degree $3c-2p$.  
\end{enumerate}
These generators are given explicitly in the proof of \cite{Li14} Theorem 3.3. In each case the irreducible quotient $L_{1,c}(\Triv)$ is a complete intersection, so letting $a\in \mathbb{N}$ be the degree of the generators, its character and Hilbert polynomial are
\begin{align*}\chi_{L_{1,c}(\Triv)}(z)&=\chi_{S\h^*}(z)\cdot (1-z^{a}[\Stand]+z^{2a}[\Sign]),\\
h_{L_{1,c}(\Triv)}(z)&=\left(\frac{1-z^{a}}{1-z}\right)^2.
\end{align*}
\end{prop}

\begin{remark}
Note some differences between this proposition and the theorem as stated in \cite{Li14}. Firstly, \cite{Li14} works with the reflection representation $V$ and we work with $\h$; the dictionary for translating results from one setting to the other is given in Proposition \ref{prop-dictionaryVvsh} and consequently we list one generator fewer ($\sigma_1^p$, which is $0$ in our setting). Secondly, there is a typo in the statement of Case (3) of \cite{Li14}; the correct degrees are listed here and the proof in \cite{Li14} is correct. Thirdly, we calculate the characters, which is straightforward from the information \cite{Li14} provide. Fourthly, the results of \cite{Li14} are missing for $p/2<c<2p/3$, where they conjecturally give the degrees of the generators. We address this in detail in Section \ref{sect-p>3,Triv,specialc}, calculating the generators,the character of the quotient, and showing the quotient is irreducible. Finally, we want to point out that the results of \cite{Li14} in the cases $0<c<p/2$ and $2p/3<c<p$ lack three steps in the proofs: 1) a minor step showing, (in their notation) that $D_{2}(G_1)=0$, which is completely straightforward; 2) the proof that the quotient is a complete intersection and thus its Hilbert polynomial is as stated, which they explicitly assume in Theorem 3.3.; and 3) the proof that the quotient is irreducible, which relies on the yet unpublished results by Roman Bezrukavnikov and Andrei Okounkov. We do not do this here, but note that these results all appear correct and can be verified using similar methods to our work in the $p/2<c<2p/3$ case.
\end{remark}

\subsection{Irreducible representation $L_{1,c}(\Triv)$}\label{sect-p>3,Triv,specialc}

The aim of this section is to describe $L_{1,c}(\Triv)$ in the non-modular case $p>3$ for the values of $c$ that \cite{Li14} does not cover, namely $p/2<c<2p/3$. We fix those values of the parameter throughout the section. The main result is: 

\begin{theorem}\label{thm-sect-p>3,Triv,specialc}
The irreducible representation $L_{1,c}(\Triv)$ of the rational Cherednik algebra $H_{1,c}(S_3,\h)$ over an algebraically closed field of characteristic $p>3$ for $c\in \mathbb{F}_p$, $p/2<c<2p/3$, is the quotient of the Verma module $M_{1,c}(\Triv)$ by the submodule generated by the vectors
\begin{align*}
v_{6c-3p}&=q^{2c-p}=(-4\sigma_2^3-27\sigma_3^2)^{c-\frac{p+1}{2}}q\\
v_p&=\sigma_2^{2p-3c}\sigma_3 \sum_{j=0}^{c-\frac{p+1}{2}} {c-\frac{p+1}{2} \choose j}\frac{1}{2j+1}(-4\sigma_2^3)^{c-\frac{p+1}{2}-j} (-27\sigma_3^2)^j.
\end{align*}
This quotient is a complete intersection, and its character and Hilbert polynomial are
\begin{align*}
\chi_{L_{1,c}(\Triv)}(z)&=\chi_{S\h^*}(z)(1-z^{6c-3p}[\Sign]-z^p[\Triv]+z^{6c-2p}[\Sign])\\
h_{L_{1,c}(\Triv)}(z)&=\frac{(1-z^{6c-3p})(1-z^p)}{(1-z)^2}.
\end{align*}
\end{theorem}
\begin{proof}
Firstly, let us use the Casimir element $\Omega$. If $M$ is a $H_{t,c}(S_3,\h)$ module equal to $M_{1,c}(\Triv)$ or one of its quotients, and $\tau\subseteq M^l$ is an irreducible $S_3$-subrepresentation consisting of singular vectors, then by Lemma \ref{OmegaAction} $\Omega$ acts on $\tau$ by the following scalar: 
\begin{equation}\Omega|_\Triv=0\cdot \mathrm{id}, \quad 
\Omega|_\Sign=6c \cdot \mathrm{id}, \quad 
\Omega|_\Stand=3c \cdot \mathrm{id}.\label{Omega-repeated}\end{equation}
At the same time, because $\tau\subseteq M^l$, $\Omega$ acts on $\tau$ by $l\cdot \mathrm{id}$. 
As a consequence, singular vectors  of type $\Triv$ in $M$ can only appear in degrees of the form $kp$, singular vectors of type $\Sign$ in degrees of the form $6c+kp$, and singular vectors of type $\Stand$ in degrees of the form $3c+kp$. 

Secondly, note that for $c\in \mathbb{F}_p$ with $\frac{p}{2}<c<\frac{2p}{3}$ we have
\begin{equation}0<6c-3p<3c-p<p,\label{Triv-inequality}\end{equation} and these are the only integers of the form $kp,3c+kp$ and $6c+kp$ with $k\in \mathbb{Z}$ in the interval $[0,p]$. 

These two facts together now spell out a strategy for describing $L_{1,c}(\Triv)$: starting from the lowest degrees, we test for singular vectors. The first fact tells us in which degrees can the singular vectors appear, and in which isotypic component. Then we use the basis from Theorem \ref{decmposeShp>3} to reduce the size of the space where we are looking for the singular vectors, by only looking at specific isotypic components in specific degrees. The second fact tells us in which order to look at these graded pieces. This order is important because as soon as we take a quotient by some submodule in some degree, all calculations in higher degrees need to be done in the quotient module and no longer in the Verma module. This is because there could be singular vectors in the quotient of the Verma module which do not lift to a singular vector in the Verma module. (See Lemma \ref{Triv-special-lemma4} for an example of this.)

By equations \eqref{Omega-repeated} and \eqref{Triv-inequality}, the first space to consider is the $\Sign$ isotypic component in degree $6c-3p$. In Lemma \ref{Triv-special-lemma1} we show that 
$$v_{6c-3p}=q^{2c-p}=(-4\sigma_2^3-27\sigma_3^2)^{c-\frac{p+1}{2}}q$$
is a singular vector in degree $6c-3p$, generating a submodule of $M_{1,c}(\Triv)$ isomorphic to $M_{1,c}(\Sign)[-(6c-3p)]$. In Lemma \ref{Triv-special-lemma2} we show there are no other singular vectors in degree $6c-3p$. 

Next, we take the quotient $M_{1,c}(\Triv)/\left<v_{6c-3p}\right>$ and continue examining the degrees as dictated by equations \eqref{Omega-repeated} and \eqref{Triv-inequality}. In Lemma \ref{Triv-special-lemma3} we show that there are no singular vectors in degree $3c-p$ of the module $M_{1,c}(\Triv)/\left<v_{6c-3p}\right>$. In Lemma \ref{Triv-special-lemma4} we show that there is a one dimensional space of singular vectors in degree $p$ of the module $M_{1,c}(\Triv)/\left<v_{6c-3p}\right>$, in the $\Triv$ isotypic component and spanned by $$v_{p}=\sigma_2^{2p-3c}\sigma_3 \sum_{j=0}^{c-\frac{p+1}{2}} {c-\frac{p+1}{2} \choose j}\frac{1}{2j+1}(-4\sigma_2^3)^{c-\frac{p+1}{2}-j} (-27\sigma_3^2)^j.$$ We note that $v_p$ is only singular in the quotient $M_{1,c}(\Triv)/\left<q^{2c-p}\right>$, and is not an image of a singular vector in $M_{1,c}(\Triv)$. 

In Lemma \ref{complete-intersection} 
we show that $M_{1,c}(\Triv)/\left<v_{6c-3p},v_p\right>$ is a complete intersection and calculate its character and Hilbert series. Its Hilbert series turns out to be $$h_{L_{1,c}(\Triv)}(z)=\frac{(1-z^{6c-3p})(1-z^p)}{(1-z)^2}.$$
By equations \eqref{Omega-repeated} and \eqref{Triv-inequality}, the next degree to consider and search for singular vectors is $6c-2p$. However, the above Hilbert series shows that the maximal degree of $M_{1,c}(\Triv)/\left<v_{6c-3p},v_p \right>$ is $6c-3p-1+p-1=6c-2p-2<6c-2p$. This means the quotient $M_{1,c}(\Triv)/\left<v_{6c-3p},v_p \right>$ has no singular vectors, and is thus irreducible and equal to $L_{1,c}(\Triv)$.

\end{proof}

We now proceed with the details of the above proof, listed as a sequence of lemmas. We keep the assumptions on the parameters $t,c,p$ listed at the start of Section \ref{sect-p>3,Triv,specialc}. 

\begin{lemma}\label{Triv-special-lemma1}
The vector 
$$v_{6c-3p}=q^{2c-p}=(-4\sigma_2^3-27\sigma_3^2)^{c-\frac{p+1}{2}}q$$
is singular in $M_{1,c}^{6c-3p}(\Triv)$, and generates a subrepresentation $\left< v_{6c-3p}\right>$ isomorphic to $M_{1,c}(\Sign)[-(6c-3p)]$.
\end{lemma}
\begin{proof}
The vector $v_{6c-3p}$ spans a copy of $\Sign$ inside $M_{1,c}^{6c-3p}(\Triv)$, so by Lemma \ref{KerD1} to check it is singular it is enough to show $D_{y_1}(v_{6c-3p})=0$. We check this directly, using computations from Section \ref{sect-auxiliary} and the fact that $q^{2c-p-1}$ is a symmetric polynomial.
\begin{align*}
D_{y_1}(q^{2c-p})&=(2c-p)q^{2c-p-1}\partial_{y_1}(q)-cq^{2c-p-1}\frac{id-(12)}{x_1-x_2}(q)-cq^{2c-p-1}\frac{id-(13)}{x_1-x_3}(q)\\
&=cq^{2c-p-1}\left( 2\cdot \frac{-1}{4}(-b_+^2+3b_-^2)+2\cdot \frac{1}{4}(2b_+b_-)+2 \sigma_2+\frac{1}{3} (-b_+^2+3b_-^2) - \right.\\
& \quad  \quad  \quad \left. -2 \sigma_2+\frac{1}{6} (-b_+^2+3b_-^2)-\frac{1}{2}\cdot 2b_+b_- \right) =0.\\
\end{align*}

This shows $v_{6c-3p}$ is singular, so it induces a map
$$\varphi:M_{1,c}(\Sign)[-(6c-3p)]\to M_{1,c}(\Triv).$$
Seen as a map $S\h^*[-(6c-3p)]\cong M_{1,c}(\Sign)[-(6c-3p)] \to M_{1,c}(\Triv)\cong S\h^*$ this is multiplication by $q^{2c-p}$ which is injective. So, $\left< v_{6c-3p}\right> \cong M_{1,c}(\Sign)[-(6c-3p)]$.
\end{proof}

\begin{lemma}\label{Triv-special-lemma2}
The only singular vectors in $M_{1,c}^{6c-3p}(\Triv)$ are multiples of $v_{6c-3p}$. 
\end{lemma}
\begin{proof}
By Theorem \ref{decmposeShp>3} %a basis for the $\Sign$ isotypic component of $M_{1,c}^{6c-3p}(\Triv)$ is 
%$$\left\{ \sigma_2^a \sigma_3^b\cdot q \,\,\mid \,\, 2a+3b+3=6c--3p\right\},$$
%so 
we are looking for all singular vectors of the form 
$$v=\sum_{i=0}^{c-\frac{p+1}{2}} \alpha_i \sigma_2^{3i} \sigma_3^{2c-p-1-2i}\cdot q.$$
The condition $D_{y_1}(v)=0$ yields, using Section \ref{sect-auxiliary}, the system
$$\alpha_{i+1}=  \frac{2^2}{3^3}\frac{(c-\frac{p+1}{2}-i)}{i+1} \alpha_i,  \quad  \quad  0\le i \le c-\frac{p+1}{2},$$
which has a unique (up to overalll scaling) solution  
$$\alpha_i=(-1)^{c-\frac{p+1}{2}} 4^i\cdot 27^{c-\frac{p+1}{2}-i}\cdot {{c-\frac{p+1}{2}}\choose i} , \quad 0\le i \le c-\frac{p+1}{2},$$
leading to 
\begin{align*}
v&=(-1)^{c-\frac{p+1}{2}} \sum_{i=0}^{c-\frac{p+1}{2}} {{c-\frac{p+1}{2}}\choose i} 4^i  \sigma_2^{3i} \cdot 27^{c-\frac{p+1}{2}-i} \sigma_3^{2(c-\frac{p+1}{2}-i)}\cdot q\\
&=(-4  \sigma_2^{3} - 27 \sigma_3^{2})^{c-\frac{p+1}{2}}\cdot q=q^{2c-p}=v_{6c-3p}.
\end{align*}
\end{proof}

\begin{lemma}\label{Triv-special-lemma3}
There are no singular vectors in degree $3c-p$ of $M_{1,c}(\Triv)/\left<v_{6c-3p}\right>$.\end{lemma}
\begin{proof}
It is straightforward to show that there are no singular vectors in degree $3c-p$ of $M_{1,c}(\Triv)$, using a similar computation to the above and the basis 
\begin{equation}\label{basis1}\left\{\sigma_2^a\sigma_3^bb_+ \mid 2a+3b+1=3c-p \right\}\cup \left\{\sigma_2^a\sigma_3^b(-b_+^2+3b_-^2) \mid 2a+3b+2=3c-p \right\},
\end{equation}
 of the $S_2$ invariant part of the $\Stand$ isotypic component of $M_{1,c}^{3c-p}(\Triv)$. However, this lemma claims more - that there are no singular vectors modulo $v_{6c-3p}=q^{2c-p}$. The basis \eqref{basis1} is not well suited to taking this quotient and the calculations are more involved, so we first change basis to the one in which taking this quotient becomes very easy. We then prove the lemma by direct computation. We distinguish two cases, depending on the remainder of $p$ modulo $3$. 

{\bf Case 1. $p\equiv 2 \pmod 3$}
The new basis of the $S_2$ invariant part of the $\Stand$ isotypic component of $M_{1,c}^{3c-p}(\Triv)$ we propose to use in this case is 
\begin{equation}\label{basis4} \left\{q^{2i}\sigma_3^{c-\frac{p+1}{3}-2i} b_+ \mid 0\le i \le \frac{c}{2}-\frac{p+1}{6} \right\} \cup \left\{q^{2i+1}\sigma_3^{c-\frac{p+1}{3}-2i-1} b_- \mid 0\le i \le \frac{c}{2}-\frac{p+1}{6}-\frac{1}{2} \right\}.
\end{equation}
This set is indeed a basis as it lies in the $S_2$ invariant part of the $\Stand$ isotypic component of $M_{1,c}^{3c-p}(\Triv)$, is linearly independent and has the correct number of elements.

After taking the quotient by $\left<q^{2c-p}\right>$ the set \eqref{basis4} is reduced to 
\begin{equation}\label{basis5} \left\{q^{2i}\sigma_3^{c-\frac{p+1}{3}-2i} b_+ \mid 0\le i \le c-\frac{p+1}{2} \right\} \, \cup \,  \left\{q^{2i+1}\sigma_3^{c-\frac{p+1}{3}-2i-1} b_- \mid 0\le i \le c-\frac{p+1}{2}-1 \right\}.
\end{equation}
This set spans the $S_2$ invariant part of the $\Stand$ isotypic component of the module $M_{1,c}^{3c-p}(\Triv)/\left<q^{2c-p}\right>$ and contains $2c-p$ elements. On the other hand, the multiplicity of $\Stand$ in 
$M_{1,c}^{3c-p}(\Triv)$ is $\frac{3c-p+2}{3}$, the multiplicity of $\Stand$ in $M_{1,c}^{3c-p}(\Sign)[6c-3p]$ is $\frac{2p-3c+2}{3}$, so the multiplicity of $\Stand$ in $M_{1,c}^{3c-p}(\Triv)/\left<q^{2c-p} \right>$ is $$\frac{3c-p+2}{3}-\frac{2p-3c+2}{3}=2c-p.$$ This shows that the set \eqref{basis5} is a basis. We calculate the Dunkl operators in this basis: 
\begin{align}\label{dunkl-trivstand-1}
&D_{y_1}(q^{2i}\sigma_3^{c-\frac{p+1}{3}-2i}\cdot b_+)=\\
\notag &\quad=(1-\frac{3c+p+1}{2})q^{2i}\sigma_3^{c-\frac{p+1}{3}-2i}+\\
\notag &\quad\quad +\frac{1}{2}(c-\frac{p+1}{3}-2i)q^{2i+1}\sigma_3^{c-\frac{p+1}{3}-2i-1}+(-27i)q^{2i-1}\sigma_3^{c-\frac{p+1}{3}-2i+1}+\\
\notag &\quad\quad + (-3i)q^{2i-1}\sigma_3^{c-\frac{p+1}{3}-2i}\sigma_2b_+ + \frac{1}{6}(c-\frac{p+1}{3}-2i)q^{2i}\sigma_3^{c-\frac{p+1}{3}-2i-1}\sigma_2b_++\\
\notag &\quad\quad + (-3i)q^{2i-1}\sigma_3^{c-\frac{p+1}{3}-2i}\sigma_2b_- - \frac{1}{2}(c-\frac{p+1}{3}-2i)q^{2i}\sigma_3^{c-\frac{p+1}{3}-2i-1}\sigma_2b_-,
\end{align}
\begin{align}\label{dunkl-trivstand-2}
&D_{y_1}(q^{2i+1}\sigma_3^{c-\frac{p+1}{3}-2i-1}\cdot b_-)=\\
\notag &\quad=\frac{-1}{6}(c-\frac{p+1}{3}-2i-1)q^{2i+2}\sigma_3^{c-\frac{p+1}{3}-2i-2}+\frac{9}{2}(2i+1-2c)q^{2i}\sigma_3^{c-\frac{p+1}{3}-2i}+\\
\notag &\quad\quad +(1+\frac{3}{2}(c-\frac{p+1}{3}))q^{2i+1}\sigma_3^{c-\frac{p+1}{3}-2i-1}+\\
\notag &\quad\quad + (\frac{-1}{2}(2i+1)+c)q^{2i}\sigma_3^{c-\frac{p+1}{3}-2i-1}\sigma_2b_+- \frac{1}{6}(c-\frac{p+1}{3}-2i-1)q^{2i+1}\sigma_3^{c-\frac{p+1}{3}-2i-2}\sigma_2b_+ +\\
\notag &\quad\quad  + (\frac{3}{2}(2i+1)-3c)q^{2i}\sigma_3^{c-\frac{p+1}{3}-2i-1}\sigma_2b_-- \frac{1}{6}(c-\frac{p+1}{3}-2i-1)q^{2i+1}\sigma_3^{c-\frac{p+1}{3}-2i-2}\sigma_2b_-.
\end{align}

Assume that 
$$w=\sum_{i=0}^{c-\frac{p+1}{2}}\alpha_i q^{2i}\sigma_3^{c-\frac{p+1}{3}-2i}\cdot b_++\sum_{i=0}^{c-\frac{p+1}{2}-1}\beta_i q^{2i+1}\sigma_3^{c-\frac{p+1}{3}-2i-1}\cdot b_- $$ is a singular vector. 
Using \eqref{dunkl-trivstand-1} and \eqref{dunkl-trivstand-2}, we get that the $\Triv$ component of $D_{y_1}(w)$ is $0$ if and only if 
\begin{align*}
& \sum_{i=0}^{c-\frac{p+1}{2}}\alpha_i \left(1-\frac{3c+p+1}{2} \right)q^{2i}\sigma_3^{c-\frac{p+1}{3}-2i}+\\
& \quad + \sum_{i=0}^{c-\frac{p+1}{2}-1}\beta_i \frac{-1}{6}(c-\frac{p+1}{3}-2i-1)q^{2i+2}\sigma_3^{c-\frac{p+1}{3}-2i-2}+\\
& \quad \quad + \sum_{i=0}^{c-\frac{p+1}{2}-1}\beta_i \frac{9}{2}(2i+1-2c)q^{2i}\sigma_3^{c-\frac{p+1}{3}-2i} =0.
\end{align*}
%This equality holds in $M_c(Triv)/\left<q^{2c-p}\right>$ if and only if it holds in $M_c(Triv)$. 
%It can be rewritten as 
%\begin{align*}
%& \sum_{i=0}^{c-\frac{p+1}{2}}\alpha_i \left(1-\frac{3c+p+1}{2} \right)q^{2i}\sigma_3^{c-\frac{p+1}{3}-2i}+\\
%& \quad + \sum_{i=1}^{c-\frac{p+1}{2}}\beta_{i-1} \frac{-1}{6}(c-\frac{p+1}{3}-2i+1)q^{2i}\sigma_3^{c-\frac{p+1}{3}-2i}+\\
%& \quad \quad + \sum_{i=0}^{c-\frac{p+1}{2}-1}\beta_i \frac{9}{2}(2i+1-2c)q^{2i}\sigma_3^{c-\frac{p+1}{3}-2i} =0.
%\end{align*}
This leads to %the system of equations for $\alpha_i,\beta_i$: 
\begin{align}
\alpha_0\cdot \left(1-\frac{3c+p+1}{2} \right) + \beta_0 \cdot \frac{9}{2}(1-2c)&= 0, \label{lemma-nosingvectTrivStand-sys1} \\% \quad i=0  \\
\alpha_i\cdot \left(1-\frac{3c+p+1}{2} \right) + \beta_i \cdot \frac{9}{2}(2i+1-2c)+ \beta_{i-1} \cdot \frac{1}{6}(2i-1-c+\frac{p+1}{3})&= 0, \label{lemma-nosingvectTrivStand-sys2} \\%\quad i=1, \ldots, c-\frac{p+1}{2}-1 \\
\alpha_{c-\frac{p+1}{2}}\cdot \left(1-\frac{3c+p+1}{2} \right) +  \beta_{c-\frac{p+1}{2}-1} \cdot \frac{1}{6}(c-\frac{2(p+1)}{3}-1)&= 0.\label{lemma-nosingvectTrivStand-sys3}  %\quad i=c-\frac{p+1}{2} 
\end{align}
Similarly, \eqref{dunkl-trivstand-1} and \eqref{dunkl-trivstand-2} imply that the $\Sign$ component of $D_{y_1}(w)$ is $0$ if and only if 
%\begin{align*}
%& \sum_{i=0}^{c-\frac{p+1}{2}}\alpha_i \left(c-\frac{p+1}{3}-2i \right)\frac{1}{2}q^{2i+1}\sigma_3^{c-\frac{p+1}{3}-2i-1}+\\
%& \quad + \sum_{i=0}^{c-\frac{p+1}{2}}\alpha_i \left(-27i\right) q^{2i-1}\sigma_3^{c-\frac{p+1}{3}-2i+1}+\\
%& \quad \quad + \sum_{i=0}^{c-\frac{p+1}{2}-1}\beta_i \left(1+\frac{3}{2}(c-\frac{p+1}{2})\right)q^{2i+1}\sigma_3^{c-\frac{p+1}{3}-2i-1} =0,
%\end{align*}
%which holds in $M_c(Triv)/\left<q^{2c-p}\right>$ if and only if it holds in $M_c(Triv)$,  can be rewritten as 
%\begin{align*}
%& \sum_{i=0}^{c-\frac{p+1}{2}}\alpha_i \left(c-\frac{p+1}{3}-2i \right)\frac{1}{2}q^{2i+1}\sigma_3^{c-\frac{p+1}{3}-2i-1}+\\
%& \quad + \sum_{i=0}^{c-\frac{p+1}{2}-1}\alpha_{i+1} \left(-27(i+1)\right) q^{2i+1}\sigma_3^{c-\frac{p+1}{3}-2i-1}+\\
%& \quad \quad + \sum_{i=0}^{c-\frac{p+1}{2}-1}\beta_i \left(1+\frac{3}{2}(c-\frac{p+1}{2})\right)q^{2i+1}\sigma_3^{c-\frac{p+1}{3}-2i-1} =0,
%\end{align*}
%and leads to the system 
\begin{align}
\alpha_0\cdot \left(c-\frac{p+1}{3} \right) + \beta_0 \cdot \left(1+\frac{3}{2}(c-\frac{p+1}{2}) \right)+\alpha_1\cdot (-27)&= 0,  \label{lemma-nosingvectTrivStand-sys4} \\ %\quad i=0  \\
\alpha_i\cdot \left(c-\frac{p+1}{3} -2i\right) + \beta_i \cdot \left(1+\frac{3}{2}(c-\frac{p+1}{2}) \right)+ \alpha_{i+1} \cdot \left(-27(i+1)\right)&= 0, \label{lemma-nosingvectTrivStand-sys5} \\%\quad i=1, \ldots, c-\frac{p+1}{2}-1 \\
\alpha_{c-\frac{p+1}{2}}\cdot \left(-c+\frac{2(p+1)}{3} \right) &= 0 \label{lemma-nosingvectTrivStand-sys6}%, \quad i=c-\frac{p+1}{2}. 
\end{align}

The system of equations \eqref{lemma-nosingvectTrivStand-sys1}-\eqref{lemma-nosingvectTrivStand-sys6} can now be shown to have no nonzero solutions by using \eqref{lemma-nosingvectTrivStand-sys6} to deduce $\alpha_{c-\frac{p+1}{2}}=0$, 
\eqref{lemma-nosingvectTrivStand-sys3} to deduce $\beta_{c-\frac{p+1}{2}-1}=0$, 
then alternating \eqref{lemma-nosingvectTrivStand-sys2} and \eqref{lemma-nosingvectTrivStand-sys5} to show that $\beta_{i-1}=0$ and $\alpha_{i}=0$ for all $i=1, \ldots, c-\frac{p+1}{2}-1$ and finally using either  \eqref{lemma-nosingvectTrivStand-sys1} or  \eqref{lemma-nosingvectTrivStand-sys4} to show that $\alpha_{0}=0$. The only constants we divide by in this calculation are $c-\frac{p+1}{3}-2i$ and $c-\frac{p+1}{3}-2i+1$, which are never equal to $0$ for these $i$ and $\frac{p}{2}<c<\frac{2p}{3}$. This shows that the only solution to the equation $D_{y_1}(w)=0$ in the $S_2$ invariant part of the $\Stand$ isotypic component of $M_{1,c}^{3c-p}(\Triv)/\left<q^{2c-p}\right>$ is $w=0$, proving the claim of the Lemma in Case 1.

{\bf Case 2. $p\equiv 1 \pmod 3$} Similar. The basis of the $S_2$-invariant part of the $\Stand$ isotypic component of $M_{1,c}^{3c-p}(\Triv)/\left<q^{2c-p}\right>$ is
\begin{equation*}\left\{q^{2i}\sigma_3^{c-\frac{p+2}{3}-2i}\cdot (-b_+^2+3b_-^2)\,,\, q^{2i+1}\sigma_3^{c-\frac{p+2}{3}-2i-1}\cdot 2b_+b_- \mid 0\le i < c-\frac{p+1}{2} \right\}.
\end{equation*}

We calculate the Dunkl operators in this basis to be: 
\begin{align*}
&D_{y_1}(q^{2i}\sigma_3^{c-\frac{p+1}{3}-2i}\cdot (-b_+^2+3b_-^2) )=\\
 &\quad= 2(c-\frac{p+2}{3}-2i)q^{2i}\sigma_3^{c-\frac{p+2}{3}-2i-1}\sigma_2^2+(-9i)q^{2i-1}\sigma_3^{c-\frac{p+1}{3}-2i}\sigma_2^2+\\
 &\quad\quad +\textrm{something in the Stand component of } M_{1,c}^{3c-p-1}(\Triv)
\end{align*}
\begin{align*}
&D_{y_1}(q^{2i+1}\sigma_3^{c-\frac{p+1}{3}-2i-1}\cdot 2b_+b_-) )=\\
 &\quad=6(2i+1-2c)q^{2i}\sigma_3^{c-\frac{p+2}{3}-2i-1}\sigma_2^2+2(c-\frac{p+2}{3}-2i-1)q^{2i+1}\sigma_3^{c-\frac{p+1}{3}-2i-2}\sigma_2^2+\\
 &\quad\quad +\textrm{something in the Stand component of } M_{1,c}^{3c-p-1}(\Triv).
\end{align*}

Requiring that a vector 
\begin{equation*}w=\sum_{i=0}^{c-\frac{p+1}{2}}\alpha_i q^{2i}\sigma_3^{c-\frac{p+2}{3}-2i}\cdot (-b_+^2+3b_-^2)+\sum_{i=0}^{c-\frac{p+1}{2}-1}\beta_i q^{2i+1}\sigma_3^{c-\frac{p+2}{3}-2i-1}\cdot 2b_+b_-
\end{equation*}
satisfies $D_{y_1}(w)=0$ leads to to a system of equations for $\alpha_i,\beta_i$. 
By asking that the $\Triv$ component of $D_{y_1}(w)$ is $0$ we get
\begin{align*}
\alpha_{c-\frac{p+1}{2}}&=0\\
\alpha_{i}\cdot2(c-\frac{p+2}{3}-2i)+\beta_i\cdot 6(2i+1-2c) &=0, \qquad i=0, \ldots, c-\frac{p+1}{2}-1
\end{align*}
and by asking that the $\Sign$ component of $D_{y_1}(w)$ is $0$ we get
\begin{align*}
\alpha_{i+1}\cdot (-9)(i+1) +\beta_i\cdot 2(c-\frac{p+2}{3}-2i-1) &=0, \qquad i=0, \ldots, c-\frac{p+1}{2}.
\end{align*}
This system has a unique solution $\alpha_i=0,\beta_i=0$ for all $i$, so $w=0$. This proves the Lemma in case 2. 
\end{proof}

\begin{lemma}\label{Triv-special-lemma4}
There is a one dimensional space of singular vectors in degree $p$ of the module $M_{1,c}(\Triv)/\left<q^{2c-p}\right>$, in the $\Triv$ isotypic component and spanned by $$v_p=\sigma_2^{2p-3c}\sigma_3 \sum_{j=0}^{c-\frac{p+1}{2}} {c-\frac{p+1}{2} \choose j}\frac{1}{2j+1}(-4\sigma_2^3)^{c-\frac{p+1}{2}-j} (-27\sigma_3^2)^j.$$ Note that $v_p$  does not lift to a singular vector in $M_{1,c}(\Triv)$. 
\end{lemma}
\begin{proof}
We will be calculating modulo $$q^{2c-p}=(-4\sigma_2^3-27\sigma_3^2)^{c-\frac{p+1}{2}}q.$$
The basis we are working with in $\Triv$ isotypic component of $M_{1,c}^p(\Triv)$ is $$\left\{\sigma_2^a\sigma_3^b\, \mid \, 2a+3b=p, 0\le a,b\right\}.$$
In $M_{1,c}(\Triv)/\left<q^{2c-p}\right>$ these elements are not linearly independent, as they satisfy
$$\sigma_2^a\sigma_3^b(-4\sigma_2^3-27\sigma_3^2)^{c-\frac{p+1}{2}}q^2=\sigma_2^a\sigma_3^b(-4\sigma_2^3-27\sigma_3^2)^{c-\frac{p+1}{2}+1}=0,$$ for all $a,b$ such that  $2a+3b+6(c-\frac{p+1}{2}+1)=p.$ In particular, in this quotient we can express any power of $\sigma_3^b$  with $b\ge 2c-p+1$ in terms of polynomials of lower degree in $\sigma_3$. The basis of the $\Triv$ isotypic component of $M_{1,c}(\Triv)^p/\left<q^{2c-p}\right>$ is thus
\begin{align*}\left\{\sigma_2^a\sigma_3^b\, \mid \, 2a+3b=p, 0\le a, 0\le b <  2c-p+1\right\}=
\left\{\sigma_2^{\frac{p-3}{2}-3j}\sigma_3^{2j+1}\, \mid \, 0\le j \le c-\frac{p+1}{2}\right\}.
\end{align*}

Write an arbitrary vector $v$ in the $\Triv$ isotypic component of $M_{1,c}^p(\Triv)/\left<q^{2c-p}\right>$ as 
$$v=\sum_{j=0}^{c-\frac{p+1}{2}} \alpha_j \sigma_2^{\frac{p-3}{2}-3j} \sigma_3^{2j+1}$$
for some $\alpha_j\in \Bbbk$ (depending on $c$), and let us check when $v$ is singular. We have 
\begin{align*}
D_{y_1}(\sigma_2^a\sigma_3^b)&=\frac{-a}{6}\sigma_2^{a-1}\sigma_3^b(b_++3b_-)+\frac{b}{36}\sigma_2^a\sigma_3^{b-1}((-b_+^2+3b_-^2)+3\cdot 2b_+b_-), \\
D_{y_1}(v)&=\sum_{j=0}^{c-\frac{p+1}{2}} \alpha_j\left( \frac{-(\frac{p-3}{2}-3j)}{6}\sigma_2^{\frac{p-3}{2}-3j-1}\sigma_3^{2j+1}(b_++3b_-)+\right.\\
& \quad \quad \left. +\frac{2j+1}{36}\sigma_2^{\frac{p-3}{2}-3j}\sigma_3^{2j}((-b_+^2+3b_-^2)+3\cdot 2b_+b_-)\right)\\
&=\sum_{j=0}^{c-\frac{p+1}{2}} \alpha_j \frac{2j+1}{12} \sigma_2^{\frac{p-3}{2}-3j-1}\sigma_3^{2j}\left(3\sigma_3b_++\frac{1}{3}\sigma_2(-b_+^2+3b_-^2)\right)\\
& \quad \quad -\alpha_j \frac{2j+1}{12} \sigma_2^{\frac{p-3}{2}-3j-1}\sigma_3^{2j}\left(  -9\sigma_3b_+   -\sigma_2\cdot 2b_+b_- \right)\\
&=\sum_{j=0}^{c-\frac{p+1}{2}} \alpha_j \frac{2j+1}{12} \sigma_2^{\frac{p-3}{2}-3j-1}\sigma_3^{2j}qb_--\alpha_j \frac{2j+1}{12} \sigma_2^{\frac{p-3}{2}-3j-1}\sigma_3^{2j}qb_+\\
&=\frac{1}{12}\sigma_2^{2p-3c}\left( \sum_{j=0}^{c-\frac{p+1}{2}} \alpha_j (2j+1) \sigma_2^{3(c-\frac{p+1}{2}-j)}\sigma_3^{2j}\right)q\cdot (b_--b_+).
\end{align*}

Recall that we are working in the quotient by $q^{2c-p}$. The above vector $D_{y_1}(v)$ is a multiple of 
$$q^{2c-p}=(-4\sigma_2^3-27\sigma_3^2)^{c-\frac{p+1}{2}}q=\sum_{j=0}^{c-\frac{p+1}{2}}{c-\frac{p+1}{2} \choose j}(-4\sigma_2^3)^{c-\frac{p+1}{2}-j}(-27\sigma_3^2)^j$$
if and only if (up to overalll scaling) we have for all $j$
$$\alpha_j  (2j+1)={c-\frac{p+1}{2} \choose j}(-4)^{c-\frac{p+1}{2}-j}(-27)^j.$$ 
This proves there is (up to scalars) exactly one singular vector in the $\Triv$ isotypic component of $M_{1,c}^p(\Triv)/\left<q^{2c-p}\right>$ equal to 
$$v_p=\sigma_2^{2p-3c}\sigma_3 \sum_{j=0}^{c-\frac{p+1}{2}} {c-\frac{p+1}{2} \choose j}\frac{1}{2j+1}(-4\sigma_2^3)^{c-\frac{p+1}{2}-j} (-27\sigma_3^2)^j.$$
\end{proof}

\begin{remark}
We can rewrite it as
$$v_p=\sigma_2^{2p-3c}\sigma_3\sum_{i=0}^{c-\frac{p+1}{2}} (-4\sigma_2^3-27\sigma_3^2)^{c-\frac{p+1}{2}-i}(-4\sigma_2^3)^i\frac{\prod_{j=0}^{i-1}(2m-2j)}{\prod_{j=0}^{i}(2m+1-2j)}$$
though it is not clear if we gain anything by this nor if there is a closed formula for $v_p$. 

\end{remark}

\begin{remark}
When $c=\frac{1}{2}=\frac{p+1}{2}$, it lies in the interval considered here as $p/2<\frac{p+1}{2}<2p/3$. For this $c$ the above formula says 
$v_p=\sigma_2^{\frac{p-3}{2}}\sigma_3.$
This special case has already been noted in \cite{Li14} Remark 3.5. 
\end{remark}

\begin{lemma}\label{complete-intersection}
The vectors $v_{6c-3p}, v_p$ form a regular sequence, and $M_{1,c}(\Triv)/\left<v_{6c-3p}, v_{p}\right>$ is a complete intersection. Its character and Hilbert polynomial are 
\begin{align*}
\chi_{M_{1,c}(\Triv)/\left<v_{6c-3p}, v_{p}\right>}(z)&=\chi_{S\h^*}(z)(1-z^{6c-3p}[\Sign]-z^p[\Triv]+z^{6c-2p}[\Sign])\\
h_{M_{1,c}(\Triv)/\left<v_{6c-3p}, v_{p}\right>}(z)&=\frac{(1-z^{6c-3p})(1-z^p)}{(1-z)^2}.
\end{align*}

\end{lemma}

To prove lemma \ref{complete-intersection} we will need two auxiliary facts, which will not be used elsewhere. 

\begin{lemma}\label{lemma-aux-1}
For any $m\in \left\{1,2,\ldots, p-1\right\}$ we have
$$\sum_{j=0}^{m} {m \choose j}\frac{1}{2j+1}(-1)^j\ne 0.$$
\end{lemma}
\begin{proof}
Set 
$$f_m(z)=\sum_{j=0}^{m} {m \choose j}\frac{1}{2j+1}(-1)^jz^{2j+1}.$$
The claim is then that $f_{m}(1)\ne 0$.

Differentiating formally we get that 
$$f_m'(z)=\sum_{j=0}^{m} {m \choose j}(-1)^jz^{2j}=(1-z^2)^m,$$
so 
$$f_m(z)=\int (1-z^2)^m dz, \quad f_m(0)=0$$
is the primitive function of $(1-z^2)^m$ with no constant term.

There is a recursive formula for evaluating this integral, given by 
$$\int (1-z^2)^m dz =\frac{z(1-z^2)^m}{2m+1}+\frac{2m}{2m+1}\int (1-z^2)^{m-1}dz.$$
As $\frac{z(1-z^2)^m}{2m+1}|_{z=0}=0$, we have 
$$f_m(z)=\frac{z(1-z^2)^m}{2m+1}+\frac{2m}{2m+1}f_{m-1}(z),$$
and so  $$f_m(1)=0+\frac{2m}{2m+1}f_{m-1}(1)$$
and in particular (using $m<p$ and $p\ne 2$) we get $f_m(1)\neq 0$ if and only if $f_{m-1}(1)\neq 0$. Inductively, it remains to show that $f_1(1)\neq 0$, which is straightforward from 
$$f_1(z)=\int 1-z^2 dz=z-\frac{1}{3}z^3$$
and 
$$f_1(1)=\frac{2}{3}\ne 0.$$

\end{proof}

Before the next lemma, recall that $q^2=-4\sigma_2^3-27\sigma_3^2.$

\begin{lemma}\label{lemma-aux-2}
If $k\in \mathbb{N}$ and $A,B\in  (S\mathfrak{h}^*)^{S_3}$ are symmetric polynomials such that 
$$A(-4\sigma_2^3-27\sigma_3^2)^{k}=Bv_p,$$
then there exists a symmetric polynomial $C$ such that 
$$A=Cv_p, \quad B=C(-4\sigma_2^3-27\sigma_3^2)^{k}.$$
\end{lemma}
\begin{proof}\
This is an statement about symmetric polynomials, which, by the the fundamental theorem of symmetric polynomials form an algebra isomorphic to a polynomial algebra $(S\mathfrak{h}^*)^{S_3}\cong \Bbbk[a,b]$ with the isomorphism $\sigma_2\mapsto a,\sigma_3\mapsto b$. Abusing notation slightly, the problem becomes about $A,B \in \Bbbk[a,b]$ such that 
$$A(-4a^3-27b^2)^{k}=Ba^{2p-3c}b \sum_{j=0}^{c-\frac{p+1}{2}} {c-\frac{p+1}{2} \choose j}\frac{1}{2j+1}(-4a^3)^{c-\frac{p+1}{2}-j} (-27b^2)^j.$$

The polynomial $-4a^3-27b^2$ is irreducible. It does not divide the polynomial $v_p=a^{2p-3c}b \sum_{j=0}^{c-\frac{p+1}{2}} {c-\frac{p+1}{2} \choose j}\frac{1}{2j+1}(-4a^3)^{c-\frac{p+1}{2}-j} (-27b^2)^j$; to see that, plug in the values $a=-3$ and $b=2$ to get that 
$-4(-3)^3-27\cdot 2^2=0$, while by Lemma \ref{lemma-aux-1} 
\begin{align*}
v_p|_{a=-3,b=2}&=(-3)^{2p-3c}\cdot 2\cdot (4\cdot 27)^{c-\frac{p+1}{2}} \sum_{j=0}^{c-\frac{p+1}{2}} {c-\frac{p+1}{2} \choose j}\frac{1}{2j+1}(-1)^j \ne 0. \\
\end{align*}
So, $(-4a^3-27b^2)$ divides $B$, and proceeding inductively we get that $(-4a^3-27b^2)^{k}$ divides $B$. This means there exists $C\in \Bbbk[a,b]$ such that 
$$B=(-4a^3-27b^2)^{k} C.$$
It now follows that 
$$A=C\cdot  a^{2p-3c}b \sum_{j=0}^{c-\frac{p+1}{2}} {c-\frac{p+1}{2} \choose j}\frac{1}{2j+1}(-4a^3)^{c-\frac{p+1}{2}-j} (-27b^2)^j.$$
Replacing $a,b$ by $\sigma_2,\sigma_3$ we get the claim.

\end{proof}

\begin{proof}[Proof of Lemma \ref{complete-intersection}]

We claim that in $S\h^*\cong M_{1,c}(\Triv)$ we have
$$\left< v_{6c-3p}\right> \cap \left< v_p \right> = \left< v_{6c-3p}v_p \right>.$$ This is enough to show that the sequence is regular, as it is equivalent to saying that $v_p$ is not a zero divisor in $S\h^*/\left< v_{6c-3p}\right>$. 

Consider any vector $u\in \left< v_{6c-3p}\right> \cap \left< v_p \right>$. As we are in the non-modular case, by using $S_3$ projection formulas we can assume that $u$ is in an isotypic component of type $\Triv$, $\Sign$, or $\Stand$. It is of the form 
$$u=A v_{6c-3p}=B v_p$$
for some $A,B\in S\mathfrak{h}^*$. 
Recall that $$v_{6c-3p}=q^{2c-p}=(-4\sigma_2^3-27\sigma_3^2)^{c-\frac{p+1}{2}}q,$$
that $v_{6c-3p}$ is in the $\Sign$ isotypic component and that $v_p$ is in the $\Triv$ isotypic component.

{\bf Case 1.} If the vector $u$ is in the $\Sign$ isotypic component then $A$ is symmetric and $B$ is antisymmetric. Write $B=B'\cdot q$ for some symmetric polynomial $B'$. The problem now becomes finding $A,B'\in (S\mathfrak{h}^*)^{S_3}$ such that
$$A (-4\sigma_2^3-27\sigma_3^2)^{c-\frac{p+1}{2}}=B' v_p$$
Using Lemma \ref{lemma-aux-2} we get that there exists a symmetric polynomial $C$ such that  
$$A=Cv_p, B'=C(-4\sigma_2^3-27\sigma_3^2)^{c-\frac{p+1}{2}}, B=C(-4\sigma_2^3-27\sigma_3^2)^{c-\frac{p+1}{2}}q=Cq^{2c-p}$$
and so the vector 
$u=C v_{6c-3p}v_{p}$ lies in $\left< v_{6c-3p}v_p \right>$ as claimed.

{\bf Case 2.} If the vector $u$ is in the $\Triv$ isotypic component then $A$ is antisymmetric and $B$ is symmetric. Write $A=A'\cdot q$ for some symmetric polynomial $A'$, reducing the problem to 
$$A'\cdot (-4\sigma_2^3-27\sigma_3^2)^{c-\frac{p+1}{2}+1}=Bv_p$$
with $A',B\in (S\mathfrak{h}^*)^{S_3}$. Using Lemma \ref{lemma-aux-2} again we get that there exists a symmetric polynomial $C$ such that   
$$A'=Cv_p, A=Cv_pq, B=C(-4\sigma_2^3-27\sigma_3^2)^{c-\frac{p+1}{2}+1}=Cq^{2c-p+1},$$
and that
so the vector 
$u=C v_{6c-3p}v_{p}q$ lies in $\left< v_{6c-3p}v_p \right>$ as claimed.

{\bf Case 3.} If the vector $u$ is in the $\Stand$ isotypic component, we can assume without loss of generality that it is $S_2$ invariant. Then $A$ lies in the $S_2$ anti-invariant part of the $\Stand$ isotypic component of $S\h^*$, $B$ in the $S_2$ invariant, and using the basis from Theorem \ref{decmposeShp>3} we can write
$$A=A_1b_-+A_2\cdot 2b_+b_-, \quad B=B_1b_++B_2(-b_+^2+3b_-^2)$$
for some $A_1,A_2,B_1,B_2\in (S\mathfrak{h}^*)^{S_3}.$ The equation $A q^{2c-p}=B v_p$ becomes
$$(A_1b_-q+A_2\cdot 2b_+b_-  q)(-4\sigma_2^3-27\sigma_3^2)^{c-\frac{p+1}{2}}=(B_1b_++B_2(-b_+^2+3b_-^2))v_p,$$
%or
%$$\left(A_1 \cdot 3\sigma_3b_++A_1\cdot  \frac{1}{3}\sigma_2(-b_+^2+3b_-^2) +A_2\cdot 4\sigma_2^2 b_+-A_2\cdot 3\sigma_3 (-b_+^2+3b_-^2)\right)(-4\sigma_2^3-27\sigma_3^2)^{c-\frac{p+1}{2}}=$$
%$$=\left(B_1b_++B_2(-b_+^2+3b_-^2)\right)v_p.$$
%Rewriting further we get the system
which can be rewritten as the system
\begin{align*}\left(A_1 \cdot 3\sigma_3+A_2\cdot 4\sigma_2^2\right)(-4\sigma_2^3-27\sigma_3^2)^{c-\frac{p+1}{2}}&=B_1v_p\\
\left(A_1\cdot  \frac{1}{3}\sigma_2-A_2\cdot 3\sigma_3 \right)(-4\sigma_2^3-27\sigma_3^2)^{c-\frac{p+1}{2}}&=B_2v_p.
\end{align*}
This is again a problem about symmetric polynomials. We use Lemma \ref{lemma-aux-2} twice to get that there exist symmetric polynomials $C_1,C_2$ such that
$$A_1 \cdot 3\sigma_3+A_2\cdot 4\sigma_2^2=C_1v_p, \quad B_1=C_1 (-4\sigma_2^3-27\sigma_3^2)^{c-\frac{p+1}{2}}$$
$$A_1\cdot  \frac{1}{3}\sigma_2-A_2\cdot 3\sigma_3=C_2v_p, \quad B_2=C_2 (-4\sigma_2^3-27\sigma_3^2)^{c-\frac{p+1}{2}}.$$
From here we get
\begin{align*}
A_1\cdot(-4\sigma_2^3-27\sigma_3^2)&=(-3)(C_1\cdot 3\sigma_3+C_2\cdot 4\sigma_2^2)v_p\\
A_2\cdot(-4\sigma_2^3-27\sigma_3^2)&=(-3)(C_1\cdot \frac{1}{3}\sigma_2^3+C_2\cdot 3\sigma_3)v_p.
\end{align*}
We now use Lemma \ref{lemma-aux-2} twice more to get that there exist symmetric polynomials $D_1,D_2$ such that 
$$A_1=D_1v_p, A_2=D_2v_p.$$
Finally, it follows that the vector $u=A q^{2c-p}=Bv_p$ can be written as 
$$u=A v_{6c-3p}=(D_1b_-+D_2\cdot 2b_+b_-)v_{pc-3p}v_p,$$
which proves it in $\left<v_{pc-3p}v_p \right>$ as claimed. 

The computation of the character and the Hilbert series is now standard, using that $L_{1,c}(\Triv)$ is a quotient of $M_{1,c}(\Triv)$ by $\left<v_{6c-3p}\right>\cong M_{1,c}(\Sign)[-(6c-3p)]$ and $\left<v_{p}\right>\cong M_{1,c}(\Triv)[-p]$, and that  $\left<v_{6c-3p}\right>\cap \left<v_{p}\right>=\left<v_{6c-3p}v_p\right>\cong M_{1,c}(\Sign)[-(6c-2p)]$. 
\end{proof}

\subsection{Irreducible representation $L_{1,c}(\Stand)$, part 1}\label{section-p>3,Stand,specialc-1}

The aim of this Section is to prove Theorem \ref{p>3,specialc,Stand,Thm1}, which describes $L_{1,c}(\Stand)$ for $p>3$ and $c\in \mathbb{F}_p$, $0<c<p/3$. We have only been able to prove Lemma \ref{p>3-Stand-specialc-lemma-3c}, used in the proof of Theorem \ref{p>3,specialc,Stand,Thm1}, with the following technical assumption on the parameter $c$. We believe the assumption is satisfied for all $c\in \mathbb{F}_p$ with $0<c<p/3$. 

\begin{assumption}\label{assum}
	Let $p>3$ and $c\in \mathbb{F}_p$ with $0<c<p/3$. Assume that $c$ satisfies one of the following: 
	\begin{enumerate}
		\item Either $0<c<p/6$; 
		\item Or $p\equiv 1 \pmod 3$, $p/6<c<p/3$, and 
		$$\sum_{k=0}^{\lfloor \frac{1}{2} (\frac{p-1}{3}-c) \rfloor} \frac{1}{3^k}\frac{\prod_{j=1}^k(3j-1)(\frac{p-1}{3}-c)^{\underline{2k}}}{k!(c-2)^{\underline{2k}}}\ne 0;$$
		
		\item Or $p\equiv 2 \pmod 3$, $c=\frac{p+1}{6}$; 
		
		\item Or $p\equiv 2 \pmod 3$, $(p+1)/6<c<p/3$, and 
		$$\sum_{k=0}^{\lfloor \frac{1}{2} (\frac{p-2}{3}-c) \rfloor}\frac{1}{3^{k-1}}\frac{\prod_{j=1}^{k+1}(3j-2)(\frac{p-2}{3}-c)^{\underline{2k}}}{k!(c-1)^{\underline{2k+2}}}\ne 0.$$
	\end{enumerate}
\end{assumption}
We have checked this assumption in Magma \cite{Magma} for all primes $p<2023$.

\begin{theorem}\label{p>3,specialc,Stand,Thm1}
The irreducible representation $L_{1,c}(\Stand)$ of the rational Cherednik algebra $H_{1,c}(S_3,\h)$ over an algebraically closed field of characteristic $p>3$ for $c\in \mathbb{F}_p$, $0<c<p/3$, is the quotient of the Verma module $M_{1,c}(\Stand)$ by the submodule generated by the vectors: 
\begin{itemize}
\item $v_{p-3c}$ in degree $p-3c$ from Lemma \ref{p>3-Stand-specialc-lemma-p-3c};
\item $v_+,v_-$ in degree $p$ from Lemma \ref{p>3,t=1,generic,Stand};
\item $v_{p+3c}$ in degree $p+3c$ from Lemma \ref{p>3-Stand-specialc-lemma-p+3c}.
\end{itemize}
Its character and Hilbert polynomial are
\begin{align*}
\chi_{L_{1,c}(\Stand)}(z)&=\chi_{S\h^*}(z)\left( [\Stand]-[\Triv]z^{p-3c}-[\Stand]z^p-[\Sign]z^{3c+p}+[\Stand]z^{2p}\right)\\
h_{L_{1,c}(\Stand)}(z)&=\frac{2-z^{p-3c}-2z^p-z^{3c+p}+2z^{2p}}{(1-z)^2}.
\end{align*}
\end{theorem}
\begin{proof}

The vectors $v_{p-3c}, v_+,v_-$ and $v_{p+3c}$ are described and proved to be singular in Lemmas \ref{p>3-Stand-specialc-lemma-p-3c}, \ref{p>3,t=1,generic,Stand} and \ref{p>3-Stand-specialc-lemma-p+3c}. After that, the strategy is to calculate the Hilbert polynomials of $L_{1,c}(\Stand)$ and $M_{1,c}(\Stand)/\left<v_{p-3c}, v_+,v_-,v_{p+3c}\right>$ and show they are both equal to $\frac{2-z^{p-3c}-2z^p-z^{3c+p}+2z^{2p}}{(1-z)^2}$. The claim of the Theorem will then follow immediately. 

Lemma \ref{p>3-Stand-specialc-ordering} tells us that the module $M_{1,c}(\Stand)$ is an extension of the modules $L_{1,c}(\Stand)[-kp]$ with $k\ge 0$, $L_{1,c}(\Triv)[-(kp-3c)]$ with $k\ge 1$ and $L_{1,c}(\Sign)[-(kp+3c)]$ with $k\ge 0$. Consequently the Hilbert series of any quotient of $M_{1,c}(\Stand)$ will be a linear combination, with integer coefficients, of terms of the form $\frac{2z^{kp}}{(1-z)^2}$, $\frac{z^{kp-3c}}{(1-z)^2}$ and $\frac{z^{kp+3c}}{(1-z)^2}$. 

First, consider the degrees $0,1,\ldots, p-1$ of $L_{1,c}(\Stand)$. Lemma \ref{p>3-Stand-specialc-ordering} tells us that the only degrees among $0,1,\ldots, p-1$ where $M_{1,c}(\Stand)$ can have singular vectors and thus differ from $L_{1,c}(\Stand)$ are $p-3c$ and $3c$. Lemmas \ref{p>3-Stand-specialc-lemma-p-3c} and \ref{p>3-Stand-specialc-lemma-3c} (the latter using Assumption \ref{assum}) look for those singular vectors, finding one in degree $p-3c$ and none in degree $3c$, and let us conclude that 
$$h_{L_{1,c}(\Stand)}(z)=\frac{1}{(1-z)^2}(2-z^{p-3c}+O(z^p)).$$
(Here $O(z^p)$ denotes any sum of terms of degree greater or equal to $p$.)

This in turn lets us conclude that up to degree $p-1$, any quotient of $M_{1,c}(\Stand)$ looks either like $M_{1,c}(\Stand)$ and has the Hilbert series $\frac{1}{(1-z)^2}(2+O(z^p))$ or looks like $L_{1,c}(\Stand)$ and has the Hilbert series $\frac{2-z^{p-3c}+O(z^p)}{(1-z)^2}$

Next, we turn our attention to $M_{1,c}(\Stand)/\left<v_{p-3c}, v_+,v_-,v_{p+3c}\right>$. Lemma \ref{p>3-Stand-specialc-char1} shows that the quotient $M_{1,c}(\Stand)/\left<v_{p-3c}, v_{p+3c}\right>$ has the Hilbert series
$$h_{M_{1,c}(\Stand)/\left<v_{p-3c}, v_{p+3c}\right>}(z)=\frac{1}{(1-z)^2}(2-z^{p-3c}-z^{p+3c}).$$
Lemma \ref{p>3-Stand-specialc-char2} then shows that $v_+$ and $v_-$ are not zero in $M_{1,c}(\Stand)/\left<v_{p-3c}, v_{p+3c}\right>$. Thus, they generate a module $\left<v_+, v_-\right>$ in degree $p$, which is a quotient of $M_{1,c}(\Stand)$, and so, by the previous paragraph, up to degree $2p-1$ looks either like $M_{1,c}(\Stand)[-p]$ with the Hilbert series $\frac{2z^p+O(z^{2p})}{(1-z)^2}$ or like $L_{1,c}(\Stand)[-p]$ with the Hilbert series $\frac{2z^p-z^{2p-3c}+O(z^{2p})}{(1-z)^2}$. To see what the submodule $\left<v_+, v_-\right>$ of $M_{1,c}(\Stand)/\left<v_{p-3c,p+3c} \right>$ looks like in degrees up to $2p-1$, we notice it is the image of the map $$\varphi:M_{1,c}(\Stand)[-p] \to M_{1,c}(\Stand)/\left<v_{p-3c,p+3c} \right>$$ defined by $\varphi(b_{\pm})=v_{\pm}$. In Lemma \ref{p>3-Stand-specialc-char6} we show $\varphi(v_{p-3c})\ne 0$, which lets us conclude that $\left<v_+, v_-\right>$ has the Hilbert series $\frac{2z^p+O(z^{2p})}{(1-z)^2}$ and that $M_{1,c}(\Stand)/\left<v_{p-3c}, v_+,v_-,v_{p+3c}\right>$ has the Hilbert series 
\begin{equation}\label{p>3,cscpecial,Stand,somecharfmula}h_{M_{1,c}(\Stand)/\left<v_{p-3c},v_+,v_-, v_{p+3c}\right>}(z)= \frac{1}{(1-z)^2}\left( 2-z^{p-3c}- 2 \cdot z^p -  z^{p+3c}  +O(z^{2p}) \right).\end{equation}

Lemma \ref{p>3-Stand-specialc-char7} further shows that degree $2p-1$ of \eqref{p>3,cscpecial,Stand,somecharfmula} is $0$, which lets us conclude that $M_{1,c}(\Stand)/\left<v_{p-3c},v_+,v_-, v_{p+3c}\right>$ is concentrated in degrees up to $2p-2$, and has the Hilbert series
$$h_{M_{1,c}(\Stand)/\left<v_{p-3c},v_+,v_-, v_{p+3c}\right>}(z)= \frac{1}{(1-z)^2}\left( 2-z^{p-3c}- 2 \cdot z^p -  z^{p+3c} +2z^{2p} \right).$$
(The last statement follows because both are polynomials of degree $2p-2$, whose coefficients match up to degree $2p-1$). 

We now return to $L_{1,c}(\Stand)$. It is a quotient of $M_{1,c}(\Stand)/\left<v_{p-3c},v_+v_-, v_{p+3c}\right>$ by some submodule. This submodule is concentrated in degrees $p$ and above by Lemmas \ref{p>3-Stand-specialc-ordering}, \ref{p>3-Stand-specialc-lemma-p-3c} and \ref{p>3-Stand-specialc-lemma-3c}. Let $M$ be the number of irreducible composition factors of this submodule of type $L_{1,c}(\Stand)[-p]$, let $N$ be the number of irreducible composition factors of this submodule of type $L_{1,c}(\Triv)[-(2p-3c)]$ and let $K$ be the number of irreducible composition factors of this submodule of type $L_{1,c}(\Sign)[-(p+3c)]$. In the Grothendieck group, 
\begin{align*}[L_{1,c}(\Stand)]&=[M_{1,c}(\Stand)/\left<v_{p-3c},v_+,v_-, v_{p+3c}\right>]- M\cdot [L_{1,c}(\Stand)[-p]]-\\
& \quad- K\cdot [L_{1,c}(\Sign)[-(3c+p)]]-N\cdot [L_{1,c}(\Triv)[-(2p-3c)]].
\end{align*}
In terms of Hilbert polynomials%, using the previously calculated Hilbert series of $M_{1,c}(\Stand)/\left<v_{p-3c},v_+,v_-, v_{p+3c}\right>$, $L_{1,c}(\Triv)$ and $L_{1,c}(\Sign)$, we have 
\begin{align}\label{p>3-Stand-specialc-char-form}h_{L_{1,c}(\Stand)}(z)&=h_{M_{1,c}(\Stand)/\left<v_{p-3c},v_+,v_-, v_{p+3c}\right>}(z)-M\cdot z^p \cdot h_{L_{1,c}(\Stand)}(z)-\\
& \quad -K\cdot z^{3c+p}\cdot h_{L_{1,c}(\Sign)}(z)-N\cdot z^{2p-3c} h_{L_{1,c}(\Triv)}(z). \notag
\end{align}

Lemma \ref{p>3-Stand-specialc-char4} then lets us conclude that $M= K=N=0$, and that
$$L_{1,c}(\Stand)=M_{1,c}(\Stand)/\left<v_{p-3c},v_+,v_-, v_{p+3c}\right>,$$
with the Hilbert series is as claimed. The character computation is direct. 
\end{proof}

We will now state and prove all the lemmas used in the above proof. 

Recall that Lemma \ref{OmegaAction}
 limits the degrees in which singular vectors can occur, because if $\tau$ is an irreducible representation consisting of singular vectors in $M^k_{1,c}(\sigma)$ or one of its quotients, then $\Omega|_\tau=\Omega|_\sigma+k$.
Thus, to describe all degrees $k$ for which singular vectors of type $\tau$ can occur in $M^k_{1,c}(\Stand)$ or one of its quotients, we need to describe all $k\in \mathbb{N}$ of the form $k=\Omega|_\tau-\Omega|_\Stand.$ Their relative order is also important, as we look for singular vectors in order from smaller degrees to bigger, taking quotients each time and subsequently looking for singular vectors in the quotient. Putting that information together for $0<c<p/3$, we get the following lemma. 

\begin{lemma}\label{p>3-Stand-specialc-ordering}
If $\tau$ is an irreducible representation consisting of singular vectors inside $M^k_{1,c}(\Stand)$ or one of its quotients, then the pairs $(k,\tau)$  are of the form: 
\begin{enumerate}
\item For $c\in \left\{0,1,\ldots p-1\right\}$ with $0<c<p/6$: 
$$\begin{array}{ccccccc}
k: & \quad 3c < & p-3c  < &  p  < & 3c+p  < & 2p-3c < & \ldots \\
\tau:  & \quad \Sign & \Triv & \Stand & \Sign & \Triv & 
\end{array}$$
\item For $c\in \left\{0,1,\ldots p-1\right\}$ with $p/6<c<p/3$: 
$$\begin{array}{ccccccc}
k: & \quad p-3c  < & 3c <  &  p  < & 2p-3c < & 3c+p  < &  \ldots \\
\tau:  & \quad \Triv & \Sign &  \Stand & \Triv & \Sign &  
\end{array}$$
\end{enumerate}
\end{lemma}

We will first examine the $\Triv$ and $\Sign$ isotypic components in degrees $p-3c$ and $3c$.

\begin{lemma}\label{p>3-Stand-specialc-lemma-p-3c}
There is a one dimensional space of singular vectors in $M_{1,c}^{p-3c}(\Stand)$, in the $\Triv$ isotypic component and spanned by $v_{p-3c}$ given by:
\begin{enumerate}
\item If $p\equiv 1 \bmod 3$, 
\begin{align*}
v_{p-3c}&=\sum_{i=0}^{\lfloor \frac{1}{2}(\frac{p-1}{3}-c)\rfloor}\frac{(-1)^i}{9^i}\frac{(\frac{p-1}{3}-c)^{\underline{2i}}}{ i! \cdot  \prod_{j=1}^{i}(3j-2) } \sigma_2^{3i}\sigma_3^{\frac{p-1}{3}-c-2i} \cdot (b_+\otimes b_++ 3 b_-\otimes b_-)+ \\
&  +\sum_{i=0}^{\lfloor\frac{1}{2}(\frac{p-1}{3}-c-1)\rfloor}\frac{(-1)^{i+1}}{6\cdot 9^i} \frac{(\frac{p-1}{3}-c)^{\underline{2i+1}}}{ i! \cdot \prod_{j=1}^{i+1}(3j-2) } \sigma_2^{3i+1}\sigma_3^{\frac{p-1}{3}-c-2i-1} \cdot\\
& \quad\quad\quad \cdot \left( (-b_+^2+3b_-^2)\otimes b_++3\cdot (2b_+b_-)\otimes b_-\right);\end{align*}
\item If $p\equiv 2 \bmod 3$, 
\begin{align*}
v_{p-3c}&=\sum_{i=0}^{\lfloor \frac{1}{2}(\frac{p-2}{3}-c-1)\rfloor}\frac{(-1)^i}{9^i}\frac{2 (\frac{p-2}{3}-c)^{\underline{2i+1}}}{ i! \prod_{j=1}^{i+1}(3j-1)} \sigma_2^{3i+2}\sigma_3^{\frac{p-2}{3}-c-2i-1} \cdot (b_+\otimes b_++ 3 b_-\otimes b_-)+\\
& \quad +\sum_{i=0}^{\lfloor\frac{1}{2}(\frac{p-2}{3}-c)\rfloor}\frac{(-1)^{i}}{9^i} \frac{(\frac{p-2}{3}-c)^{\underline{2i}}}{ i! \prod_{j=1}^{i}(3j-1) } \sigma_2^{3i}\sigma_3^{\frac{p-2}{3}-c-2i} \cdot \\
& \quad \quad \quad \cdot \left( (-b_+^2+3b_-^2)\otimes b_++3\cdot (2b_+b_-)\otimes b_-\right).\end{align*}
\end{enumerate}

\end{lemma}
\begin{proof}
By Theorem \ref{decmposeVermap>3}, a basis for the $\Triv$ isotypic component of $M^{p-3c}_{1,c}(\Stand)$ is 
$$\left\{\sigma_2^a\sigma_3^b \cdot (b_+\otimes b_++ 3 b_-\otimes b_-)\mid  \quad 2a+3b=p-3c-1\right\} \,\,\cup$$
$$\cup \,\, \left\{\sigma_2^a\sigma_3^b \cdot \left( (-b_+^2+3b_-^2)\otimes b_++3\cdot (2b_+b_-)\otimes b_-\right)\mid \quad 2a+3b=p-3c-2\right\}.$$

Using Section \ref{sect-auxiliary}, we calculate the values of $D_{y_1}$ on this basis to be: 
\begin{align}\label{D1-Stand-special-Triv-1}
 D_{y_1}&(\sigma_2^a\sigma_3^b \cdot (b_+\otimes b_++ 3 b_-\otimes b_-))= \\ 
& =\frac{3c+1}{2}\sigma_2^a\sigma_3^b \otimes (b_++3b_-)+\notag \\
&\quad +\frac{b}{6}\sigma_2^{a+1}\sigma_3^{b-1}\cdot \left( b_+\otimes b_+-3b_-\otimes b_+-3b_+\otimes b_--3b_-\otimes b_-\right)+\notag \\
&\quad +\frac{a}{12}\sigma_2^{a-1}\sigma_3^{b}\cdot \left( (-b_+^2+3b_-^2)\otimes b_+-3(2b_+b_-)\otimes b_+-\right.\notag\\
& \quad \quad\quad\quad\quad\quad\quad\quad 
\left. -3(-b_+^2+3b_-^2)\otimes b_--3(2b_+b_-)\otimes b_-\right)+\notag \\
&\quad  +\frac{b}{2}\sigma_2^{a}\sigma_3^{b-1}q\otimes (b_+-b_-);\notag \\
D_{y_1}&(\sigma_2^a\sigma_3^b \cdot \left( (-b_+^2+3b_-^2)\otimes b_++3\cdot (2b_+b_-)\otimes b_-\right))=\label{D1-Stand-special-Triv-2} \\ 
& =-9a\sigma_2^{a-1}\sigma_3^{b+1} \otimes (b_++3b_-)+2b\sigma_2^{a+2}\sigma_3^{b-1} \otimes (b_++3b_-)+ \notag \\
&\quad +a\sigma_2^{a}\sigma_3^{b}\cdot \left( b_+\otimes b_+-3b_-\otimes b_+-3b_+\otimes b_--3b_-\otimes b_-\right)+\notag \\
&\quad +\frac{-b}{6}\sigma_2^{a+1}\sigma_3^{b-1}\cdot \left( (-b_+^2+3b_-^2)\otimes b_+-3(2b_+b_-)\otimes b_+- \right. \notag \\
& \quad\quad\quad\quad\quad\quad\quad\quad \left. -3(-b_+^2+3b_-^2)\otimes b_--3(2b_+b_-)\otimes b_-\right)+\notag \\
&\quad  +3a\sigma_2^{a-1}\sigma_3^{b}q\otimes (b_+-b_-).\notag 
\end{align}

We now consider the cases $p\equiv 1,2 \bmod 3$ separately to parametrize the integral solutions to $2a+3b=p-3c-1$ and $2a+3b=p-3c-2$. 

\begin{enumerate}
\item Assume $p\equiv 1 \bmod 3$. A vector in the $\Triv$ isotypic component of $M^{p-3c}_{1,c}(\Stand)$ is of the form
\begin{align*}v_{p-3c}&=\sum_{i=0}^{\lfloor \frac{1}{2}(\frac{p-1}{3}-c)\rfloor}\lambda_i \sigma_2^{3i}\sigma_3^{\frac{p-1}{3}-c-2i} \cdot (b_+\otimes b_++ 3 b_-\otimes b_-)+\\
& \quad +\sum_{i=0}^{\lfloor\frac{1}{2}(\frac{p-1}{3}-c-1)\rfloor}\mu_i \sigma_2^{3i+1}\sigma_3^{\frac{p-1}{3}-c-2i-1} \cdot \left( (-b_+^2+3b_-^2)\otimes b_++3\cdot (2b_+b_-)\otimes b_-\right).
\end{align*}

Using equations \eqref{D1-Stand-special-Triv-1} and \eqref{D1-Stand-special-Triv-2}, the condition $D_{y_1}(v_{p-3c})=0$ leads to the following equations for the coefficients $\lambda_i$, $\mu_i$, which need to hold for all $i$: 
\begin{align}
\label{eq-Stand-special-Triv-1}
\lambda_i \cdot \frac{1}{2}(\frac{p-1}{3}-c-2i)+\mu_i\cdot 3 (3i+1)&=0\\
\label{eq-Stand-special-Triv-2}
\lambda_{i+1}\cdot \frac{3(i+1)}{12}-\mu_i\cdot \frac{\frac{p-1}{3}-c-2i-1}{6}&=0\\
\label{eq-Stand-special-Triv-3}
\lambda_i \cdot \frac{3c+1}{2}+\mu_i \cdot (-9)(3i+1)+\mu_{i-1} \cdot 2 (\frac{p-1}{3}-c-2i+1)&=0.
\end{align}

The system has a unique (up to overall scaling) solution
\begin{align}\label{eq-Stand-special-p-3c-coeffs-case1}
\lambda_i&=
\frac{(-1)^i}{9^i}\frac{(\frac{p-1}{3}-c)^{\underline{2i}}}{ i! \cdot  \prod_{j=1}^{i}(3j-2) }\\
\mu_i&=\frac{(-1)^{i+1}}{6\cdot 9^i} \frac{(\frac{p-1}{3}-c)^{\underline{2i+1}}}{ i! \cdot \prod_{j=1}^{i+1}(3j-2) }
.\notag
\end{align}

\item Assume $p\equiv 2 \bmod 3$. A vector in the $\Triv$ isotypic component of $M^{p-3c}_{1,c}(\Stand)$ is of the form
\begin{align*}v_{p-3c}&=\sum_{i=0}^{\lfloor \frac{1}{2}(\frac{p-2}{3}-c-1)\rfloor}\lambda_i \sigma_2^{3i+2}\sigma_3^{\frac{p-2}{3}-c-2i-1} \cdot (b_+\otimes b_++ 3 b_-\otimes b_-)+\\
& \quad +\sum_{i=0}^{\lfloor\frac{1}{2}(\frac{p-2}{3}-c)\rfloor}\mu_i \sigma_2^{3i}\sigma_3^{\frac{p-2}{3}-c-2i} \cdot \left( (-b_+^2+3b_-^2)\otimes b_++3\cdot (2b_+b_-)\otimes b_-\right).
\end{align*}

Using equations \eqref{D1-Stand-special-Triv-1} and \eqref{D1-Stand-special-Triv-2}, the condition $D_{y_1}(v_{p-3c})=0$ leads to linear equations 
%\begin{align}
%\label{eq-Stand-special-Triv-4}
%\lambda_{i-1} \cdot \frac{1}{2}(\frac{p-2}{3}-c-2i+1)+\mu_{i}\cdot 3 (3i)&=0\\
%\label{eq-Stand-special-Triv-5}
%\lambda_{i} \cdot \frac{3i+2}{2}-\mu_i \cdot (\frac{p-1}{3}-c-2i)&=0\\
%\label{eq-Stand-special-Triv-6}
%\lambda_i \cdot \frac{3c+1}{2}-\mu_{i+1} \cdot 27(i+1)+\mu_{i} \cdot 2 (\frac{p-2}{3}-c-2i)&=0,
%\end{align}
with a unique (up to overall scaling) solution
\begin{align}\label{eq-Stand-special-p-3c-coeffs-case2}
\lambda_i&=\frac{(-1)^i}{9^i}\frac{2 (\frac{p-1}{3}-c)^{\underline{2i+1}}}{ i! \prod_{j=1}^{i+1}(3j-1)}\\
\mu_i&=\frac{(-1)^{i}}{9^i} \frac{(\frac{p-1}{3}-c)^{\underline{2i}}}{ i! \prod_{j=1}^{i}(3j-1) }.\notag
\end{align}
\end{enumerate}
\end{proof}

\begin{lemma}\label{p>3-Stand-specialc-lemma-3c} Assume $c$ satisfies the conditions in Assumption \ref{assum}. Then
there are no singular vectors in $M_{1,c}(\Stand)^{3c}/\left<v_{p-3c}\right>$. 
\end{lemma}
\begin{proof} We distingush the cases $0<c<p/6$ and $p/6<c<p/3$. 

{\bf Case 1: $0<c<p/6$.} By Lemma \ref{p>3-Stand-specialc-ordering} $3c<p-3c$, so $M_{1,c}(\Stand)^{3c}/\left<v_{p-3c}\right>=M_{1,c}(\Stand)^{3c}$, and it is enough to show there are no singular vectors in the $\Sign$ component of $M_{1,c}^{3c}(\Stand)$. This is the easier case as we are not working modulo $v_{p-3c}$. By Theorem \ref{decmposeVermap>3}, any vector in the $\Sign$ component of $M_{1,c}^{3c}(\Stand)$ can be uniquely expressed as
\begin{align*}v&=\sum_{i=0}^{\lfloor \frac{1}{2}(c-1)\rfloor}\nu_i \sigma_2^{3i+1}\sigma_3^{c-1-2i} \cdot (b_+\otimes b_--b_-\otimes b_+)+\\
& \quad +\sum_{i=0}^{\lfloor\frac{1}{2}(c-2)\rfloor}\xi_i \sigma_2^{3i+2}\sigma_3^{c-2-2i} \cdot \left( (-b_+^2+3b_-^2)\otimes b_-- (2b_+b_+)\otimes b_+\right).
\end{align*}

We calculate that 
\begin{align}
D_{y_1}&(\sigma_2^{a}\sigma_3^{b} \cdot (b_+\otimes b_--b_-\otimes b_+))=\label{p>3,genc,Stand,D1(Sign)-p1}\\
&=\frac{1-3c}{2}\cdot \sigma_2^{a}\sigma_3^{b} \otimes \left(b_--b_+\right)
\notag\\
& \quad + \frac{b}{6} \cdot \sigma_2^{a+1}\sigma_3^{b-1} \cdot \left(b_+\otimes b_++b_-\otimes b_++b_+\otimes b_--3b_-\otimes b_-)\right) \notag \\
& \quad + \frac{a}{12} \cdot \sigma_2^{a-1}\sigma_3^{b} \cdot \left((-b_+^2+3b_-^2)\otimes b_++(2b_+b_-)\otimes b_++\right.\notag \\
& \quad \quad \quad \quad \quad \quad \quad \quad \left.+(-b_+^2+3b_-^2)\otimes b_--3(2b_+b_-)\otimes b_-)\right) \notag \\
& \quad + \frac{b}{6}\cdot \sigma_2^{a}\sigma_3^{b-1} q \otimes \left(3b_--b_+\right),\notag
%& =\frac{1-3c}{2}\cdot \sigma_2^{3i+1}\sigma_3^{c-1-2i} \otimes (b_--b_+) \notag\\
%& \quad + \frac{c-1-2i}{6} \cdot \sigma_2^{3i+2}\sigma_3^{c-2-2i} \cdot (b_+\otimes b_++b_-\otimes b_++b_+\otimes b_--3b_-\otimes b_-)) \notag \\
%& \quad + \frac{3i+1}{12} \cdot \sigma_2^{3i}\sigma_3^{c-1-2i} \cdot ((-b_+^2+3b_-^2)\otimes b_++(2b_+b_-)\otimes b_++(-b_+^2+3b_-^2)\otimes b_--3(2b_+b_-)\otimes b_-)) \notag \\
%& \quad + \frac{c-1-2i}{6}\cdot \sigma_2^{3i+1}\sigma_3^{c-2-2i} q \otimes (3b_--b_+),\notag
\end{align}
\begin{align}
D_{y_1}&(\sigma_2^{a}\sigma_3^{b} \cdot \left( (-b_+^2+3b_-^2)\otimes b_-- (2b_+b_+)\otimes b_-\right))=\label{p>3,genc,Stand,D1(Sign)-p2}\\
& =-9a \cdot \sigma_2^{a-1}\sigma_3^{b+1} \otimes \left(b_--b_+\right)+2b \sigma_2^{a+2}\sigma_3^{b-1} \otimes \left(b_--b_+\right) \notag\\
& \quad + a \cdot \sigma_2^{a}\sigma_3^{b} \cdot \left(b_+\otimes b_++b_-\otimes b_++b_+\otimes b_--3b_-\otimes b_-)\right) \notag \\
& \quad - \frac{b}{6} \cdot \sigma_2^{a+1}\sigma_3^{b-1} \cdot \left((-b_+^2+3b_-^2)\otimes b_++(2b_+b_-)\otimes b_++\right.\notag \\
& \quad \quad \quad \quad \quad \quad \quad \quad \left. +(-b_+^2+3b_-^2)\otimes b_--3(2b_+b_-)\otimes b_-)\right) \notag \\
& \quad + a\cdot \sigma_2^{a-1}\sigma_3^{b} q \otimes \left(3b_--b_+\right).\notag
%& =-9(3i+2) \cdot \sigma_2^{3i+1}\sigma_3^{c-1-2i} \otimes (b_--b_+)+2(c-2-2i) \sigma_2^{3i+4}\sigma_3^{c-3-2i} \otimes (b_--b_+) \notag\\
%& \quad + (3i+2) \cdot \sigma_2^{3i+2}\sigma_3^{c-2-2i} \cdot (b_+\otimes b_++b_-\otimes b_++b_+\otimes b_--3b_-\otimes b_-)) \notag \\
%& \quad - \frac{c-2-2i}{6} \cdot \sigma_2^{3i+3}\sigma_3^{c-3-2i} \cdot ((-b_+^2+3b_-^2)\otimes b_++(2b_+b_-)\otimes b_++(-b_+^2+3b_-^2)\otimes b_--3(2b_+b_-)\otimes b_-)) \notag \\
%& \quad + (3i+2)\cdot \sigma_2^{3i+1}\sigma_3^{c-2-2i} q \otimes (3b_--b_+).\notag
\end{align}

From here, $D_{y_1}(v)=0$ gives the system 
\begin{align*}
\frac{1-3c}{2}\cdot \nu_i -9(3i+2) \cdot \xi_i+2(c-2i)\cdot \xi_{i-1}&=0\\
\frac{c-1-2i}{6} \cdot \nu_i+(3i+2) \cdot \xi_i&=0\\
\frac{3i+1}{12} \cdot \nu_i - \frac{c-2i}{6} \cdot \xi_{i-1}&=0,
\end{align*}
the only solution to which is $\nu_i=0,\xi_i=0$ for all $i$, so $v=0$.

{\bf Case 2: $p/6<c<p/3$.} This task is harder because $3c>p-3c$ and we have to work modulo $v_{p-3c}$. We consider further two cases. 

{\bf Case 2.1. $p\equiv 1 \bmod 3$.} A vector in the $\Sign$ component of $M^{3c}(\Stand)/\left<v_{p-3c} \right>$ can be uniquely written as 
\begin{align*}
v&=\sum_{i=0}^{\lfloor \frac{c-1}{2}\rfloor} \nu_i \sigma_2^{3i+1}\sigma_3^{c-1-2i}\cdot (b_+\otimes b_--b_-\otimes b_+)+\\
&\quad +\sum_{i=0}^{\frac{p-1}{6}-\lfloor \frac{c-1}{2} \rfloor-2}\xi_i \sigma_2^{3i+2}\sigma_3^{c-2-2i}\cdot \left( (-b_+^2+3b_-^2)\otimes b_-- (2b_+b_+)\otimes b_-\right)).
\end{align*}
(Note the boundaries of summation, due to the fact that in $M(\Stand)^{3c}/\left<v_{p-3c} \right>$ the multiplicity of $\Sign$ is $\frac{p-1}{6}$, and the above terms form a basis.)

Using Lemma \ref{KerD1}, $v$ is singular in $M^{3c}(\Stand)/\left<v_{p-3c} \right>$ if and only if $D_{y_1}(v)\in \left<v_{p-3c} \right>,$ meaning that there exist  $\pi_j,\rho_j\in \Bbbk$  such that 
\begin{align}
D_{y_1}(v)&=\sum_{j=0}^{c-\frac{p-1}{6}-1}\pi_j\sigma_2^{3j}\sigma_3^{2c-\frac{p-1}{3}-1-2j}v_{p-3c}(-b_++b_-)\label{v-sing-case2.1} \\
&\quad + \sum_{j=0}^{c-\frac{p-1}{6}-1}\rho_j \sigma_2^{3j+1}\sigma_3^{2c-\frac{p-1}{3}-2-2j}v_{p-3c}(-(-b_+^2+3b_-^2)+(2b_+b_-)).\notag
\end{align}

Expanding this out using the explicit form for $v_{p-3c}$ from Lemma \ref{p>3-Stand-specialc-lemma-p-3c} and equations \eqref{p>3,genc,Stand,D1(Sign)-p1} and \eqref{p>3,genc,Stand,D1(Sign)-p2} we get that \eqref{v-sing-case2.1} is equivalent to the following being satisfied for all $k$: 
\begin{align}\label{3c-technical1}
\nu_k \cdot \frac{1-3c}{2}+\xi_k\cdot (-9)(3k+2)+&\xi_{k-1}\cdot 2(c-2k)=\notag \\
\quad \quad \quad  \quad =\sum_{j=0}^k -6\pi_j\lambda_{k-j}+&54 \pi_j\mu_{k-j}+54\rho_j\lambda_{k-j}+72\rho_j\mu_{k-j}\\
\label{3c-technical2}
\nu_k \cdot \frac{c-1-2k}{6}+\xi_k\cdot (3k+2)&=\sum_{j=0}^k -6\pi_j\mu_{k-j}-6\rho_j\lambda_{k-j}+108\rho_j\mu_{k-j}\\
\label{3c-technical3}
\nu_k \cdot \frac{3k+1}{12}+\xi_{k-1}\cdot \frac{-c+2k}{6}&=\sum_{j=0}^k \frac{1}{2}\pi_j\lambda_{k-j}+6\rho_j\mu_{k-j-1}\\
\label{3c-technical4}
\nu_k \cdot \frac{c-1-2k}{6}+\xi_k\cdot (3k+2)&=\sum_{j=0}^k -6\pi_j\mu_{k-j}+6\rho_j\lambda_{k-j},
\end{align}
where
%\begin{align}\label{3c-technical5}
%\lambda_i=\frac{(-1)^i}{9^i}\frac{(\frac{p-1}{3}-c)^{\underline{2i}}}{ i! \cdot  \prod_{j=1}^{i}(3j-2) }, & \quad  
%\mu_i=\frac{(-1)^{i+1}}{6\cdot 9^i} \frac{(\frac{p-1}{3}-c)^{\underline{2i+1}}}{ i! \cdot \prod_{j=1}^{i+1}(3j-2) }
%\end{align}
$\lambda_i$ and $\mu_i$ are as in Lemma \ref{p>3-Stand-specialc-lemma-p-3c}. 

Subtracting equations \eqref{3c-technical2} and \eqref{3c-technical4} we get that for all $k$
$$\sum_{j=0}^k\rho_j(-12\lambda_{k-j}+108\mu_{k-j})=0,$$
which (using the explicit form of $\lambda_i$ and $\mu_i$ from Lemma \ref{p>3-Stand-specialc-lemma-p-3c} to check that $-12\lambda_{k-j}+108\mu_{k-j}\ne 0$) implies that $\rho_j=0$ for all $j$. 

The system \eqref{3c-technical1}-\eqref{3c-technical4} is now equivalent to requiring that for all $k$
\begin{align}
\nu_k \cdot \frac{c-1-2k}{6}+\xi_k\cdot (3k+2)&=-6\sum_{j=0}^k \pi_j\mu_{k-j} \label{3c-technical5}\\
\nu_k \cdot \frac{3k+1}{12}+\xi_{k-1}\cdot \frac{-c+2k}{6}&=\frac{1}{2} \sum_{j=0}^k \pi_j\lambda_{k-j}.\label{3c-technical6}
\end{align}

We split into two further cases. 

{\bf Case 2.1.1. Assume $c$ is even.} The relevance of this assumption is in the range of $k$ for which the unknowns $\nu_k,\xi_k,\pi_k$ and the parameters $\lambda_k,\mu_k$ are potentially nonzero. Namely, we are looking for all solutions to \eqref{3c-technical5}-\eqref{3c-technical6} where: 
\begin{align*}
\nu_k\ldots 0\le k\le \frac{c}{2}-1\\
\xi_k\ldots 0\le k\le \frac{p-1}{6}-\frac{c}{2}-1\\
\pi_k\ldots 0\le k\le c-\frac{p-1}{6}-1\\
\lambda_k\ldots 0\le k\le \frac{p-1}{6}-\frac{c}{2}-1\\
\mu_k\ldots 0\le k\le \frac{p-1}{6}-\frac{c}{2}-1.
\end{align*}

Let us prove by downwards induction on $m$ that $$\nu_{\frac{p-1}{6}-\frac{c}{2}+m}=0, \quad \pi_m=0 \quad \textrm{ for all } m>0.$$
It is certainly true for large $m$ (by convention). Assume it is true for $m+1,m+2, \ldots$, and let us prove it for $m$. Equation \eqref{3c-technical5} for $k=\frac{p-1}{6}-\frac{c}{2}+m$ reads $\nu_{\frac{p-1}{6}-\frac{c}{2}+m}=0$, and equation \eqref{3c-technical6} for $k=\frac{p-1}{6}-\frac{c}{2}+m$ implies $\pi_{m}=0$. 

The system \eqref{3c-technical5}-\eqref{3c-technical6} is thus reduced to 
\begin{align}
\nu_k \cdot \frac{c-1-2k}{6}+\xi_k\cdot (3k+2)&=-6 \pi_0 \mu_{k} \label{3c-technical7}\\
\nu_k \cdot \frac{3k+1}{12}+\xi_{k-1}\cdot \frac{-c+2k}{6}&=\frac{1}{2} \pi_0\lambda_{k}.\label{3c-technical8}
\end{align}
We multiply \eqref{3c-technical7} by $a_k$ and \eqref{3c-technical8} by $b_k$ for 
$$a_k=\frac{(-1)^{k+1}}{2}\frac{(3k+1)!}{k!(c-1)^{\underline{2k+1}}}, \quad,  b_k=(-1)^{k}\frac{(3k)!}{k!(c-1)^{\underline{2k}}},$$ and sum over all $k$ to get, for $n=\frac{p-1}{6}-\frac{c}{2}$,
\begin{align*}
0&=\pi_0\cdot \left( \sum_{k=0}^{n-1} (-6)\mu_ka_k+\sum_{k=0}^{n} \frac{1}{2}\lambda_k b_k\right)\\
&=\pi_0\cdot \left( \sum_{k=0}^{n-1} \frac{-1}{2\cdot 3^k}\frac{\prod_{j=1}^k(3j-1)(2n)^{\underline{2k+1}}}{k!(c-1)^{\underline{2k+1}}}+\sum_{k=0}^{n} \frac{1}{2\cdot 3^k}\frac{\prod_{j=1}^k(3j-1)(2n)^{\underline{2k}}}{k!(c-1)^{\underline{2k}}} \right)\\
&=\pi_0\cdot \frac{c-\frac{p-1}{6}-\frac{1}{2}}{c-1}\sum_{k=0}^{n} \frac{1}{3^k}\frac{\prod_{j=1}^k(3j-1)(2n)^{\underline{2k}}}{k!(c-2)^{\underline{2k}}}.
\end{align*}
As $\frac{c-\frac{p-1}{6}-\frac{1}{2}}{c-1}\ne 0$ for $c\in \mathbb{F}_p$, $p/6<c<p/3$, and as $\sum_{k=0}^{n} \frac{1}{3^k}\frac{\prod_{j=1}^k(3j-1)(2n)^{\underline{2k}}}{k!(c-2)^{\underline{2k}}}\ne 0$ by Assumption \ref{assum}, we conclude that $\pi_0=0$. From here we proceed like in part (1) to deduce that $\nu_k=0$, $\xi_k=0$ for all $k$, and $v=0$.

{\bf Case 2.1.2. Assume $c$ is odd.} In this case we are looking for all solutions to \eqref{3c-technical5}-\eqref{3c-technical6} with the boundaries: 
\begin{align*}
\nu_k\ldots 0\le k\le \frac{c-1}{2}\\
\xi_k\ldots 0\le k\le \frac{p-1}{6}-\frac{c-1}{2}-2\\
\pi_k\ldots 0\le k\le c-\frac{p-1}{6}-1\\
\lambda_k\ldots 0\le k\le \frac{p-1}{6}-\frac{c-1}{2}-1\\
\mu_k\ldots 0\le k\le \frac{p-1}{6}-\frac{c-1}{2}-1.
\end{align*}

As before, we prove by downwards induction on $m$ that $$\nu_{\frac{p-1}{6}-\frac{c-1}{2}-1+m}=0 \quad \pi_m=0 \quad \textrm{ for all } m>0.$$ Assuming this holds for $m+1,m+2,\ldots,$ we write equations \eqref{3c-technical5}-\eqref{3c-technical6} for $k=\frac{p-1}{6}-\frac{c-1}{2}-1+m$ to get a system of two equations with two unknowns $\nu_{\frac{p-1}{6}-\frac{c-1}{2}-1+m}$ and $\pi_m=0$, the only solution of which is $\nu_{\frac{p-1}{6}-\frac{c-1}{2}-1+m}=0$, $\pi_m=0$. 

The remaining system is identical to the system \eqref{3c-technical7}-\eqref{3c-technical8}, and we proceed as in case (i).

{\bf Case 2.2. $p\equiv 2 \bmod 3$.} Similar. A vector in the $\Sign$ component of $M_{1,c}^{3c}(\Stand)/\left<v_{p-3c} \right>$ can be uniquely written as 
\begin{align*}
v&=\sum_{i=0}^{\lfloor \frac{c-1}{2}\rfloor} \nu_i \sigma_2^{3i+1}\sigma_3^{c-1-2i}\cdot (b_+\otimes b_--b_-\otimes b_+)+\\
&\quad +\sum_{i=0}^{\frac{p+1}{6}-\lfloor \frac{c-1}{2} \rfloor-2}\xi_i \sigma_2^{3i+2}\sigma_3^{c-2-2i}\cdot \left( (-b_+^2+3b_-^2)\otimes b_-- (2b_+b_+)\otimes b_-\right)).
\end{align*}
Such a vector is singular in $M(\Stand)/\left<v_{p-3c} \right>$ if and only if there exist $\pi_j,\rho_j\in \Bbbk$ such that 
\begin{align*}
D_{y_1}(v)&=\sum_{j=0}^{c-\frac{p+1}{6}-1}\pi_j\sigma_2^{3j+1}\sigma_3^{2c-\frac{p+1}{3}-2j-1}v_{p-3c}(-b_++b_-)\\
&\quad + \sum_{j=0}^{c-\frac{p+1}{6}-1}\rho_j \sigma_2^{3j+2}\sigma_3^{2c-\frac{p+1}{3}-2-2j}v_{p-3c}(-(-b_+^2+3b_-^2)+(2b_+b_-)).
\end{align*}

%This is equivalent to the following system being satisfied for all $k$: 
%\begin{align}\label{3c-technical9}
%\nu_k \cdot \frac{1-3c}{2}+\xi_k\cdot (-9)(3k+2)+\xi_{k-1}\cdot 2(c-2k)&=\sum_{j=0} -6\pi_j\lambda_{k-j-1}+54 \pi_j\mu_{k-j}+54\rho_j\lambda_{k-j-1}+72\rho_j\mu_{k-j-1}
%\end{align}
%\begin{align}\label{3c-technical10}
%\nu_k \cdot \frac{c-1-2k}{6}+\xi_k\cdot (3k+2)&=\sum_{j=0} -6\pi_j\mu_{k-j}-6\rho_j\lambda_{k-j-1}+108\rho_j\mu_{k-j}
%\end{align}
%\begin{align}\label{3c-technical11}
%\nu_k \cdot \frac{3k+1}{12}+\xi_{k-1}\cdot \frac{-c+2k}{6}&=\sum_{j=0} \frac{1}{2}\pi_j\lambda_{k-j-1}+6\rho_j\mu_{k-j-1}
%\end{align}
%\begin{align}\label{3c-technical12}
%\nu_k \cdot \frac{c-1-2k}{6}+\xi_k\cdot (3k+2)&=\sum_{j=0} -6\pi_j\mu_{k-j}+6\rho_j\lambda_{k-j-1},
%\end{align}
%where $\lambda_i$, $\mu_i$ are as in Lemma \ref{p>3-Stand-specialc-lemma-p-3c}. 

Reading off the coefficients we get a system of equations which lets us immediately conclude that $\rho_j=0$ for all $j$. After this, the system is:
\begin{align}
\nu_k \cdot \frac{c-1-2k}{6}+\xi_k\cdot (3k+2)&=-6\sum_{j=0}^k \pi_j\mu_{k-j} \label{3c-technical13}\\
\nu_k \cdot \frac{3k+1}{12}+\xi_{k-1}\cdot \frac{-c+2k}{6}&=\frac{1}{2} \sum_{j=0}^{k-1} \pi_j\lambda_{k-j-1}.\label{3c-technical14}
\end{align}

As before, distinguishing the cases of even and odd $c$, we show by downwards induction on $m$ that $\pi_m=0$ and $\nu_{\frac{p+1}{6}-\frac{c}{2}-1+m}=0$ for all $m>0$.

This leaves us with the system
\begin{align}
\nu_k \cdot \frac{c-1-2k}{6}+\xi_k\cdot (3k+2)&=-6 \pi_0\mu_{k} \label{3c-technical15}\\
\nu_k \cdot \frac{3k+1}{12}+\xi_{k-1}\cdot \frac{-c+2k}{6}&=\frac{1}{2} \pi_0\lambda_{k-1},\label{3c-technical14}
\end{align}
which can be shown by telescoping to have no nonzero solutions precisely when Assumption \ref{assum} is satisfied. We thus conclude that $v=0$.  
\end{proof}

Next, as dictated by the proof of Theorem \ref{p>3,specialc,Stand,Thm1}, we search $M^{p+3c}_{1,c}(\Stand)$ for singular vectors of type $\Sign$. 

\begin{lemma}\label{p>3-Stand-specialc-lemma-p+3c}
	The vector $v_{p+3c}\in M^{p+3c}_{1,c}(\Stand)$ given below is singular.
	\begin{enumerate}
		\item If ${{p}\,\equiv\,{1}\:(\mathrm{mod}\:3)}$, 
		\begin{align*}
			&v_{p+3c}=\sum_{i=0}^{\lfloor \frac{1}{2}(\frac{p-1}{3}+c)\rfloor}\frac{(-1)^i}{9^i}\frac{(\frac{p-1}{3}+c)^{\underline{2i}}}{ i! \cdot  \prod_{j=1}^{i}(3j-2) } \sigma_2^{3i}\sigma_3^{\frac{p-1}{3}+c-2i} \cdot (b_+\otimes b_-- b_-\otimes b_+)+ \\
			& \quad\quad  +\sum_{i=0}^{\lfloor\frac{1}{2}(\frac{p-1}{3}+c-1)\rfloor}\frac{(-1)^{i+1}}{6\cdot 9^i} \frac{(\frac{p-1}{3}+c)^{\underline{2i+1}}}{ i! \cdot \prod_{j=1}^{i+1}(3j-2) } \sigma_2^{3i+1}\sigma_3^{\frac{p-1}{3}+c-2i-1} \cdot &\\
			&\quad\quad\quad\quad\quad\quad\quad\quad\cdot \bigg( (-b_+^2+3b_-^2)\otimes b_++3\cdot (2b_+b_-)\otimes b_-\bigg);\end{align*}
		\item If ${{p}\,\equiv\,{2}\:(\mathrm{mod}\:3)}$, 
		\begin{align*}
			&v_{p+3c}=\sum_{i=0}^{\lfloor \frac{1}{2}(\frac{p-2}{3}+c-1)\rfloor}\frac{(-1)^i}{9^i}\frac{2 (\frac{p-2}{3}+c)^{\underline{2i+1}}}{ i! \prod_{j=1}^{i+1}(3j-1)} \sigma_2^{3i+2}\sigma_3^{\frac{p-2}{3}+c-2i-1} \cdot (b_+\otimes b_-- b_-\otimes b_+)+\\
			&\quad\quad  +\sum_{i=0}^{\lfloor\frac{1}{2}(\frac{p-2}{3}+c)\rfloor}\frac{(-1)^{i}}{9^i} \frac{(\frac{p-2}{3}+c)^{\underline{2i}}}{ i! \prod_{j=1}^{i}(3j-1) } \sigma_2^{3i}\sigma_3^{\frac{p-2}{3}+c-2i} \cdot \\
		&\quad\quad\quad\quad\quad\quad\quad\quad\cdot \bigg( (-b_+^2+3b_-^2)\otimes b_- -  (2b_+b_-)\otimes b_+\bigg).\end{align*}
			
\end{enumerate}

\end{lemma}
\begin{proof}
The proof, very similar to the proofs of Lemmas \ref{p>3-Stand-specialc-lemma-3c} and \ref{p>3-Stand-specialc-lemma-p+3c}, is a direct computation using formulas \eqref{p>3,genc,Stand,D1(Sign)-p1} and \eqref{p>3,genc,Stand,D1(Sign)-p2}, showing that $D_{y_1}(v_{p+3c})=0.$
\end{proof}

After this, as outlined in the proof of Theorem \ref{p>3,specialc,Stand,Thm1}, in the next lemmas we calculate the Hilbert polynomial of the quotient of $M_{1,c}(\Stand)$ by $\left<v_{p-3c},v_+,v_-,v_{p+3c}\right>$. 

\begin{lemma}\label{p>3-Stand-specialc-char1}
The Hilbert polynomial of $M_{1,c}(\Stand)/\left<v_{p-3c,v_{p+3c}}\right>$ equals
$$h_{M_{1,c}(\Stand)/\left<v_{p-3c}, v_{p+3c}\right>}(z)=\frac{1}{(1-z)^2}(2-z^{p-3c}-z^{p+3c}).$$
\end{lemma}
\begin{proof}
It is straightforward to see that $\left<v_{p-3c}\right>\cong M_{1,c}(\Triv)[-(p-3c)]$ and $\left<v_{p+3c}\right> \cong M_{1,c}(\Sign)[-(p+3c)]$, so we only need to calculate Hilbert series of $\left<v_{p-3c}\right>\cap \left<v_{p+3c}\right>$. This is done similar to the proof of Lemma \ref{complete-intersection}, or Lemma \ref{p>3,Det1} and \ref{p>3,t=1,generic,Stand-character}. 

Write 
\begin{align*}
v_{p-3c}&=f\otimes b_++g\otimes b_-\\
v_{p+3c}&=h\otimes b_++k\otimes b_-
\end{align*}
with $f,g,h,k\in S\h^*$.
A vector is in $\left<v_{p-3c}\right>\cap \left<v_{p-3c}\right>$ if and only if it can be written as 
$Av_{p-3c}=Bv_{p+3c}$
for some $A,B\in S\h^*$. This leads to the system 
\begin{align}
Af-Bh&=0 \label{somesystem}\\
Ag-Bk&=0,\notag
\end{align}
and to see if this has any nontrivial solutions, we will calculate its determinant $fk-gh.$

First, the facts that $v_{p-3c}$ is $S_3$ invariant and that $v_{p+3c}$ is $S_3$ anti-invariant lead to
\begin{align*}
(12).f&=f, & (12).g&=-g, & (12).h&=-h, & (12).k&=k \\
(23).f&=\frac{-f+g}{2}, & (23).g&=\frac{3f+g}{2}, & (23).h&=\frac{h-k}{2}, & (23).k&=\frac{-3h-k}{2} \\
(13).f&=\frac{-f-g}{2}, & (23).g&=\frac{-3f+g}{2}, & (23).h&=\frac{h+k}{2}, & (23).k&=\frac{3h-k}{2}.\end{align*}

Next, the fact that $v_{p-3c}$ and $v_{p+3c}$ are both singular leads to 
\begin{align*}
\partial_{y_1}(f)&=-c\frac{1}{x_1-x_3}\frac{3f+g}{2}\\
\partial_{y_1}(g)&=-2c\frac{g}{x_1-x_2}-c\frac{1}{x_1-x_3}\frac{3f+g}{2}\\
\partial_{y_1}(h)&=c\frac{2h}{x_1-x_2}+c\frac{1}{x_1-x_3}\frac{h-k}{2}\\
\partial_{y_1}(k)&=-3c\frac{1}{x_1-x_3}\frac{h-k}{2}.
\end{align*}

From here it is straightforward to show that $fk-gh$ is an $S_3$ invariant with the property $D_{y_1}(fk-gh)=0$, so $\partial_{y}(fk-gh)=0$ for all $y$ and thus $fk-gh$ is a $p$-th power and an invariant. Given that its degree is $2p$, we conclude that $fk-gh$ is a scalar multiple of $\sigma_2^p$.

To calculate this scalar, we calculate $fk-gh$ explicitly, using formulas in Lemmas \ref{p>3-Stand-specialc-lemma-p-3c} and \ref{p>3-Stand-specialc-lemma-p+3c}, while disregarding all terms with a nonzero power of $\sigma_3$ (as we know from the above that they don't contribute to the final expression as $fk-gh$ is a scalar multiple of $\sigma_2^p$). We distinguish four cases, and claim that in all of them $fk-gh\ne 0$. 
\begin{enumerate}
\item If $p\equiv 1 \bmod 3$ and $m=\frac{p-1}{3}-c$ is even, then $\frac{p-1}{3}+c=m+2c$ is even and we get
\begin{align*}
v_{p-3c}&=\frac{(-1)^{m/2}}{9^{m/2}}\frac{m!}{ (m/2)! \cdot  \prod_{j=1}^{m/2}(3j-2) } \sigma_2^{3m/2}\cdot (b_+\otimes b_++ 3 b_-\otimes b_-)+ \sigma_3\ldots \\
v_{p+3c}&=\frac{(-1)^{m/2+c}}{9^{m/2+c}}\frac{(m+2c)!}{ (m/2+c)! \cdot  \prod_{j=1}^{m/2+c}(3j-2) } \sigma_2^{3(m/2+c)}\cdot (b_+\otimes b_-- b_-\otimes b_+)+ \sigma_3\ldots \\
 fk-gh&= \frac{(-1)^{m/2}}{9^{m/2}}\frac{m!}{ (m/2)!   \prod_{j=1}^{m/2}(3j-2) } \frac{(-1)^{m/2+c}}{9^{m/2+c}}\frac{(m+2c)!}{ (m/2+c)!   \prod_{j=1}^{m/2+c}(3j-2) }\cdot \\
&   \quad \cdot \sigma_2^{3m/2}\sigma_2^{3(m/2+c)} (b_+^2+3b_-^2)\\
 &= \frac{(-1)^{\frac{p-1}{3}}(-12)({\frac{p-1}{3}}-c)! ({\frac{p-1}{3}}+c)! }{ 9^{\frac{p-1}{3}}(\frac{1}{2}({\frac{p-1}{3}}-c))!  (\frac{1}{2}({\frac{p-1}{3}}+c))!   \prod_{j=1}^{\frac{1}{2}({\frac{p-1}{3}}-c)}(3j-2)  \prod_{j=1}^{\frac{1}{2}({\frac{p-1}{3}}+c)}(3j-2)} \sigma_2^{p} \ne 0.
\end{align*}
\item If $p\equiv 1 \bmod 3$ and $m=\frac{p-1}{3}-c$ is odd - similar. 
\item If $p\equiv 2 \bmod 3$ and $m=\frac{p-2}{3}-c$ is even - similar. 
\item If $p\equiv 2 \bmod 3$ and $m=\frac{p-2}{3}-c$ is odd - similar. 
\end{enumerate}

This proves that $fk-gh\ne 0$, which implies that the only solution to the system \eqref{somesystem} is $A=B=0$, so $\left<v_{p-3c}\right>\cap \left<v_{p-3c}\right>=0$. Consequently, 
\begin{align*}
h_{M_{1,c}(\Stand)/\left<v_{p-3c}, v_{p+3c}\right>}(z)&=h_{M_{1,c}(\Stand)}(z)-h_{\left<v_{p-3c}\right>}(z)-h_{\left<v_{p+3c}\right>}(z)+h_{\left<v_{p-3c}\right>\cap \left<v_{p-3c}\right>}(z)\\
& =\frac{2}{(1-z)^2}-z^{p-3c}\cdot \frac{1}{(1-z)^2}-z^{p+3c}\cdot \frac{1}{(1-z)^2}+0.
\end{align*}
\end{proof}

\begin{lemma}\label{p>3-Stand-specialc-char2}
The images of the vectors $v_+,v_-$ from Lemma \ref{p>3,t=1,generic,Stand} in the quotient module $M_{1,c}(\Stand)/\left<v_{p-3c}, v_{p+3c}\right>$ are nonzero.
\end{lemma}
\begin{proof}
Let us show that $v_+$ is not contained in $\left<v_{p-3c}, v_{p+3c}\right>$.
As $v_-$ can be calculated from $v_+$ by the action of $\Bbbk(S_3)\subseteq H_{1,c}(S_3,\h)$ and $\left<v_{p-3c}, v_{p+3c}\right>$ is an $H_{1,c}(S_3,\h)$ submodule, so the claim for $v_-$ will follow. Further, noting that $\deg v_+=p>p-3c=\deg v_{p-3c}$ and $\deg v_+=p<p+3c=\deg v_{p+3c}$, the claim is equivalent to showing that $v_+\notin \left< v_{p-3c} \right>$. 

Assume the contrary, that $v_+\in \left< v_{p-3c} \right>$. In that case there exists $A\in S^{3c}\h^*$ such that $v_+=Av_{p-3c}$. Considering how these vectors transform under the action of $S_3$ ($v_{p-3c}$ spans $\Triv$ and $v_+$ spans the $(12)$-invariant part of $\Stand$), we can see that $A$ needs to transform like the $S_2$-invariant part of $\Stand$ and thus needs to be of the form $$A=gb_++h(-b_{+}^2+3b_-^2)$$ for some $g\in (S^{3c-1}\h^*)^{S_3}$, $h\in (S^{3c-2}\h^*)^{S_3}$. We will consider the equation $v_+=Av_{p-3c}$ modulo $\sigma_2$ to show it has no solutions. For that purpose, we distinguish two cases. 

\begin{enumerate}
\item If $p\equiv 1 \bmod 3$, then 
\begin{align*}
v_+&=\sigma_3^{\frac{p-1}{3}}(-b_+\otimes b_++3b_-\otimes b_-)  + \sigma_2\ldots\\
v_{p-3c}&=\sigma_3^{\frac{p-1}{3}-c}(b_+\otimes b_++3b_-\otimes b_-)  + \sigma_2\ldots.
\end{align*}
Modulo $\sigma_2$, the equation $Av_{p-3c}=v_+$ becomes
\begin{align*}
\sigma_3^{\frac{p-1}{3}}(-b_+\otimes b_++3b_-\otimes b_-)&=(gb_++h(-b_{+}^2+3b_-^2))\cdot \\
& \quad \cdot \sigma_3^{\frac{p-1}{3}-c}(b_+\otimes b_++3b_-\otimes b_-) + \sigma_2\ldots \\
\sigma_3^{c}(-b_+\otimes b_++3b_-\otimes b_-)&=(gb_+^2+h(-b_{+}^2+3b_-^2)\cdot b_+)\otimes b_++\\
&+(3gb_+b_-+ 3h(-b_{+}^2+3b_-^2)\cdot b_-)\otimes b_- + \sigma_2\ldots \\
&=(\frac{-1}{2}g(-b_{+}^2+3b_-^2)+h\cdot 54\sigma_3)\otimes b_++\\
&+(\frac{1}{2}g(2b_+b_-)+ 3h\cdot (-6q))\otimes b_- + \sigma_2\ldots .
\end{align*}
Reading off the coefficient of $(-b_+\otimes b_++3b_-\otimes b_-)$ we get
\begin{align*}
\sigma_3^{c}&=0 
\end{align*}
which is a contradiction.

\item If $p\equiv 2 \bmod 3$, then 
\begin{align*}
v_+&=\sigma_3^{\frac{p-2}{3}}(-(-b_+^2+3b_-^2)\otimes b_++3(2b_+b_-)\otimes b_-)  + \sigma_2\ldots\\
v_{p-3c}&=\sigma_3^{\frac{p-2}{3}-c}((-b_+^2+3b_-^2)\otimes b_++3(2b_+b_-)\otimes b_-)  + \sigma_2\ldots
\end{align*}

The equation $Av_{p-3c}=v_+$, similar to above, becomes 
\begin{align*}
%\sigma_3^{\frac{p-2}{3}}(-(-b_+^2+3b_-^2)\otimes b_++3(2b_+b_-)\otimes b_-)&=(gb_++h(-b_{+}^2+3b_-^2)) \cdot \sigma_3^{\frac{p-2}{3}-c}((-b_+^2+3b_-^2)\otimes b_++3(2b_+b_-)\otimes b_-)\\
%\sigma_3^{c}(-(-b_+^2+3b_-^2)\otimes b_++3(2b_+b_-)\otimes b_-)&=(gb_+(-b_+^2+3b_-^2)+h(-b_{+}^2+3b_-^2)^2)\otimes b_+ +\\
%&+3(gb_+(2b_+b_-)+h(-b_{+}^2+3b_-^2)(2b_+b_-))\otimes b_-\\
\sigma_3^{c}(-(-b_+^2+3b_-^2)\otimes b_++3(2b_+b_-)\otimes b_-)&=(g\cdot 54\sigma_3+h\cdot (-108)\sigma_3b_+ )\otimes b_+ +\\
&+3(g\cdot 6q+h\cdot 108\sigma_3b_-)\otimes b_-  + \sigma_2\ldots
\end{align*}
\end{enumerate}
leading to
$\sigma_3^{c}=0$ which is a contradiction. 
\end{proof}

\begin{lemma}\label{p>3-Stand-specialc-char6}
Let $$\varphi:M_{1,c}(\Stand)[-p] \to M_{1,c}(\Stand)/\left<v_{p-3c},v_{p+3c} \right>$$ be the map of $H_{1,c}(S_3,\h)$-modules given by $\varphi(f\otimes b_{\pm})=f\cdot v_{\pm}$. Then $\varphi(v_{p-3c})\ne 0$.
\end{lemma}
\begin{proof}
Let $\psi:M_{1,c}(\Stand)[-p] \to M_{1,c}(\Stand)$ be the map of $H_{1,c}(S_3,\h)$ modules given by the same formula. Then $\varphi(v_{p-3c})= 0$ if and only if there exist $A,B\in S \h^*$ such that $$\psi(v_{p-3c})=A\cdot v_{p-3c}+B\cdot v_{p+3c}.$$ Considering how they transform under $S_3$, we get that $A\in (S\h^*)^{S_3}$ and $B=B'q$ for $B'\in (S\h^*)^{S_3}$. We then calculate modulo $\sigma_2$, using the explicit expressions for $v_\pm,v_{p-3c},v_{p+3c}$ from Lemmas \ref{p>3-Stand-specialc-lemma-p-3c}, \ref{p>3,t=1,generic,Stand} and \ref{p>3-Stand-specialc-lemma-p+3c}, and distinguishing two cases. 

\begin{enumerate}
\item If $p\equiv 1 \bmod 3$, then 
\begin{align*}
v_+&=\sigma_3^{\frac{p-1}{3}}(-b_+\otimes b_++3b_-\otimes b_-) \bmod \sigma_2\\
v_-&=\sigma_3^{\frac{p-1}{3}}(b_+\otimes b_-+b_-\otimes b_+) \bmod \sigma_2\\
v_{p-3c}&=\sigma_3^{\frac{p-1}{3}-c}(b_+\otimes b_++3b_-\otimes b_-) \bmod \sigma_2\\
v_{p+3c}&=\sigma_3^{\frac{p-1}{3}+c}(b_+\otimes b_--3b_-\otimes b_+) \bmod \sigma_2.
\end{align*} 

We calculate modulo $\sigma_2$:
\begin{align*}
\psi(v_{p-3c})&=\sigma_3^{\frac{p-1}{3}-c}(b_+\cdot v_++3b_- \cdot v_-) \bmod \sigma_2\\
&=\sigma_3^{2\frac{p-1}{3}-c}((-b_+^2+3b_-^2)\otimes b_++3(2b_+b_-)\otimes b_- ) \bmod \sigma_2.
\end{align*}
At the same time, for any $A,B'\in (S\h^*)^{S_3}$ we have 
\begin{align*}
A\cdot v_{p-3c}+B'\cdot qv_{p+3c}&=A\sigma_3^{\frac{p-1}{3}-c}(b_+\otimes b_++3b_- \otimes b_-)+\\
& \quad +B' \sigma_3^{\frac{p-1}{3}+c}q(b_+\otimes b_--b_- \otimes b_+) \bmod \sigma_2\\
&=\left( A\sigma_3^{\frac{p-1}{3}-c}-3B' \sigma_3^{\frac{p-1}{3}+c+1}\right)
(b_+\otimes b_++3b_- \otimes b_-) \bmod \sigma_2.
\end{align*}
Comparing these expressions we see that there for all $A,B$ 
$$\psi(v_{p-3c})\ne A\cdot v_{p-3c}+B\cdot v_{p+3c},$$
so $\varphi(v_{p-3c})\ne 0$ as claimed.

\item If $p\equiv 2 \bmod 3$ - similar.

\begin{comment}, then 
\begin{align*}
v_+&=\sigma_3^{\frac{p-2}{3}}(-(-b_+^2+3b_-^2)\otimes b_++3(2b_+b_-)\otimes b_-) \bmod \sigma_2\\
v_-&=\sigma_3^{\frac{p-2}{3}}((-b_+^2+3b_-^2)\otimes b_-+(2b_+b_-)\otimes b_+) \bmod \sigma_2\\
v_{p-3c}&=\sigma_3^{\frac{p-2}{3}-c}((-b_+^2+3b_-^2)\otimes b_++3(2b_+b_-)\otimes b_-) \bmod \sigma_2.
\end{align*} 
Similarly, 
\begin{align*}
\varphi(v_{p-3c})&=\sigma_3^{\frac{p-2}{3}-c}((-b_+^2+3b_-^2)\cdot v_++3(2b_+b_-)\cdot v_-)\bmod \sigma_2\\
%&=\sigma_3^{2\frac{p-2}{3}-c}\left( (-(-b_+^2+3b_-^2)^2+3(2b_+b_-)^2)\otimes b_++6(2b_+b_-)(-b_+^2+3b_-^2)\otimes b_-\right)\bmod \sigma_2\\
%&=\sigma_3^{2\frac{p-2}{3}-c}\left( 2\cdot 108 \sigma_3b_+\otimes b_++6\cdot 108 \sigma_3b_- \otimes b_-\right)\bmod \sigma_2\\
&=216\sigma_3^{2\frac{p-2}{3}-c+1}\left( b_+\otimes b_++3b_- \otimes b_-\right)\ne 0 \bmod \sigma_2.
\end{align*}
\end{comment}
\end{enumerate}
\end{proof}

Finally, the end of the proof of Theorem \ref{p>3,specialc,Stand,Thm1} requires some facts about Hilbert polynomials, which we prove in the last two Lemmas of this subsection. 

\begin{lemma}\label{p>3-Stand-specialc-char7}
The coefficient of $z^{2p-1}$ in the expansion of
$$\frac{1}{(1-z)^2}\left( 2-z^{p-3c}- 2 \cdot z^p -  z^{p+3c}\right)$$
around $z=0$ is $0$.
\end{lemma}
\begin{proof}
The coefficient of $z^k$ in $\frac{1}{(1-z)^2}$ is $k+1$, so the coefficient of $z^{2p-1}$ in the above expression is
\begin{align*}
2(2p-1+1)-(2p-1-(p-3c)+1)-2(2p-1-p+1)-(2p-1-(p+3c)+1)=0.
\end{align*}
\end{proof}

At this point in the proof of Theorem \ref{p>3,specialc,Stand,Thm1} we conclude that the Hilbert polynomial of the irreducible module $L_{1,c}(\Stand)$ has the form \eqref{p>3-Stand-specialc-char-form}. The following Lemma shows that the only module with such  Hilbert polynomial is $M_{1,c}(\Stand)/\left<v_{p-3c},v_+,v_-, v_{p+3c}\right>$ itself.

\begin{lemma}\label{p>3-Stand-specialc-char4}
Let $Q$ be a quotient of $M_{1,c}(\Stand)/\left<v_{p-3c},v_+,v_-, v_{p+3c}\right>$ with a Hilbert series of the form
\begin{align*}h_{Q}(z)&=h_{M_{1,c}(\Stand)/\left<v_{p-3c},v_+,v_-, v_{p+3c}\right>}(z)-M\cdot z^p \cdot h_{L_{1,c}(\Stand)}(z)-\\
& \quad -K\cdot z^{3c+p}\cdot h_{L_{1,c}(\Sign)}(z)-N\cdot z^{2p-3c} h_{L_{1,c}(\Triv)}(z)+O(z^{2p})
\end{align*}
for some $M,N,K \in \mathbb{N}_0$. Then: 
\begin{enumerate}
\item $M=0$;
\item $K=N=0$;
\item $Q=M_{1,c}(\Stand)/\left<v_{p-3c},v_+,v_-, v_{p+3c}\right>$.  
\end{enumerate}
\end{lemma}

\begin{proof}
\begin{enumerate}
\item We distinguish two cases. 

{\bf Case 1. If $0<c<p/6$,} calculate the dimension of $Q^{2p-3c}$. We have $p+3c<2p-3c$, so 
\begin{align*}
\dim{Q^{2p-3c}}&=\dim{M_{1,c}^{2p-3c}(\Stand)/\left<v_{p-3c},v_+,v_-, v_{p+3c}\right>}-\\
& -M \dim{L_{1,c}^{p-3c}(\Stand)}-K \dim{L_{1,c}^{p-6c}(\Sign)}-N \dim{L_{1,c}^{0}(\Triv)}.
\end{align*}
By equation \eqref{p>3,cscpecial,Stand,somecharfmula} we know 
\begin{align*}
\dim&(M_{1,c}^{2p-3c}(\Stand)/\left<v_{p-3c},v_+,v_-, v_{p+3c}\right>)=\\
& \quad =2(2p-3c+1)-(2p-3c-(p-3c)+1)-\\
& \quad \quad \quad-2 (2p-3c-p+1)-(2p-3c-(p-3c)+1)\\
&\quad =6c-2.
\end{align*}
Equation \eqref{p>3-Stand-specialc-char-form} lets us conclude that 
\begin{align*}\dim{L_{1,c}^{p-3c}(\Stand)}&=\dim M_{1,c}^{p-3c}(\Stand)/\left<v_{p-3c},v_+,v_-, v_{p+3c}\right>\\
&=2(p-3c+1)-1=2p-6c+1.
\end{align*}
By Theorem \ref{thm-p>3special} $h_{L_{1,c}(\Sign)}(z)=\left( \frac{1-z^{p-3c}}{1-z}\right)^2$ so
$$\dim{L_{1,c}^{p-6c}(\Sign)}=p-6c+1.$$
We also know that $\dim{L_{1,c}^{0}(\Triv)}=1$. 
Putting it all together we get: 
\begin{align*}
0\le \dim{Q^{2p-3c}}&=(6c-2)-M\cdot (2p-6c+1)-K\cdot (p-6c+1)-N\cdot 1.
\end{align*}
Rearranging and using $c<p/6$ we get
\begin{align*}
p-2 & > 6c-2  \\
&\ge M\cdot (2p-6c+1)+K\cdot (p-6c+1)+N\\
& \ge M\cdot (2p-6c+1)\\
&> M\cdot (p+1).
\end{align*}
The only nonnegative integer $M$ satisfying this is $M=0$. 

{\bf Case 2. $p/6<c<p/3$.} We have $2p-3c<p+3c$ and similarly calculate the dimension of $Q^{p+3c}$, getting: 
\begin{align*}
0&\le \dim{Q^{p+3c}}= \dim {M_{1,c}^{p+3c}(\Stand)/\left<v_{p-3c},v_+,v_-, v_{p+3c}\right>}-\\
&\quad -M\cdot \dim{L_{1,c}^{3c}(\Stand)} -K\cdot \dim{L_{1,c}^0(\Sign)}-N\cdot \dim{L_{1,c}^{6c-p}(\Triv)}\\
&=\left(2\cdot (p+3c+1)-(p+3c-p+3c+1)-2(p+3c-p+1)-1\right)-\\
&\quad -M\cdot \left(2\cdot(3c+1)-(3c-p+3c+1)\right) -K\cdot 1 -N\cdot (6c-p+1) \\
&=\left(2p-6c-2\right)- M\cdot \left(p+1)\right) -K -N\cdot (6c-p+1) \\
& \le \left(2p-6c-2\right)- M\cdot \left(p+1\right).
\end{align*}
We rearrange this as 
$$M\cdot \left(p+1\right)\le 2p-6c-2< p+6c-6c-2=p-2,$$
and notice that the only nonnegative integer $M$ satisfying this is $M=0$.

\item Using part (1), we calculate the dimension of $Q^{2p-1}$ as
\begin{align*}
\dim{Q}^{2p-1}&=\dim{M_{1,c}^{2p-1}(\Stand)/\left<v_{p-3c},v_+,v_-, v_{p+3c}\right>}-\\
& \quad -K \cdot \dim{L_{1,c}^{p-3c-1}(\Sign)}-N\cdot  \dim{L_{1,c}^{3c-1}(\Triv)}\\
&=\left(2\cdot 2p - (2p-1-p+3c+1)-2 (2p-1-p+1)-(2p-1-p-3c+1) \right)-\\
& \quad -K \cdot (p-3c-1+1)-N\cdot (3c-1+1)\\
&=-K \cdot (p-3c)-N\cdot 3c.
\end{align*}
Using that $\dim{Q}^{2p-1}\ge 0$, $p-3c>0$ and $3c>0$, the only nonnegative $K,N$ satisfying this are $K=N=0.$
\item This follows immediately from (1) and (2). 
\end{enumerate}
\end{proof}

\subsection{Irreducible representation $L_{1,c}(\Stand)$, part 2}\label{section-p>3,Stand,specialc-2}

\begin{conjecture}\label{finalconj}
The irreducible representation $L_{1,c}(\Stand)$ of the rational Cherednik algebra $H_{1,c}(S_3,\h)$ over an algebraically closed field of characteristic $p>3$ for $c\in \mathbb{F}_p$, $p/3<c<p/2$, is the quotient of the Verma module $M_{1,c}(\Stand)$ by the submodule generated by vectors in the following degrees: 
\begin{itemize}
\item $-p+3c$ - one dimensional space, of type $\Sign$
\item $3p-3c$ - one dimensional space, of type $\Triv$
\item $p+3c$  - one dimensional space, of type $\Sign$
\item $5p-3c$ - one dimensional space, of type $\Triv$
\end{itemize}

Its character and Hilbert polynomial are
\begin{align*}
\chi_{L_{1,c}(\Stand)}(z)&=\chi_{S\h^*}(z)\left( [\Stand]-z^{-p+3c}[\Sign]-z^{3p-3c}[\Triv]-\right. \\
& \quad  \quad \left. -z^{p+3c}[\Sign]-z^{5p-3c}[\Triv]+z^{4p}[\Stand] \right)\\
h_{L_{1,c}(\Stand)}(z)&=\frac{2-z^{-p+3c}-z^{3p-3c}-z^{p+3c}-z^{5p-3c}+2z^{4p}}{(1-z)^2}.
\end{align*}

\end{conjecture}

Compare this with Lemma \ref{p=2,t=1,Stand,c=1}, where an analogous result is proved for $p=2$. Notice also that the vectors $v_+,v_-$ from Lemma \ref{p>3,t=1,generic,Stand}, which are singular in the generic case, are in this case contained in the submodule generated in degrees $-p+3c$ and $3p-3c$ so do not have to be included in the list of generators of the maximal proper submodule here. 

We have checked this conjecture using Magma \cite{Magma} for all $p\le 13$ and all $c$, and a small number of larger $p$ and specific $c$. For example, when $p=13$ and $c=6$, $L_{1,c}(\Stand)$ has  $50$ nonzero graded pieces and dimension $886$.

\bibliographystyle{alpha}
\bibliography{references}

\end{document}